\documentclass[a4paper, reqno,10pt]{amsart}

\usepackage{amsthm,bm,amsfonts,amssymb,amsmath,amsxtra}
\usepackage[all]{xy}
\usepackage[utf8]{inputenc}

\SelectTips{cm}{}
\usepackage{tikz}
\usepackage{verbatim}

\usepackage[margin=1.25in]{geometry}
\usepackage{mathrsfs}

\RequirePackage{xspace}
\RequirePackage{etoolbox}
\RequirePackage{varwidth}
\RequirePackage{enumitem}
\RequirePackage{tensor}
\RequirePackage{mathtools}
\RequirePackage{longtable}
\RequirePackage{multirow}
\RequirePackage{makecell}
\RequirePackage{bigints}
\RequirePackage{xcolor}

\usepackage[backref=page]{hyperref} %

\newcommand{\fka}{\ensuremath{\mathfrak{a}}\xspace}

\newcommand{\fkd}{\ensuremath{\mathfrak{d}}\xspace}

\newcommand{\fkg}{\ensuremath{\mathfrak{g}}\xspace}

\newcommand{\fkq}{\ensuremath{\mathfrak{q}}\xspace}

\newcommand{\BA}{\ensuremath{\mathbb{A}}\xspace}

\newcommand{\BC}{\ensuremath{\mathbb{C}}\xspace}

\newcommand{\BE}{\ensuremath{\mathbb{E}}\xspace}

\newcommand{\BG}{\ensuremath{\mathbb{G}}\xspace}

\newcommand{\BL}{\ensuremath{\mathbb{L}}\xspace}

\newcommand{\BP}{\ensuremath{\mathbb{P}}\xspace}
\newcommand{\BQ}{\ensuremath{\mathbb{Q}}\xspace}
\newcommand{\BR}{\ensuremath{\mathbb{R}}\xspace}

\newcommand{\BV}{\ensuremath{\mathbb{V}}\xspace}

\newcommand{\BX}{\ensuremath{\mathbb{X}}\xspace}

\newcommand{\BZ}{\ensuremath{\mathbb{Z}}\xspace}

\newcommand{\bB}{\ensuremath{\mathbf{B}}\xspace}

\newcommand{\bE}{\ensuremath{\mathbf{E}}\xspace}
\newcommand{\bF}{\ensuremath{\mathbf{F}}\xspace}

\newcommand{\bH}{\ensuremath{\mathbf{H}}\xspace}

\newcommand{\bK}{\ensuremath{\mathbf{K}}\xspace}

\newcommand{\bR}{\ensuremath{\mathbf{R}}\xspace}

\newcommand{\CA}{\ensuremath{\mathcal{A}}\xspace}

\newcommand{\CD}{\ensuremath{\mathcal{D}}\xspace}
\newcommand{\CE}{\ensuremath{\mathcal{E}}\xspace}
\newcommand{\CF}{\ensuremath{\mathcal{F}}\xspace}
\newcommand{\CG}{\ensuremath{\mathcal{G}}\xspace}

\newcommand{\CM}{\ensuremath{\mathcal{M}}\xspace}
\newcommand{\CN}{\ensuremath{\mathcal{N}}\xspace}
\newcommand{\CO}{\ensuremath{\mathcal{O}}\xspace}

\newcommand{\CS}{\ensuremath{\mathcal{S}}\xspace}

\newcommand{\CX}{\ensuremath{\mathcal{X}}\xspace}

\newcommand{\CZ}{\ensuremath{\mathcal{Z}}\xspace}

\newcommand{\RT}{\ensuremath{\mathrm{T}}\xspace}
\newcommand{\U}{\ensuremath{\mathrm{U}}\xspace}
\newcommand{\RV}{\ensuremath{\mathrm{V}}\xspace}

\newcommand{\adm}{\mathrm{adm}}

\DeclareMathOperator{\Aut}{Aut}

\newcommand{\Ch}{{\mathrm{Ch}}}
\DeclareMathOperator{\charac}{char}

\newcommand{\del}{\operatorname{\partial Orb}}
\newcommand{\delJ}{\partial J}
\DeclareMathOperator{\diag}{diag}
\DeclareMathOperator{\Diff}{Diff}

\renewcommand{\div}{{\mathrm{div}}}

\newcommand{\Ei}{\mathrm{Ei}}

\DeclareMathOperator{\End}{End}

\DeclareMathOperator{\Gal}{Gal}

\newcommand{\GL}{\mathrm{GL}}

\newcommand{\GU}{\mathrm{GU}}

\DeclareMathOperator{\Hom}{Hom}

\newcommand{\id}{\ensuremath{\mathrm{id}}\xspace}
\let\Im\relax

\DeclareMathOperator{\Im}{Im}

\newcommand{\Ind}{{\mathrm{Ind}}}
\DeclareMathOperator{\Int}{\ensuremath{\mathrm{Int}}\xspace}

\DeclareMathOperator{\Lie}{Lie}

\DeclareMathOperator{\Nm}{Nm}

\DeclareMathOperator{\Orb}{Orb}

\DeclareMathOperator{\Ros}{Ros}

\renewcommand{\Re}{{\mathrm{Re}}}

\DeclareMathOperator{\Res}{Res}
\newcommand{\rs}{\ensuremath{\mathrm{rs}}\xspace}

\newcommand{\Sh}{\mathrm{Sh}}

\newcommand{\SL}{{\mathrm{SL}}}

\DeclareMathOperator{\Spec}{Spec}
\DeclareMathOperator{\Spf}{Spf}

\newcommand{\SO}{{\mathrm{SO}}}

\newcommand{\ssm}{\smallsetminus}
\newcommand{\SU}{{\mathrm{SU}}}
\DeclareMathOperator{\supp}{supp}

\DeclareMathOperator{\tr}{tr}

\DeclareMathOperator{\vol}{vol}

\newcommand{\wt}{\widetilde}
\newcommand{\wh}{\widehat}

\newcommand{\pair}[1]{\langle {#1} \rangle}

\newcommand{\ov}{\overline}
\newcommand{\incl}{\hookrightarrow}
\newcommand{\lra}{\longrightarrow}

\newcommand{\bs}{\backslash}
\newcommand{\la}{\langle}
\newcommand{\ra}{\rangle}
\newcommand{\lv}{\lvert}
\newcommand{\rv}{\rvert}

\newcommand{\LN}{\,^\BL\!\CN}
\newcommand{\CCM}{\ensuremath{\mathcal{C\!M}}\xspace}
\newcommand{\LCM}{\,^\BL\CCM}
\newcommand{\Ltimes}{\stackrel{\BL}{\otimes}}
\newcommand{\jiao}{\stackrel{\BL}{\cap}}

\newtheorem{theorem}{Theorem}
\newtheorem{proposition}[theorem]{Proposition}
\newtheorem{lemma}[theorem]{Lemma}

\newtheorem{conjecture}[theorem]{Conjecture}
\newtheorem{`conjecture'}[theorem]{``Conjecture''}
\newtheorem{corollary}[theorem]{Corollary}

\theoremstyle{definition}
\newtheorem{definition}[theorem]{Definition}

\newtheorem{remark}[theorem]{Remark}

\newenvironment{altenumerate}
   {\begin{list}
      {(\theenumi) }
      {\usecounter{enumi}
       \setlength{\labelwidth}{0pt}
       \setlength{\labelsep}{0pt}
       \setlength{\leftmargin}{0pt}
       \setlength{\itemsep}{\the\smallskipamount}
       \renewcommand{\theenumi}{\roman{enumi}}
      }}
   {\end{list}}

\newenvironment{altitemize}
   {\begin{list}
      {$\bullet$}
      {\setlength{\labelwidth}{0pt}
	   \setlength{\itemindent}{5pt}
       \setlength{\labelsep}{5pt}
       \setlength{\leftmargin}{0pt}
       \setlength{\itemsep}{\the\smallskipamount}
      }}
   {\end{list}}

\numberwithin{equation}{section}
\numberwithin{theorem}{section}

\setitemize[0]{leftmargin=*,itemsep=\the\smallskipamount}
\setenumerate[0]{leftmargin=*,itemsep=\the\smallskipamount}

\renewcommand{\to}{%
   \ifbool{@display}{\longrightarrow}{\rightarrow}%
   }
\let\shortmapsto\mapsto
\renewcommand{\mapsto}{%
   \ifbool{@display}{\longmapsto}{\shortmapsto}%
   }
\newcommand{\hooklongrightarrow}{\mathrel{\mkern 0.5mu\lhook\mkern -3.5mu\relbar\mkern -3mu \rightarrow }}
\newcommand{\inj}{%
   \ifbool{@display}{\hooklongrightarrow}{\hookrightarrow}
   }
\newcommand{\isoarrow}{%
   \ifbool{@display}{\overset{\sim}{\longrightarrow}}{\xrightarrow\sim}%
   }
\newlength{\olen}
\newlength{\ulen}
\newlength{\xlen}
\newcommand{\xra}[2][]{%
   \ifbool{@display}%
      {\settowidth{\olen}{$\overset{#2}{\longrightarrow}$}%
       \settowidth{\ulen}{$\underset{#1}{\longrightarrow}$}%
       \settowidth{\xlen}{$\xrightarrow[#1]{#2}$}%
       \ifdimgreater{\olen}{\xlen}%
          {\underset{#1}{\overset{#2}{\longrightarrow}}}%
          {\ifdimgreater{\ulen}{\xlen}%
             {\underset{#1}{\overset{#2}{\longrightarrow}}}
             {\xrightarrow[#1]{#2}}}}%
      {\xrightarrow[#1]{#2}}
   }
\makeatother
\newcommand{\xyra}[2][]{%
   \settowidth{\xlen}{$\xrightarrow[#1]{#2}$}%
   \ifbool{@display}%
      {\settowidth{\olen}{$\overset{#2}{\longrightarrow}$}%
       \settowidth{\ulen}{$\underset{#1}{\longrightarrow}$}%
       \ifdimgreater{\olen}{\xlen}%
          {\mathrel{\xymatrix@M=.12ex@C=3.2ex{\ar[r]^-{#2}_-{#1} &}}}%
          {\ifdimgreater{\ulen}{\xlen}%
             {\mathrel{\xymatrix@M=.12ex@C=3.2ex{\ar[r]^-{#2}_-{#1} &}}}
             {\mathrel{\xymatrix@M=.12ex@C=\the\xlen{\ar[r]^-{#2}_-{#1} &}}}}}%
      {\mathrel{\xymatrix@M=.12ex@C=\the\xlen{\ar[r]^-{#2}_-{#1} &}}}%
   }
\makeatletter
\newcommand{\xla}[2][]{%
   \ifbool{@display}%
      {\settowidth{\olen}{$\overset{#2}{\longleftarrow}$}%
       \settowidth{\ulen}{$\underset{#1}{\longleftarrow}$}%
       \settowidth{\xlen}{$\xleftarrow[#1]{#2}$}%
       \ifdimgreater{\olen}{\xlen}%
          {\underset{#1}{\overset{#2}{\longleftarrow}}}%
          {\ifdimgreater{\ulen}{\xlen}%
             {\underset{#1}{\overset{#2}{\longleftarrow}}}
             {\xleftarrow[#1]{#2}}}}%
      {\xleftarrow[#1]{#2}}
   }
\renewcommand{\lra}{%
   \ifbool{@display}{\longleftrightarrow}{\leftrightarrow}%
   }

\usepackage{unicode-math}

\newcommand{\iso}{\cong}
\newcommand{\mr}{\mathrm}

\newcommand{\mbA}{\mathbb{A}}

\newcommand{\mbC}{\mathbb{C}}

\newcommand{\mbE}{\mathbb{E}}
\newcommand{\mbF}{\mathbb{F}}
\newcommand{\mbG}{\mathbb{G}}

\newcommand{\mbL}{\mathbb{L}}

\newcommand{\mbN}{\mathbb{N}}

\newcommand{\mbQ}{\mathbb{Q}}
\newcommand{\mbR}{\mathbb{R}}

\newcommand{\mbV}{\mathbb{V}}

\newcommand{\mbX}{\mathbb{X}}
\newcommand{\mbY}{\mathbb{Y}}
\newcommand{\mbZ}{\mathbb{Z}}

\newcommand{\mcC}{\mathcal{C}}
\newcommand{\mcD}{\mathcal{D}}
\newcommand{\mcE}{\mathcal{E}}
\newcommand{\mcF}{\mathcal{F}}
\newcommand{\mcG}{\mathcal{G}}

\newcommand{\mcL}{\mathcal{L}}
\newcommand{\mcM}{\mathcal{M}}
\newcommand{\mcN}{\mathcal{N}}
\newcommand{\mcO}{\mathcal{O}}

\newcommand{\mcS}{\mathcal{S}}

\newcommand{\mcX}{\mathcal{X}}

\newcommand{\mcZ}{\mathcal{Z}}

\newcommand{\mfd}{\mathfrak{d}}

\newcommand{\mfp}{\mathfrak{p}}

\newcommand{\LNSch}{\ensuremath{(\mathrm{LNSch})}\xspace}

\newcommand{\bfF}{\mathbf{F}}
\newcommand{\bfE}{\mathbf{E}}

\begin{document}

\title{On the Arithmetic Fundamental Lemma conjecture over a general $p$-adic field}

\author{A. Mihatsch}
\author{W. Zhang}

\address{Universität Bonn, Mathematisches Institut, Endenicher Allee 60, 53115 Bonn, Germany}
\email{mihatsch@math.uni-bonn.de}
\address{Massachusetts Institute of Technology, Department of Mathematics, 77 Massachusetts Avenue, Cambridge, MA 02139, USA}
\email{weizhang@mit.edu}

\date{\today}

\begin{abstract}

We prove the arithmetic fundamental lemma conjecture over a general $p$-adic field with odd residue cardinality $q\geq \dim V$. Our strategy is similar to the one used by the second author during his proof of the AFL over $\mbQ_p$ \cite{Z19}, but only requires the modularity of divisor generating series on the Shimura variety (as opposed to its integral model). The resulting increase in flexibility allows us to work over an arbitrary base field.
To carry out the strategy, we also generalize results of Howard \cite{H-CM} on CM cycle intersection and of Ehlen--Sankaran \cite{ES} on Green function comparison from $\mbQ$ to general totally real base fields.
\end{abstract}

\maketitle

\tableofcontents

\section{Introduction}
\subsection{Main results}
\label{ss:Intro Main}

Let $F/F_0$ be an unramified quadratic extension of $p$-adic local fields with $p\neq 2$. The  arithmetic fundamental lemma (AFL) conjecture, first formulated by the second author in \cite{Z12}, is an identity 
relating the following two types of quantities.
\begin{enumerate}
\item
Special values of the derivative of orbital integrals for the symmetric space
\begin{equation*}
     S_{n}:=(\Res_{F/F_0}\GL_{n})/\GL_{n,F_0}\simeq \{\,g\in \Res_{F/F_0}\GL_{n}\mid g\ov g=1_{n}\,\},
\end{equation*}
for a natural action of $\GL_{n-1,F_0}$;
\item Arithmetic intersection numbers on a Rapoport--Zink formal moduli space (abbreviated as ``RZ space" henceforth) of $p$-divisible groups with PEL type structure associated to a unitary group.
\end{enumerate}
We refer to \S\ref{ss:RZ}--\ref{ss:AFL} for the precise statements. The main result of this paper is a confirmation of the AFL conjecture under the assumption that the residue field of $F_0$ has at least $n$ elements, cf. Thm.~\ref{thm main}.\footnote{In the more recent preprint \cite{ZZ}, Z. Zhang adapts and extends the strategy of the present article to remove the assumption that the residue field of $F_0$ has at least $n$ elements, and to establish the analog of the AFL in many cases of bad reduction.} The strategy is similar to, but more refined than the one employed by the second author in \cite{Z19}, where he shows the AFL for $F_0=\BQ_p$ and $p\geq n$. We describe our strategy in more detail below.

The AFL conjecture arises from the relative trace formula approach to the arithmetic Gan--Gross--Prasad conjecture (for $\U_n\times\U_{n-1}$) and we refer to \cite{RSZ3,Z19} for a more detailed introduction to its origin and its relation to the Gross--Zagier theorem \cite{GZ}. In particular, the orbital integrals appear on the geometric side of Jacquet--Rallis' relative trace formula \cite{JR}, while certain Rankin–Selberg L-functions appear on the spectral side. The arithmetic intersection numbers on the RZ space appear as a local component of a global intersection problem via the uniformization theorem \cite{RZ}. Therefore, the AFL conjecture implies certain identities for intersection numbers on integral models of unitary Shimura varieties, which has been made precise by Rapoport, Smithling and the second author in \cite{RSZ3}. We prove their ``semi-global conjecture'' in the hyperspecial case in Thm.~\ref{thm semi-global} under the assumption of not too small residue field cardinality.

\subsection{Proof of the AFL}
The overall strategy is similar to \cite{Z19}: We extract the local intersection number from a global one that may be accessed by using the modularity of certain generating series of cycles on unitary Shimura varieties. In order to explain these ideas and to point out the differences with \cite{Z19}, we now recall the arithmetic intersection pairing from \cite{Z19} and then explain the variant employed here.

Let $\bfF$ be a number field and $\CX$ a regular, flat, projective scheme over $\Spec O_\bfF$ with smooth generic fiber $\CX_\bfF$. Denote by $\wt\CZ_1(\CX)$ the group of $1$-cycles on $\CX$ and let $\wh\Ch{}^1(\CX)$ denote the arithmetic Chow group of cycles of codimension $1$. The latter is isomorphic to the group of isomorphism classes of metrized line bundles $\wh{\mr{Pic}}(\mcX)$. In this paper, all Chow groups have $\BQ$-coefficients. There is an arithmetic intersection pairing (cf.  \cite[\S2.3]{BGS})
\begin{align}
 \begin{gathered}
	\xymatrix@R=0ex{
	 (\cdot\,,\cdot)\colon \wh\Ch{}^1(\CX)\times\wt \CZ_{1}(\CX)  \ar[r] & \mbR.
	}
	\end{gathered}
\end{align}
It is given by taking the arithmetic degree of the metrized line bundle restricted to the curve, $(\wh{\mcL}, C) = \wh{\deg}(\wh{\mcL}\vert_C)$. 
 Let $ \CZ_{1}(\CX) $  be the quotient of $\wt \CZ_{1}(\CX)$ by the subgroup generated by $1$-cycles on $\CX$ that are contained in a closed fiber and rationally trivial within that fiber. Then by \cite[Prop.\ 2.3.1 (ii)]{BGS}, the above pairing factors through
 \begin{align}
 \begin{gathered}
	\xymatrix@R=0ex{
(\cdot\,,\cdot)\colon\wh\Ch{}^1(\CX)\times \CZ_{1}(\CX)   \ar[r] & \mbR.
	}
	\end{gathered}
\end{align}
Our case of interest is when $\mcX/O_\bfF$ is an integral model for a unitary Shimura variety. The Shimura datum will be defined for a hermitian $F$-vector space where $F/F_0$ is a quadratic CM extension of a totally real field that globalizes the local fields from \S\ref{ss:Intro Main}. Then $\bfF$ is our notation for the reflex field.

Just as in \cite{Z19}, the two kinds of cycles of interest are the Kudla--Rapoport divisors (elements of $\wh\Ch{}^1(\CX)$) and certain (derived) CM cycles (elements of $\CZ_{1}(\CX)$). One of the key inputs in \cite{Z19} was the modularity of generating series of arithmetic KR divisors as elements in $\wh\Ch{}^1(\CX)$. More precisely, the modularity of the generating series of intersection numbers with a fixed CM cycle allows to run an inductive argument. This modularity was proved by Bruinier, Howard, Kudla, Rapoport and Yang in \cite{BHKRY} when $F_0=\BQ$ by a detailed study of the divisors of Borcherds products on the integral model in question. However, it appears that such modularity results are difficult to obtain when $F_0\neq \BQ$ and this was why the restriction $F_0 = \BQ$ appeared in \cite{Z19}.

The new idea in this paper is to only use modularity of the KR generating series for the generic fiber, i.e. with coefficients in $\Ch^1(\mcX_\bfF)$. This result for general $F_0$ is deduced by Liu \cite{Liu1} from its analogue for orthogonal Shimura varieties which was proved  by Yuan, S. Zhang and the second author in \cite{YZZ1}; the modularity for the generating series with coefficients in cohomology groups was shown by Kudla--Millson in much greater generality \cite{Kud-Mil} and was a key ingredient in \cite{YZZ1}. Moreover, we assume $F_0\neq \mbQ$ which implies properness for $\mcX$ and allows to simplify the above intersection theory. For simplicity, we now additionally assume that $\CX$ is smooth over $\Spec O_\bfF$. Let $\CZ_{1}(\CX) _{\deg=0}$ be the subgroup of $\CZ_{1}(\CX) $ consisting of $1$-cycles with vanishing degree on every  connected component of $\mcX_\bfF$. Let $\wh\Ch{}^{1,\adm}(\CX)\subset \wh\Ch{}^1(\CX)$ be the subgroup generated by cycles with \emph{admissible} Green function, meaning they have harmonic curvature with respect to the naturally metrized Hodge bundle on $\mcX_\bfF$. Then it is easy to see (cf. \S\ref{ss:AIP}) that the restriction of the pairing \eqref{eq:AI} to $\wh\Ch{}^{1,\adm}(\CX)\times \CZ_{1}(\CX) _{\deg=0}$ factors through
 \begin{align}
 \begin{gathered}
	\xymatrix@R=0ex{
(\cdot\,,\cdot)\colon \Ch^1(\CX_\bfF)\times \CZ_{1}(\CX)_{\deg=0}   \ar[r] & \BR,
	}
	\end{gathered}
\end{align}
where $\wh\Ch{}^{1,\adm}(\CX)\to  \Ch^1(\CX_\bfF)$ is the natural map that takes the generic fiber of cycles. In effect, the pairing is computed by extending a class in $\Ch^1(\CX_\bfF)$ to one in  $\wh\Ch{}^{1,\adm}(\CX)$ and then applying the pairing \eqref{eq:AI}. Now one is in a position to apply the modularity results from \cite{YZZ1} for KR divisors in $\Ch^1(\mcX_\bfF)$. In practice, to extract a specific local term from the global intersection number, we need to allow general level structure at a finite set of non-archimedean places $S$ and hence we actually use an $S$-punctured variant of the above, see \S\ref{s:alm mod}.

Just as in \cite{Z19}, it is also crucial to understand the archimedean local intersection, i.e. the values of Green functions at CM cycles. It turns out that the automorphic Green function constructed by Bruinier in \cite{Br12} is an admissible choice in the above sense. We prove that the generating series obtained by taking the difference with the Kudla Green function is modular. This was shown before by Ehlen--Sankaran \cite{ES} for $F_0 = \mbQ$. The extension to general $F_0$ we prove here is obtained by a rather straightforward proof that relies on the properness of $\mcX_\bfF$. This is carried out in Part 1.

The other obstacle is that we need to modify the CM cycle to land in the smaller $\CZ_{1}(\CX)_{\deg=0}$. This modification should be as simple as possible so that we keep control of the resulting modification of intersection numbers. Our strategy is to subtract a suitable linear combination of \emph{elementary} CM-cycles. Their intersection pairings with KR divisors have been studied by various authors (e.g. by Howard \cite{H-CM} when $F_0=\BQ$ as well as earlier by Kudla--Rapoport--Yang \cite{KRY}) and are essentially given by (the derivative of) Fourier coefficients of Siegel--Eisenstein series. This is carried out in Part 2.

We also mention the recent work of the first author \cite{M-LC} which proved the local constancy of the intersection number occurring in the AFL conjecture for all regular semi-simple elements. This allows, both for the present paper and for \cite{Z19}, to spread out the AFL identity from \emph{strongly} regular semi-simple elements to all regular semi-simple ones.

The strategy in this paper is quite flexible and  in a future paper the authors plan to explore the possibility of a new proof of the averaged Colmez conjecture. Moreover, it seems likely that some enhancement of the arguments in Part 1 will also establish the modularity of generating series of arithmetic KR divisors as elements in $\wh\Ch{}^1(\CX)$ under certain conditions (e.g. when $\CX$ has good reduction everywhere). 

Naturally, the present paper makes use of several results and constructions from previous works, notably \cite{Br12,RSZ3,Z19}. Our policy is to recall the necessary background from these sources whenever needed to make the paper as self-contained as possible.

\subsection*{Acknowledgment}
We thank Congling Qiu, Michael Rapoport, and Zhiyu Zhang for comments on an earlier version of this article. We also thank the referees for their helpful comments.

This work was initiated and partially completed during a research stay of the first author (A. Mi.) at MIT. He thanks MIT for offering such a hospitable working environment and also thanks the DFG (Deutsche Forschungsgemeinschaft) for making this stay possible through grant MI 2591/1-1. W. Z. is partially supported by the NSF grant DMS \#1901642.

\subsection{Notation}\label{notation}

\subsubsection*{Notation on algebra}

\begin{itemize}
\item 
For a global field $\bF$, we denote by $\Sigma_{\bfF}$ the set of all places.
\item  Unless otherwise stated, $F$ denotes a CM number field and $F_0$ denotes its (maximal) totally real subfield of index $2$.  We denote by $a\mapsto \ov a$ the nontrivial automorphism of $F/F_0$.   Let $F_{0,+}$ (resp., $F_{0,\geq0}$) be the set of totally positive (resp., semi-positive) elements in $F_0$.
\item We denote $\bH=\SL_2$ as an algebraic group over $F_0$. Denote by $B$ the Borel subgroup of upper triangular matrices, $N$ its unipotent radical. By abuse of notation we let $\wt\bH$ denote the metaplectic double covering of $\bH=\SL_2$ even though it is not an algebraic group.
\item 
Unless otherwise stated, we write $\BA$, $\BA_{0}$, and $\BA_F$ for the adele rings of $\BQ$, $F_0$, and $F$, respectively.  We systematically use a subscript $f$ for the ring of finite adeles, and a superscript $p$ for the adeles away from the prime number $p$.

\item 
For an abelian scheme $A$ over a locally noetherian scheme $S$ on which the prime number $p$ is invertible, we write $\RT_p(A)$ for the $p$-adic Tate module of $A$ (regarded as a smooth $\BZ_p$-sheaf on $S$) and $\RV_p(A) := \RT_p(A) \otimes \BQ$ for the rational $p$-adic Tate module (regarded as a smooth $\BQ_p$-sheaf on $S$).
\item 
We use a superscript $\circ$ to denote the operation $-\otimes_\BZ \BQ$ on groups of homomorphisms of abelian schemes, so that for example $\Hom^\circ(A,A') := \Hom(A,A')\otimes_\BZ \BQ$.

\item All Chow groups and $K$-groups have $\BQ$-coefficients. 

\item \label{Herm2Quad}
For a vector space $V$  over a field $F_0$ (of characteristic not equal to $2$), a quadratic form $\fkq\colon  V\to F$ has an associated symmetric bilinear pairing defined by
\begin{align}\label{eq:q2bi}
\pair{x,y}=\fkq(x+y)-\fkq(x)-\fkq(y),\quad x,y\in V.
\end{align}
In particular, 
\begin{align}\label{eq:q2bi2}
\pair{x,x}=2\fkq(x).
\end{align}
For a quadratic field extension $F$ of $F_0$, an $F/F_0$-hermitian space is an $F_0$-vector space $V$ endowed with an $F_0$-linear action of $F$ and an ``$F/F_0$-hermitian form", i.e., a map $\pair{\cdot,\cdot}\colon  V\times V\to F$ that is $F$-linear on the first factor and conjugate symmetric. Its dimension will be the dimension as an $F$-vector space. It  induces  a symmetric bi-$F_0$-linear pairing by $(x,y)\mapsto \tr_{F/F_0}\pair{x,y}\in F_0$. In particular, the corresponding quadratic form on $V$ is 
\begin{align}\label{eq:her2q}
\fkq(x)=\pair{x,x}\in F_0.
\end{align} We will treat $V$ as an affine variety over $F_0$, and for $\xi\in F_0$ we denote by $V_\xi$ the subscheme defined by $\fkq(x)=\xi$.

\item For a $F/F_0$-hermitian space $V$ over a non-archimedean local field, and an $O_F$-lattice $\Lambda\subset V$ (of full rank), we denote by $\Lambda^\vee=\{x\in V\mid \pair{x,\Lambda}\subset O_F\}$ its dual lattice under the hermitian form. 

\item Let $R$ be a commutative ring. We denote by $\LNSch_{/R}$ the category of locally noetherian schemes over $\Spec R$. 

\item We denote by $F[t]^\circ_{\deg n}$  the set of monic polynomials $α$ of degree $n$ such that $α(0)t^n\ov{α}(t^{-1}) = α(t)$.
\end{itemize}

\subsubsection*{Notation on automorphic forms}

\begin{itemize}
\item 
Fix the non-trivial additive character $\psi=\psi_{\BQ}\circ\tr_{F_0/\BQ}\colon  F_0\bs\BA_{0}\to \BC^\times$ where $\psi_{\BQ}$ is the standard one and $\tr_{F_0/\BQ}\colon F_0\bs \BA_{0}\to\BQ\bs \BA$ is the trace map. For $\xi\in F_0$ we denote by $\psi_\xi$ the twist $\psi_\xi(x)=\psi(\xi x)$. Let $\delta_{E/F}$ denote the different ideal for any extension $E/F$ of number fields.
\item For a smooth algebraic variety $X$ over a local field $F$, we denote by $\CS(X(F))$ the space of Schwartz functions on $X(F)$.
Similarly, for a smooth algebraic variety $X$ over a global field $F$, we denote $\CS(X(\BA))$ by the space of Schwartz functions on $X(\BA)$.

\item For a quadratic space $V$ over a number field $F_0$, we  denote by $\omega$ the Weil representation of ${\rm O}(V)(\BA_0)\times\wt\bH(\BA_0)$ on  $\CS(V(\BA_0))$. If $\dim V$ is even, the representation factors through ${\rm O}(V)(\BA_0)\times\bH(\BA_0)$. If $V$ is an $F/F_0$-hermitian space, we view it as a quadratic space over $F_0$ and denote again by $\omega$ the  Weil representation. See \cite[\S11]{Z19} for an explicit description of $\omega$.

\item For $\xi\in\BR$ and $k\in\BZ$, the weight $k$ Whittaker function on $\SL_2(\BR)$ is defined by
\begin{align}\label{Whit}
W^{(k)}_{\xi}(h)= |a|^{k/2} e^{2\pi i \xi (b+ai)} \chi_k(\kappa_\theta),
\end{align}
where we write $h\in \SL_2(\BR)$ according to the Iwasawa decomposition 
\begin{align}\label{h infty}
h=\left(\begin{matrix} 1 & b \\
& 1
\end{matrix}\right)
\left(\begin{matrix} a^{1/2} & \\
& a^{-1/2}
\end{matrix}\right)\,\kappa_\theta,\quad a\in\BR_{+},\quad b\in \BR,
\end{align}
and
\begin{align}\label{kappa in SO2}
\kappa_\theta=\left(\begin{matrix}\cos\theta&\sin\theta\\- \sin\theta&
\cos\theta\end{matrix}\right)\in \SO(2,\BR).
\end{align}
Here the weight $k$ character of $\SO(2,\BR)$, for $k\in\BZ$, is defined by
\begin{align}\label{chi}
\chi_k(\kappa_\theta)= e^{i k  \theta}.
\end{align} 
The preimage $\SO(2,\BR)$ in $\wt\SL_2(\BR)$ is denoted by $\wt\SO(2,\BR)$. For $k\in\frac{1}{2}\BZ$, we similarly define the weight $k$ character of $\wt\SO(2,\BR)$ and the Whittaker function.

\item
$\CA_\infty(\bH(\BA_{0}),K_\bH, k)$: the space of smooth functions on $\bH(\BA_0)$, invariant under left $\bH(F_0)$ and right $K_\bH\subset\bH(\BA_{0,f})$ translation, and of parallel weight $k$ under the action of $\prod_{v\in \Hom(F_0,\BR)}\SO(2,\BR)$. (We do not impose any finiteness condition under the center of universal enveloping algebra.) Similarly we define $\CA_\infty(\wt\bH(\BA_{0}),K_{\wt \bH}, k)$.

\item For a left $N(F_0)$-invariant  continuous function $\phi\colon \bH(\BA_0)\to\BC$, its $\xi$-th Fourier coefficient  for $\xi\in F_0$ is defined as the function
\begin{align}\label{eq:def F coeff}
h\in\bH(\BA_0)\mapsto \phi(\xi,h) :=\int_{F_0\bs \BA_0}  \phi\left[\left(\begin{matrix} 1&b\\
&1
\end{matrix}\right) h\right]\psi_{-\xi}( b) db.
\end{align}
Then there is a Fourier expansion (by an absolutely convergent sum): For $h\in  \bH(\BA_0)$,
\begin{align}\label{eq:def F exp}
\phi(h)=\sum_{\xi\in F_0} \phi(\xi,h).
\end{align}

\item
$\CA_{\rm hol}(\bH(\BA_{0}),K_\bH, k)$: the space of automorphic forms (with moderate growth) on $\bH(\BA_0)$,  invariant under $K_\bH\subset\bH(\BA_{0,f})$,  of parallel weight $k$ under the action of $\prod_{v\in \Hom(F_0,\BR)}\SO(2,\BR)$, and holomorphic (i.e., annihilated by the element $\frac{1}{2}\left(\begin{matrix}i &1\\1&
-i\end{matrix}\right)$ in the complexifed Lie algebra of $\bH(F_{0,v})\simeq\SL_2(\BR)$ for every $v\in \Hom(F_0,\BR)$).
This is a finite dimensional vector space over $\BC$, and it has a $\ov\BQ$-structure via the Fourier expansion. For any subfield $L\subset\BC$ (for our applications, it suffices to assume $L\supset\ov\BQ$), we define  $\CA_{\rm hol}(\bH(\BA_{0}),K_\bH, k)_L$ as the subspace of $\CA_{\rm hol}(\bH(\BA_{0}),K_\bH, k)$ consisting of functions with Fourier coefficients all in $L$ (i.e., if we write the $\xi$-th term in the expansion \eqref{eq:def F exp} as $ \phi(\xi,h) =W^{(k)}_{\xi}(h_\infty) \phi_f(\xi,h_f)$, then $\phi_f(\xi,h_f)\in L$ for all $h_f\in \prod_{v\nmid\infty} \bH( O_{F_{0,v}})$). Then for any $L$-vector space $W$, we have an $L$-vector space:
\begin{align}\label{def A hol W tot}
\CA_{\rm hol}(\bH(\BA_{0}),K_\bH, k)_L\otimes_L W.
\end{align}

\end{itemize}

\part{KR divisors generating series}
\section{Shimura varieties and Special divisors}\label{s:SV KR}

We work with the so-called RSZ variant of unitary Shimura varieties throughout the paper. It goes back to \cite{KR-U2} but has been introduced systematically in \cite{RSZ3}. Its advantage is that it admits both a PEL type moduli description and the flexible definition of cycles. This section's aim is to recall all relevant definitions and properties. In that, it overlaps considerably with \cite[§6 and §7]{Z19}. We refer to \cite{RSZ4} for a much more detailed survey.

\subsection{Shimura varieties}\label{ss:S data}
Let $F/\mbQ$ be a CM field with maximal totally real subfield $F_0$ and non-trivial $F_0$-automorphism $a\mapsto \ov{a}$. Fix a CM type $Φ$ for $F$, a distinguished element $φ_0\in \Phi$ and a hermitian $F$-vector space $V$ of dimension $n$ such that, for $φ\in \Phi$,
$$\mr{sign}(V_φ) = \begin{cases} (n-1,1) & \text{if $φ=φ_0$}\\
(n,0) & \text{if $φ\neq φ_0$.}
\end{cases}
$$
Then we consider the following algebraic groups over $\Spec \mbQ$ resp. $\Spec F_0$,
\begin{align*}
Z^\mbQ &:= \{z\in \Res_{F/\mbQ} \mbG_m\mid \Nm_{F/F_0}(z)\in \mbG_m\},\\
G &:= \U(V),\ \ \text{an algebraic group over $F_0$,}\\
G^\BQ &:= \bigl\{g \in \Res_{F_0/\BQ} \GU(V) \bigm| c(g)\in \BG_m\bigr\},\ \ \mr{and}\\
\wt G &:= Z^\BQ \times_{\BG_m} G^\BQ \isoarrow Z^{\mbQ}\times \Res_{F_0/\mbQ}G.
\end{align*}
Here, $c\colon \GU(V)\to \mbG_{m,F_0}$ denotes the similitude character, while the isomorphism in the last line is given by $(z,g)\mapsto (z,z^{-1}g)$.
We now recall how to define Shimura data for the above groups. For $Z^\mbQ$ we take the Shimura datum defined by $Φ$,
\begin{equation}\label{eq Sh dat Z}\begin{aligned}
h_{Z^\mbQ}\colon \Res_{\mbC/\mbR}\mbG_m & \to  Z^\mbQ_{\mbR} \subseteq \prod_{φ\in Φ} \Res_{\mbC/\mbR}\mbG_m\\
z & \mapsto (z,\ldots,z).
\end{aligned}
\end{equation}
Next, choose $\mbC$-bases of the $V_φ,\ φ\in Φ$, such that the hermitian form is given by $\diag(1,\ldots,1)$ if $φ\neq φ_0$ resp. $\diag(1,\ldots,1,-1)$ if $φ = φ_0$. Then set
$$\begin{aligned}
h\colon \Res_{\mbC/\mbR}\mbG_m & \to  (\Res_{F_0/\mbQ}G)_\mbR = \prod_{φ\in Φ} \U(V_φ),\ \ z \mapsto \big(h_φ(z)\big)_{φ\in Φ},\\
						h_φ(z) & = \begin{cases} \diag\big(1,\ldots,1,\ov z/z\big) & φ = φ_0\\
						\mr{id}_{V_φ} & φ \neq φ_0\end{cases}
														\end{aligned}$$
where the matrix notation is with respect to the fixed choice of basis. Denote by $\{h\}$ the $(\Res_{F_0/\mbQ}G)(\mbR)$-conjugacy class of $h$. Then $(\Res_{F_0/\mbQ} G,\{h\})$ is a Shimura datum. Finally put
$$\big(\wt G, \big\{\wt h\big\}\big) := \big(Z^\mbQ \times \Res_{F_0/\mbQ} G,\ h_{Z^\mbQ} \times \{h\}\big).$$
Let $\bfF\subseteq \mbC$ denote the reflex field for $\big(\wt G, \big\{\wt h\big\}\big)$. It is the composite of the reflex fields for $Φ$ and for the conjugacy class $\{h\}$. It always contains $φ_0(F)$ which allows to view $F$ as a subfield of $\bfF$ (cf. \cite[(6.5)]{Z19}). For open compact subgroups $K_{Z^\mbQ}\subseteq Z^\mbQ(\mbA_f), K\subseteq G(\mbA_{0,f})$ and $\wt K := K_{Z^\mbQ}\times K$, we obtain Shimura varieties (over $\bfF$) denoted by
\begin{equation}\label{prod shim}
\Sh_{\wt K}\big(\wt G, \big\{\wt h\big\}\big)\simeq
   \Sh_{K_{Z^\BQ}}\big(Z^\BQ,h_{Z^\BQ}\big) \times \Sh_K\big(\Res_{F_0/\mbQ} G,\{h\}\big).
\end{equation}
Here and in the following we make the simplifying assumption that $K_{Z^\mbQ}$ and $K$ are neat. Then $\Sh_{\wt K}(\wt G, \{\wt h\})$ is an actual variety instead of a stack. (This assumption will not be used in an essential way and all our results should hold in the more general setting.) From now on we write $M_{\wt K}$ or even $M$ instead of $\Sh_{\wt K}\big(\wt G, \big\{\wt h\big\}\big)$.

\subsection{Integral Models}
\label{ss:integral models}
We next recall the integral model of $M_{\wt K}$ at places where the given data are unramified. This is the special case \cite[\S5.1]{RSZ3} and \cite[\S6.1]{RSZ4}. Let $Λ \subseteq V$ be a fixed $O_F$-lattice and let $\mfd\in \mbZ_{>0}$  be chosen such that:
\begin{altitemize}
\item The $O_F[\mfd^{-1}]$-lattice $Λ[\mfd^{-1}]$ is self-dual with respect to the $F$-valued pairing on $V$.
\item If $v\in \Sigma_{F_0}$ is non-archimedean with $v\nmid \mfd$, then $v$ has odd residue characteristic and is unramified in $F$.
\item The level $K$ factors as
$$K = K_\mfd\times K^\mfd$$
where $K_\mfd\subseteq G(F_{0,\mfd})$ is any open compact subgroup and where $K^\mfd = \mr{Stab}(\wh {Λ}^\mfd)$ is the stabilizer of the completion $\wh {Λ}$ in $G(\mbA_{0,f}^\mfd)$.
\item The level $K_{Z^\mbQ}$ factors similarly as
$$K_{Z^\mbQ} = K_{Z^\mbQ,\mfd}\times K_{Z^\mbQ}^{\mfd}$$
where $K_{Z^\mbQ, \mfd}\subseteq Z^\mbQ(\mbQ_\mfd)$ is any open compact subgroup and where $K_{Z^\mbQ}^\mfd$ is the maximal compact subgroup in $Z^\mbQ(\mbA_f^\mfd)$.
\end{altitemize}
We next consider a moduli problem $\mcM_0$ of abelian varieties with CM by $F$. It is defined as the functor on $\mr{LNSch}_{/O_{\bfF}[\mfd^{-1}]}$ that associates to $S$ the set of isomorphism classes of $(A_0, \iota_0, \lambda_0, \ov{η}_0)$, where 
\begin{altitemize}
\item $A_0$ is an abelian scheme over $S$ of dimension $[F_0:\mbQ]$;
\item $\iota_0\colon O_F[\mfd^{-1}]\to \End(A_0)[\mfd^{-1}]$ is an $O_F[\mfd^{-1}]$-action that satisfies the Kottwitz condition for $Φ$,
\begin{equation}\label{kottcondA_0}
   \charac\bigl(\iota_0(a)\mid\Lie A_0\bigr) = \prod_{\varphi\in\Phi}\bigl(T-\varphi(a)\bigr)
   \quad\text{for all}\quad
   a\in O_F[\mfd^{-1}];
\end{equation}
\item $\lambda_0$ is an away-from-$\mfd$ principal polarization on $A_0$ such that $λ_0^{-1}\circ ι_0(a)^\vee \circ λ_0 = ι_0(\ov a)$ for all $a\in O_F[\mfd^{-1}]$ and
\item $\ov{η}_0$ a $K_{Z^{\mbQ},\mfd}$-level structure. This last notion requires the choice of an auxiliary $1$-dimensional hermitian $F_{\mfd}$-module. It will never figure explicitly in this work, so we do not spell out the details and refer to \cite[\S C.3]{Liu18} instead.
\end{altitemize}
An isomorphism of two quadruples $(A_0,\iota_0,\lambda_0,\ov{η}_0)$  and $(A'_0,\iota'_0,\lambda'_0,\ov{η}'_0)$ is a quasi-isogeny $φ_0\in\Hom_{O_F}(A_0,A'_0)[\mfd^{-1}]$ such that $φ_0^*\lambda_0'= λ_0$ and $ \ov{η}'_0\circ φ_0= \ov{η}_0$. Then $\CM_0$ is representable, finite and \'etale over $\Spec O_\bfF[\mfd^{-1}]$, cf.\ \cite[Prop.\ 3.1.2]{H-CM}. Here we used our assumption that $K_{Z^\mbQ}$ is neat, without it $\mcM_0$ would have to be defined as a stack.

Depending on $F$ and $\mfd$, it might happen that $\mcM_0$ is empty. This does not happen, however, if $F/F_0$ is ramified at some place (cf. \cite[Rem.\ 3.5 (ii)]{RSZ3}) and we make that assumption throughout the paper.

Then the generic fiber $M_0$ of $\CM_0$ is a disjoint union of copies of $\Sh_{K_{Z^\BQ}}\big(Z^\BQ,\{h_{Z^\BQ}\}\big)$ (cf. \cite[Lem.\ 3.4]{RSZ3} specialized to the ideal $\fka=O_{F_0}$). This reflects the failure of the Hasse principle for the group $Z^\mbQ$. It would be possible to work with $\mcM_0$ throughout the paper, but it is more convenient to work directly with an integral model of $\Sh_{K_{Z^\mbQ}}$. For this reason, we fix one of the copies $\Sh_{K_{Z^\BQ}}\big(Z^\BQ,\{h_{Z^\BQ}\}\big)\subseteq M_0$ and \emph{redefine} $\mcM_0$ as its closure. Note that $\CM_0$ is still finite étale over $\Spec O_{\bfF}$. We also \emph{redefine} $M_0$ as the generic fiber of this new $\mcM_0$.

We next define the integral model of $M = M_{\wt K}$ over $\Spec O_{\bfF}[\mfd^{-1}]$.
\begin{definition}\label{def RSZ glob}
The functor $\CM = \CM_{\wt K}$ associates to each scheme $S$ in $\mr{LNSch}_{/O_{\bfF}[\mfd^{-1}]}$ the set of isomorphism classes of $(A_0,\iota_0,\lambda_0,\ov{η}_0, A,\iota,\lambda,\ov\eta)$, where
\begin{altitemize}
\item $(A_0,\iota_0,\lambda_0,\ov\eta_0)$ is an object of $\CM_{0}(S)$; 
\item $A$ is an abelian scheme over $S$ of dimension $n[F_0:\mbQ]$;
\item $\iota\colon O_F[\mfd^{-1}] \to \End(A)[\mfd^{-1}]$ is an action satisfying the Kottwitz condition of signature
$$\big(( n-1,1)_{\varphi_0}, (n,0)_{\varphi\in\Phi\ssm\{\varphi_0\}}\big)$$
on $O_F[\mfd^{-1}]$; and
\item $\lambda\colon A \to A^\vee$ is an away-from-$\fkd$ principal polarization such that $λ^{-1}\circ ι(a)^\vee \circ λ = ι(\ov a)$ for all $a\in O_F[\mfd^{-1}]$;
\item $\ov\eta$ is a $K_{\mfd}$-orbit  of isometries of hermitian modules  
(as smooth $F_{\fkd}=\prod_{v\mid \fkd}F_{ v} $-sheaves on $S$ endowed with its natural hermitian form induced by the polarization)
\begin{align}\label{level}
    \xymatrix{ \eta\colon  \RV_\fkd(A_0, A) \ar[r]^-\sim& V(F_{0,\fkd})},
\end{align}
where
$$
 \RV_\fkd(A_0, A):=\prod_{p\mid\fkd}  \RV_p(A_0, A), \quad \text{and} \quad   \RV_p(A_0, A)=\Hom_{F\otimes_\BQ\BQ_p}(\RV_p(A_0), \RV_p(A)),
 $$
 and
\begin{align}\label{eq:def V d}
 V(F_{0,\fkd}):= \prod_{p\mid \fkd}\mbQ_p\otimes_\mbQ V = \prod_{v\mid \fkd}F_{0,v}\otimes_{F_0}V.
\end{align}

\item Finally, we impose the Eisenstein condition (cf. \cite[\S5.2]{RSZ4}) for every place $v\nmid\fkd$ of $F_0$. Note that this condition already follows from the Kottwitz condition for $v$ that are unramified over $\mbQ$.
\end{altitemize}

An isomorphism between two objects $(A_0,\iota_0,\lambda_0,\ov\eta_0,A,\iota,\lambda,\ov\eta)$  and $(A'_0,\iota'_0,\lambda'_0,\ov\eta_0',A',\iota',\lambda',\ov\eta')$ is an isomorphism $φ_0\colon (A_0,\iota_0,\lambda_0,\ov\eta_0) \isoarrow (A_0',\iota_0',\lambda_0',\ov\eta_0')$ together with a quasi-isogeny $φ\in \Hom_{O_F} (A,A')[\mfd^{-1}]$ such that $φ^*λ' = λ$ and $\ov\eta' \circ (φ_0,φ)= \ov \eta$.
\end{definition}

\begin{theorem}[\protect{\cite[Thm. 6.2]{Z19}}]\label{thm:representabilityZ19}
The functor $\CM_{\wt K}$ is representable. The morphism $\CM_{\wt K}\to\Spec O_{\bfF}[\mfd^{-1}]$ is separated, of finite type and smooth of relative dimension $n-1$.
\end{theorem}

In \cite{Z19}, the functor $\mcM$ is defined as stack and \cite[Thm. 6.2]{Z19} states that it is a Deligne--Mumford stack. The version we gave here used the assumption that $\wt K$ is neat. Moreover, the scheme $\mcM$ is projective if $F_0\neq \mbQ$, cf. \cite[Thm. 4.4 (i)]{RSZ4}, which we will assume henceforth.

By \cite[Prop.\ 3.5]{RSZ3}, the generic fiber of $\CM_{\wt K}$ is isomorphic to $M_{\wt K}$, the canonical model of $\Sh_{\wt K}\big(\wt G, \big\{\wt h\big\}\big)$. (Note that the failure of the Hasse principle for $\wt G$ is already taken care of by our (re)definition of $\mcM_0$ above.)

Concerning terminology, we usually write $(A_0,A,\ov{η})$ for points of $\mcM$, the remaining data being implicit.

\subsection{Kudla--Rapoport divisors}\label{ss:KR}

We next recall the definition of the Kudla--Rapoport special divisors \cite{KR-U2}. These parametrize certain homomorphisms from CM abelian varieties occurring in $\mcM_0$ to abelian varieties parametrized by $\mcM$. They may equivalently be thought of as the loci in $\mcM$ where a CM abelian variety factor splits off in a prescribed fashion. The definition is based on the following observation. Assume that $(A_0,A,\ov{η})\in \mcM(S)$ with $S$ connected. Then one has $\End^\circ_{F}(A_0) = F$ and the space $\Hom_{O_F}(A_0,A)$ becomes a hermitian $O_F$-module with pairing
$$
\pair{x,y}= \lambda_0^{-1}\circ y^\vee\circ \lambda \circ x\in \End^\circ _{F}(A_0)\simeq F.
$$
This space is positive definite at all archimedean places of $F_0$ by the positivity of the Rosati involution. Recall (cf. §\ref{Herm2Quad} on Notation) that $V_ξ$ denotes the hyperboloid of $x\in V$ with $\pair{x,x} = ξ$ and that we assume all Schwartz functions in the non-archimedean setting to be $\mbQ$-valued.
\begin{definition}
Given $0 \neq \xi\in F_{0}$ and $\mu\in V(F_{0,\fkd})/K_{\fkd}$, the \emph{KR-cycle} $\CZ(\xi, \mu)$ is defined as follows. For $S\in \mr{LNSch}_{/ O_\bfF[\mfd^{-1}]}$, the $S$-points $\mcZ(ξ,µ)(S)$ are the set of isomorphism classes of $(A_0,A,\ov{η},u)$ where
\begin{altitemize}
\item $(A_0,A,\ov{η}) \in \mcM(S)$ and
\item $u\in \Hom_{O_F}(A_0,A)[\mfd^{-1}]$ such that $\pair{u,u}=\xi$, and $\ov\eta(u)\in µ.$
\end{altitemize}

Two objects $(A_0,A,\ov{η}, u)$ and $(A'_0,A',\ov{η}', u')$ are isomorphic if there is an isomorphism of underlying triples, $(φ_0,φ)\colon (A_0,A,\ov{η})\isoarrow(A'_0,A',\ov{η}')$, such that $φuφ_0^{-1} = u'$.
\end{definition}

Forgetting $u$ defines a natural morphism $i\colon  \CZ(\xi, \mu)\to \CM$ which is finite and unramified. Moreover, $\mcZ(ξ, μ)$ is \'etale locally a Cartier divisor and the morphism $\CZ(\xi, \mu)\to \Spec O_\bfF[\mfd^{-1}]$ is flat (cf. \cite[Prop. 7.3]{Z19}). Taking its pushforward, the image of $\mcZ(ξ,μ)$ along $i$ defines an effective relative Cartier divisor on $\mcM$ which we still denote by $\mcZ(ξ,µ)$. This definition extends by linearity to $ϕ_\mfd \in \mcS(V_\mfd)^{K_\mfd},$
\begin{align}\label{eq:KR gen}
\mcZ(ξ,ϕ_\mfd) := \sum_{µ\in V(F_{0,\mfd})/K_{\mfd}} ϕ_\mfd(µ) \mcZ(ξ,µ) \in \CZ^1(\mcM),
\end{align}
where the right hand side denotes the group of divisors with $\mbQ$-coefficients. It is non-zero only for $ξ\in F_{0,+}$ because the hermitian pairing on $\Hom_{O_F}(A_0,A)$ is positive definite at all archimedean places.

We finally introduce the following notation. Whenever $ϕ \in \mcS(V(\mbA_{0,f}))$ is a Schwartz function that factors as $ϕ = ϕ_\mfd \otimes 1_{\wh{Λ}^\mfd}$, we put
\begin{equation}\label{eq:defn_int_model_KR_div}
\mcZ(ξ, ϕ) := \mcZ(ξ, ϕ_\mfd).
\end{equation}

\subsection{Generating series of KR divisors}
\label{s:gen div}

We turn to the generic fiber $M = M_{\wt K}$ again and write $Z(ξ,µ)$, $Z(ξ,ϕ_\mfd)$ and $Z(ξ, ϕ)$ for the generic fibers of the above cycles. Since $\mfd$ may be varied, we actually have defined a divisor $Z(ξ,ϕ)\subseteq M$ for every $ϕ\in \mcS(V(\mbA_{0,f}))^K$ and $ξ\in F_{0,+}$. Also put
\begin{equation}\label{KR C 0}
Z(0,\phi):=-\phi(0) \,c_1(\bm{\omega})\in \Ch^1(M_{\wt K}),
\end{equation}
where $\bm{\omega}$ is the automorphic line bundle \cite{K-duke}, and $c_1$ denotes the first Chern class. Note that this is just a line bundle while we defined actual divisors when $ξ \neq 0$.
Recall from \cite[(11.1)]{Z19} the Weil representation $\omega$ of $\bH(\BA_{0,f})$ on $\CS(V(\BA_{0,f}))^{K}$.
We define a generating series on $\bH(\BA_0)$ by 
\begin{equation}\label{gen  E}
Z(h,\phi)=Z(0,\omega(h_f)\phi)W^{(n)}_{0}(h_\infty)+\sum_{\xi\in F_{0,+}} Z(\xi,\omega(h_f)\phi)\, W^{(n)}_{\xi}(h_\infty),
\end{equation}
where $h=(h_\infty,h_f)\in\bH(\BA_{0})$, $h_\infty=(h_v)_{v\mid \infty} \in  \bH(F_{0,\BR})=\prod_{v\mid\infty}\SL_{2}(F_{0,v})$, and 
$$
W^{(n)}_{\xi}(h_\infty)=\prod_{v\mid \infty} W^{(n)}_{\xi}(h_v),
$$
is the weight $n$ Whittaker function $W^{(n)}_{\xi}$  on $  \bH(F_{0,\BR})$, cf. \eqref{Whit}. Here we recall \cite[Thm.~8.1]{Z19} for future reference.

\begin{theorem}\label{thm:mod E}
The generating series $Z(h,\phi)$ lies in $ \CA_{\rm hol}(\bH(\BA_0), K_\bH, n)_{\ov\BQ}\bigotimes_{\ov\BQ} \Ch^1(M_{\wt K})_{\ov\BQ}$, where $K_\bH\subset \bH(\BA_{0,f})$ is a compact open subgroup that fixes $\phi$ under the Weil representation. 
\end{theorem}
We refer to \eqref{def A hol W tot} for the definition of the vector space $ \CA_{\rm hol}(\bH(\BA_0),K_{\bH}, n)_{\ov\BQ}$.

\section{Green functions}\label{s:Green}
We recall the Green function of Kudla \cite{K}, and the automorphic Green function \cite{Br12}. The former is more convenient when comparing with the analytic side, while the latter is more suitable for proving (holomorphic) modularity of generating series. The difference between them is studied by Ehlen--Sankaran in \cite{ES} when $F_0=\BQ$.
Here we extend their result to a general totally real field $F_0$, cf. Thm. \ref{thm ES}. We first work in the more general setting of quadratic spaces and our notation in this section is slightly different from the rest of the paper.  We then specialize the situation to the induced quadratic space structure of a hermitian space. 

\subsection{The set up}
We recall some basic constructions, cf. \cite[\S2]{Br12}.
Let $V$ be a quadratic space over the totally real field $F_0$ and denote
\begin{align}\label{eq:dim V}
\dim _{F_0}V=m+2, \quad s_0:=\frac{\dim V-2}{2}=\frac{m}{2}.
\end{align}
 Let $G=\SO(V)$, a reductive group over $F_0$. Suppose that the signature of $V$ is $(m,2)$ at one distinguished archimedean place $v_0\in \Hom_\BQ(F_0,\BR)$, and $(m+2,0)$ at the remaining archimedean places. Let $\CD=\CD_{v_0}$ be the Grassmannian of oriented negative definite 2-dimensional subspaces of the real vector space $V_{v_0}=V\otimes_{F_0,v_0}\BR$.  Then there is an isomorphism 
\begin{align}\label{eq:D=G/K}
\xymatrix{\CD\ar[r]^-\sim&  \SO(V_{v_0})/K_{v_0}},
\end{align}
where $K_{v_0}\simeq\SO(m)\times \SO(2)$ is the stabilizer of a fixed base point of $\CD$. To endow $\CD$ with a complex structure, we extend the quadratic form  $\BC$-bilinearly to  $V_{v_0,\BC}=V_{v_0}\otimes_\BR \BC$  and identify $\CD$ with the open subdomain of a quadratic hypersurface in the projective space $\BP^1( V_{v_0,\BC})$ of $\BC$-lines
$$
\xymatrix{\left\{\BC\cdot z\in \BP(V_{v_0,\BC})\mid \pair{z,z}=0, \pair{z,\ov z}<0 \right\}\ar[r]^-\sim& \CD},
$$
where the map sends $z$ to the plane $\BR\, \Re(z)\oplus\BR \Im(z)$ with the induced orientation. The tautological line bundle on $\CD$ has a  natural hermitian metric:
\begin{align}\label{eq:Pet met}
|\!|z|\!|^2:=-\frac{1}{2}\pair{ z,\ov z}.
\end{align}
Its first Chern form
$$
\Omega=-dd^c\log |\!|z|\!|^2
$$
is $G(\BR)$-invariant and positive. Here we recall the differential operators 
$$
d=\partial+\ov\partial,\quad d^c=\frac{1}{4\pi i}\left(\partial-\ov\partial\right),
$$
and 
$$dd^c=-\frac{1}{2\pi i} \partial\ov\partial.
$$
The top degree wedge power $\Omega^{m}$ defines an invariant volume form on $\CD$.
 
For a compact open subgroup $K\subset G(\BA_{0,f})$, consider the complex analytic space
$$
M_K(\BC)=G(F_0)\bs\left [\CD \times G(\BA_{0,f})/K\right ].
$$ 
We will assume that $M_K(\BC)$ is compact. When $F_0\neq \BQ$, this holds automatically. The tautological line bundle on $\CD$ descends to a line bundle, called the automorphic bundle, $\bm{\omega}_K$ on $M_K$ and so does its hermitian metric \eqref{eq:Pet met}, which will be called  the Petersson metric $|\!|\cdot|\!|_{\rm Pet}$ on $\bm{\omega}_K$.  The first Chern form of  the Petersson metric remains $\Omega$ (viewed as a $(1,1)$-form on $M_K$).

Let $\Delta=\Delta_\CD$ be the Laplacian operator on $\CD$ induced by the Casimir element of the Lie algebra of $\SO(V_{v_0})$  and by the isomorphism \eqref{eq:D=G/K}. It is a constant multiple of the Laplacian operator induced by the Petersson metric. An explicit formula will be recalled below Lem.~\ref{lem Lap}, at least for a certain class of functions. 

Similar to the hermitian case in \S\ref{s:SV KR} but only over $\BC$, there are special divisors, also denoted by $Z(\xi,\phi)$, on $M_K(\BC)$ for every $ϕ\in \mcS(V(\mbA_{0,f}))^K$ and $\xi\in F_{0,+}$ \cite{YZZ1}. 

\subsection{Laplacian operator}
We recall an explicit formula for the action of the Laplacian $\Delta$ on functions on $\CD=\SO(V_{v_0})/K_{v_0}$ that are left invariant under a subgroup $H$ of the form $H_u$, the stabilizer of a non-zero vector $u\in V_{v_0}$ (not necessarily  with negative norm). Define the majorant associated to the point $z\in\CD$ (cf. \cite[(11.14)]{K}) as
\begin{align}\label{Ruz}
R(u,z)=-\pair{u_z,u_z},
\end{align}
where $u_z$ denotes the orthogonal projection of $u$ to $z$.

\begin{lemma}\label{lem Lap}
Let $\phi\colon  (0,\infty)\to\BC $ be a real analytic function and let $\wt\phi(u,z)= \phi(R(u,z))$ be the associated function on the domain $(V_{v_0}\times\CD)\setminus\{(u,z): u_z=0\}$. 
Then we have 
$$
\Delta \wt\phi= \wt{D\phi},
$$
where $D$ is the second order differential operator on $(0,\infty)$ given by (the coordinate being denoted by $R$)
$$
D\phi=\left(4R(\pair{u,u} +R)\left(\frac{d}{dR}\right)^2+4( \pair{u,u}+(s_0+1)R)\frac{d}{dR}\right)\phi.
$$

\end{lemma}
\begin{proof}
It suffices to consider the case $\pair{u,u}>0$. Then the general case follows by the analyticity and the fact that the locus of $(u,z)$ with $\pair{u,u}>0$ is open (and non-empty) in each of the two connected components of the real analytic manifold $(V_{v_0}\times\CD)\setminus\{(u,z): u_z=0\}$. 

Now fix $u$ such that $\alpha:=\pair{u,u}>0$. We introduce a normalized version of $R(u,z)$:
\begin{align*}
r(u,z)=\frac{R(u,z)}{\pair{u,u}}.
\end{align*}
It satisfies the scaling invariance $r(\lambda u,z)=r(u,z)$  for $\lambda\in \BR^\times$ and is clearly left invariant under the stabilizer $H$ of $u$.  

Let $z_0\in\CD$ be a base point such that $u\perp z_0$. Let $K$ (resp. $H$) be the stabilizer of $z_0$ (resp. $u$). Let $Y_0\in \fkg$ (the Lie algebra of $\SO(V_{v_0})$) be the element $Y_0$ in \cite[\S1.3]{OT} (also see the proof of Thm.~4.7 of \cite{BK03}). We will not need its precise form, but we point out that $Y_0$ depends on $z_0$ and that the signature in {\it loc. cit.} is of the form $(2,m)$ rather than $(m,2)$.  
 
For $g=\exp(t Y_0)$, where $\exp\colon \fkg\to \SO(V_{v_0})$ is the exponential map, and $z=gz_0$, one has 
$$
r(u,z)=\sinh ^2t
$$
by \cite[(4.46)]{BK03}.
Define $$\Phi(g)=\phi(R(u,gz_0)), \quad g\in \SO(V_{v_0})\setminus H K.$$ It is a smooth function on $\SO(V_{v_0})\setminus H K$, left $H$-invariant and right $K$-invariant. By Cartan decomposition for the symmetric space $H\bs \SO(V_{v_0})$, we have $\SO(V_{v_0})=H\cdot \{\exp(t Y_0): t\in\BR \}\cdot K$. Therefore, to compute $\Delta \Phi$, it suffices to evaluate $(\Delta \Phi)(\exp(t Y_0))$. Then 
$$\Phi(\exp(t Y_0))=\phi(\alpha\sinh ^2t).
$$

Denote $\varphi(t)=\phi(\alpha\sinh ^2t)$. By \cite[Prop.~2.1.1]{OT} computing the radial part of the Casimir operator (we also denote it by $\Delta$ since our Laplacian is defined by the restriction of the Casimir operator), we have
 $$
 (\Delta \Phi)(\exp(t Y_0))=\left(\frac{d^2}{dt^2}+\bR(t) \frac{d}{dt}\right)\, \varphi(t),
$$
where $\bR(t)$ denotes the function in {\it loc. cit.}
$$
\bR(t)= m_\lambda^+ \coth(t) + m_\lambda^- \tanh(t) + 2m_{2\lambda}^+ \coth(2t) + 2m_{2\lambda}^- \tanh(2t).
$$

Substitute back $R=\alpha\sinh ^2t$. Then by  \cite[\S2.3 and (2.5.1)]{OT} (where the substitution was $z=-\sinh^2t$) and a simple substitution, the differential operator becomes
$$
\Delta \phi(R(u,z) )=(D \phi)(R(u,z) ).
$$
where $D\phi=\left(4R(\alpha+R)\left(\frac{d}{dR}\right)^2+4( \alpha+(s_0+1)R)\frac{d}{dR}\right)\phi$. This completes the proof.
\end{proof}

\begin{remark}
It is natural to expect the formula to hold for all  smooth $\phi$ by repeating the proof of \cite{OT}. In fact, the proof shows that the formula holds for smooth $\phi$ when we restrict the domain to a smaller subset defined by $\pair{u,u}>0$.
\end{remark}
\subsection{A family of Laplacian eigen-functions}
We follow the construction of Oda--Tsuzuki  and Bruinier \cite{OT,Br12}. The automorphic Green function is constructed by a regularization out of the following family of functions on $\CD$. 

For $u\in V_{v_0}$, we write $u_{z^\perp} = u - u_z$ for the orthogonal projection to the (positive definite) $m$-dimensional space $z^\perp$.
For $u\in V_{v_0}$ with $\pair{u,u}\neq 0$, $z\in\CD$ and $s\in\BC$, we define  \cite[\S5, (5.4)]{Br12}
\begin{align}\label{eq:phi u z}
\phi(u,z,s)=\frac{\Gamma(\frac{s}{2}+\frac{m}{4})}{\Gamma(s+1)}\left(\frac{\fkq(u)}{\fkq(u_{z^\perp})}\right)^{\frac{s}{2}+\frac{m}{4}} F\left(\frac{s}{2}+\frac{m}{4},\frac{s}{2}-\frac{m}{4}+1,s+1; \frac{\fkq(u)}{\fkq(u_{z^\perp})}\right).
\end{align}
Here $F(a,b,c;x)$ denotes the Gauss hyper-geometric function and $\fkq(x)=\frac{1}{2}\pair{x,x}$. It follows from the definition that 
$$
\phi(u,z,s)=\phi(g u,gz,s),\quad g\in \SO(V_{v_0}),
$$
and 
$$
\phi(\lambda u,z,s)=\phi( u, z,s),\quad \lambda \in \BR^\times.
$$
Moreover, it is an eigenfunction for the Laplacian \cite[(5.6)]{Br12}\footnote{Note that our normalization of $\Delta$ corresponds to the Casimir element in \cite{OT}, which corresponds to the $-\frac{1}{4}\Delta$ in \cite{OT} by \cite[Prop.~7.2.1.(3)]{OT} with $c_\fkg=2$ (cf. the last line in \cite[Proof of (7.6.2), p.~530]{OT}) and to $8$ times the $\Delta$ in \cite{Br12}. We also refer to the proof of \cite[Thm.~4.7]{BK03} for the comparison of the normalization on Laplacian and measures.}
\begin{align}\label{eq:Del phi}
\Delta\phi(u,z,s)=(s^2-s_0^2)\phi(u,z,s).
\end{align}
\begin{remark}
This function is related to the secondary spherical function $\phi_s^{(2)}$ of Oda--Tsuzuki \cite{OT} as follows. Fix a base point  $ z_0\in \CD$ and let $u\in V_{v_0}$ be a non-zero vector such that $u\perp z_0$ (and hence $\pair{u,u}>0$). Then, by \cite[(5.5)]{Br12}
$$
\phi(u,gz_0,s)=\frac{-2}{\Gamma(\frac{s}{2}-\frac{s_0}{2}+1)} \phi^{(2)}_s(g),\quad g\in \SO(V_{v_0}).
$$
This shows that $\phi(u,z,s)$ and $\phi^{(2)}_s(g)$ determine each other. Here we recall the definition of the secondary spherical function $\phi_s^{(2)}$ in  \cite{OT}. Denote by $H=H_u$ the stabilizer subgroup of $\SO(V_{v_0})$ of $u\in V_{v_0}$, and $K$ the compact subgroup of $\SO(V_{v_0})$ fixing $z_0\in \CD$. Then, by \cite[Prop.~2.4.2]{OT}, there exists a unique family of functions $\phi_s^{(2)}$ for $\Re(s)>s_0$ satisfying the following properties:
\begin{altenumerate}
\renewcommand{\theenumi}{\alph{enumi}}
\item $\phi_s^{(2)}\colon  \SO(V_{v_0})\setminus H K\to\BC$ is smooth, left $ H$-invariant and right $K$-invariant;
\item $\Delta \phi_s^{(2)}= (s^2-s_0^2)\phi_s^{(2)}$; 
\item The function $\phi_s^{(2)}(\exp(t Y_0))-\log\,t$ is bounded as $t\to 0^+$;
\item It has exponential decay near infinity: $\phi_s^{(2)}(\exp(t Y_0))=O(e^{-(\Re(s)+s_0)t})$ as $t\to+\infty$.
\end{altenumerate}
Moreover, the function $\phi_s^{(2)}$ has a meromorphic continuation to $s\in\BC$.  Here $Y_0\in \fkg$ is as in the proof of Lem.~ \ref{lem Lap}.
\end{remark}

\subsection{The automorphic Green function}
We now recall the automorphic Green function  $\CG^{\bf B}(\xi,\phi)$ for each $\xi\in F_{0,+}$, and $\phi\in\CS(V(\BA_{0,f}))^{K}$, cf. \cite{Br12,OT}. Up to adding a constant this is the same as the Green function in \cite[\S7.2]{BHKRY}.

There is a family of smooth functions on $M_K(\BC)\setminus \supp(Z(\xi,\phi))$ with a parameter $s\in\BC$ when $\Re(s)>s_0$,
\begin{align}
\CG_s(\xi,\phi)=\sum\phi(g^{-1}u)\cdot \left( \phi(u,\cdot,s)\times  {\bf 1}_{g \,K} \right), 
\end{align}
where the sum is over $(u,g)\in  V_\xi(F_0)\times G(\BA_{0,f})/K$. Here we recall that $\phi(u,\cdot,s)$ is defined by \eqref{eq:phi u z}.
The sum is absolutely convergent when $\Re(s)>s_0$. Up to a constant, the function $\CG_s(\xi,\phi)$ is characterized by the differential equation \cite[Thm.~5.7]{Br12}
\begin{align}\label{Del GBs}
\Delta \CG_s(\xi, \phi)+\frac{4s_0}{\Gamma(\frac{s}{2}-\frac{s_0}{2}+1)}\delta_{Z(\xi,\phi)}=(s^2-s_0^2) \CG_s(\xi, \phi)
\end{align}
as an equality of currents of degree $0$.  Note that here $\delta_{Z(\xi,\phi)}$ as a current of degree $0$ depends on the choice of a measure in \cite[p.~201]{Br12} (hence it should not be confused with the canonical current of degree $(1,1)$, also denoted by $\delta_{Z(\xi,\phi)}$ below \eqref{Del GB}).

By \cite[Cor.~5.9]{Br12}, the above family $\CG_s(\xi,\phi)$, viewed as distributions, admits a meromorphic continuation to $s\in\BC$ with a simple pole at $s=s_0$ with residue equal to $A(\xi,\phi)=\frac{2}{\vol(M_K)}\deg Z(\xi,\phi)$ (as a locally constant function on $M_K(\BC)$, viewed as a current of degree $0$, cf. \cite[p.~201]{Br12}). Then the automorphic Green function  is defined as the constant term of the Laurent expansion at $s=s_0$, i.e.
\begin{align}\label{Gr Bs}
\CG^\bB(\xi, \phi)=\lim_{s\to s_0}\left(\CG_s(\xi, \phi)-\frac{A(\xi,\phi)}{s-s_0}\right).
\end{align}
Then by \cite[Cor.~5.16]{Br12}, 
\begin{align}\label{ddc GB}
dd^c\,\CG^\bB(\xi, \phi)+ \delta_{Z(\xi,\phi)}\text{ is a harmonic smooth (1,1)-form}.
\end{align}
It follows from \eqref{Del GBs} and \eqref{Gr Bs} that we have an equality of currents
\begin{align}\label{Del GB}
\Delta\,\CG^\bB(\xi, \phi)+4s_0 \delta_{Z(\xi,\phi)}=2s_0A(\xi,\phi).
\end{align}
\begin{remark}
The generating series $c_0+\sum_{\xi\in F_{0,+}}A(\xi,\phi)q^\xi $ for a suitable constant term $c_0$ is essentially a Siegel--Eisenstein series $E(h_\infty, s_0, \phi)$ (evaluated at $s=s_0$) of weight $s_0+1=\dim V/2$, cf. \cite[Thm.~I]{K03} (for $F_0=\BQ$)  and \cite[\S6.2]{Br12}. We will not need this fact.
\end{remark}

\subsection{Gaussian and the function $\Ei$}
Let $z_0\in\CD$ be a fixed base point and let $u=u_++u_-$ be the orthogonal decomposition with respect to $V_{v_0}=z_0^\perp\oplus z_0$. 
Let $\Phi_0$ be the Gaussian function on $V_{v_0}$
\begin{align}\label{eq:Gau infty} 
\Phi_0(u)=e^{-2\pi \fkq(u_+)+2\pi \fkq(u_-) }.
\end{align}
Let
\begin{align}\label{eq:def Phi}
\Phi_0^+(u) :=&(-4\pi \fkq(u_+) +(\dim V/2-1)) e^{-2\pi \fkq(u_+)+2\pi \fkq(u_-) }\\
=&(-4\pi \fkq(u_+) +(\dim V/2-1))\Phi_0(u).\notag
\end{align}
Note that the function $\Phi_0$ (resp. $\Phi_0^+$) has weight $\frac{\dim V}{2}-2$ (resp. $\frac{\dim V}{2}$) under the action of $\wt\SO(2,\BR)\subset \wt\bH(F_{0,v_0})$ by the Weil representation $\omega$.

We recall the exponential integral, defined by
\begin{align}\label{Ei}
\Ei(-r)=-\int^{\infty}_{r} \frac{e^{-t}}{t}dt,\quad r>0.
\end{align}
It has  a logarithmic singularity around $0$, more precisely, when $r\to 0^+$,
$$
\Ei(-r)=\gamma+\log r+ \sum_{n=1}^\infty\frac{(-r)^n}{n\cdot n!},
$$
where $\gamma$ is the Euler constant. We would like to compute the action of the Laplacian on the function $\Ei(-2\pi a R(u,z))$.
\begin{lemma}\label{lem:Del Phi0}
Let $u\in V_{v_0}$ be a non-zero vector. In terms of the Iwasawa decomposition $h=\left(\begin{matrix} 1& b\\
& 1\end{matrix}\right)\left(\begin{matrix} \sqrt{a}& \\
& 1/\sqrt{a}\end{matrix}\right)\in \wt\bH(F_{0,v_0})$, and $\tau=b+ai$, we have
$$
\Delta\left(a^{\dim V/2}\Ei(-2\pi a R(u,z)) e^{2\pi i \tau \fkq( u)}\right)=\omega(h,g)\Phi_0^+(u).
$$
where $g\in G(F_{0,v_0})=\SO(V_{v_0})$ is such that $gz_0=z$ and $\Phi_0^+$ is defined by \eqref{eq:def Phi}.

\end{lemma}
\begin{proof}
We first compute the first and the second derivative of the function  $\phi(r):=\Ei(-r)=-\int^{\infty}_1 \frac{e^{-rt}}{t}dt$ as
$$
\phi'(r)=r^{-1}e^{-r},\quad \phi''(r)=-r^{-2}(1+r)e^{-r}.
$$
Now we apply Lem.~ \ref{lem Lap}. To simplify notation we let $r=2\pi a R(u,z)$:
\begin{align*}
\Delta\,\Ei(-2\pi a R(u,z)) &= 4\left( -r(2\pi a\pair{u,u}+r)r^{-2}(1+r) +(2\pi a\pair{u,u}+(s_0+1)r)r^{-1}\right)e^{-r}
\\&=\left(-r+(s_0-2\pi a\pair{u,u} )\right) e^{-r}
\\&=\left(-2\pi a R(u,z)+(s_0-2\pi a\pair{u,u} )\right) e^{-2\pi a R(u,z)}.
\end{align*}
By $\pair{u,u}=2\fkq(u)=2\fkq(u_z)+2\fkq(u_{z^\perp})=-R(u,z)+2\fkq(u_{z^\perp})$, we obtain
\begin{align*}
\Delta\,\Ei(-2\pi a R(u,z))=\left(-4\pi a \fkq(u_{z^\perp}) +s_0\right) e^{-2\pi a R(u,z)}.
\end{align*}
Combining \eqref{eq:dim V}, \eqref{eq:def Phi} and the formulas for the Weil representation, we complete the proof.
\end{proof}

\subsection{Kudla's Green function}

We now recall Kudla's Green function, defined for the orthogonal case in \cite[\S11]{K} (cf. the unitary case in \cite[\S4B]{Liu1}). 

Let $h_\infty=(h_v)_{v\mid\infty}\in \prod_{v\mid\infty}\wt\bH(F_{0,v})$ and $h_v=\left(\begin{matrix} 1& b_v\\
& 1\end{matrix}\right)\left(\begin{matrix} \sqrt{a_v}& \\
& 1/\sqrt{a_v}\end{matrix}\right)\kappa_{v}$ in the Iwasawa decomposition. For each {\em non-zero} vector $u\in V_{v_0}$, denote by $\mcD_{u}\subseteq \mcD$ the divisor of $z$ with $u\perp z$. Kudla \cite{K} defined a Green function for $\CD_{u}$, parameterized by $h_\infty$ \footnote{Here the constants differ slightly from \cite[\S12]{Z19} which should be corrected as the ones here.}
\begin{align}\label{Gr Ku1}
\CG^{\bf K}(u, h_\infty)(z)=-{\rm Ei}(-2\pi a_{v_0}\, R(u,z)),\quad z\in \CD\setminus \CD_{u}.
\end{align}
which has logarithmic singularity along $\CD_{u}$. Note that this is defined for {\em every} non-zero vector $u\in V_{v_0}$ (in particular,  $u$ may have null-norm). If $\CD_{u}$ is empty, the function is then smooth on $\CD$.
For all non-zero vectors $u$,  we have as currents on $\CD$ by \cite[Prop.~11.1]{K}:
\begin{align}\label{ddc GK D}
dd^c\CG^{\bf K}(u, h_\infty) +\delta_{\CD_{u}}\text{ is  a smooth (1,1)-form}. 
\end{align}
 When $u=0$, we set
\begin{align}\label{Gr Ku m=0}
\CG^{\bf K}(0, h_\infty)=-\log|a_{v_0}|.
\end{align}

Continue to let $\phi\in\CS(V(\BA_{0,f}))^{K}$.
Now we descend the Green function from $\CD_{v_0}$ to the quotient $M_K(\BC)$: For all $\xi \in F_0$, define
\begin{align}\label{Gr Ku2}
\CG^{\bK}(\xi,h_\infty, \phi)=\sum \phi(g^{-1}u)\cdot \left(\CG^{\bf K}(u, h_\infty)\times  {\bf 1}_{g \,K} \right) 
\end{align}
where the sum is over $(u,g)\in V_\xi(F_0)\times G(\BA_{0,f})/K$. This defines a Green function for the divisor $Z(\xi,\phi)$. For all $\xi\in F_0$ we have as currents on $M_{K}(\BC)$
\begin{align}\label{ddc GK}
dd^c\CG^{\bf K}(\xi, h_\infty,\phi) +\delta_{Z(\xi,\phi)}\text{ is  a smooth (1,1)-form},
\end{align}
where $Z(\xi,\phi)$ is understood as zero unless $\xi\in F_{0,+}$.
Define its special value at $h_\infty=1$ 
\begin{align}\label{Gr Ku h=1}
\CG^{\bf K}(\xi,\phi):=\CG^{\bf K}(\xi,1, \phi).
\end{align}

We define the generating series of Kudla's Green functions:
\begin{align}\label{eq:sum GK}
\CG^{\bK}(h, \phi) =\sum_{\xi\in F_0}\CG^{\bK}(\xi,h_\infty, \omega(h_f)\phi) W^{(s_0+1)}_\xi(h_\infty),\quad h\in \wt \bH(\BA_0),
\end{align}
where  the Whittaker function is defined  by \eqref{Whit}. Note that this generating series is understood as a formal sum over $\xi\in F_0$ and we call the terms its Fourier coefficients.

We would like to compute the action of the Laplacian operator on Kudla's Green functions.
Recall that the theta function associated to $\Phi\in \CS(V(\BA_0))$ is defined as
$$
\theta_\Phi(g,h)=\sum_{u\in V(F)}\omega(g,h)\Phi(u),\quad g\in G(\BA_0),\ h\in\wt \bH(\BA_0).
$$
It is well known that the sum is absolutely convergent and defines a smooth function on  $G(\BA_0)\times \wt \bH(\BA_0)$, invariant under $G(F_0)\times \bH(F_0)$. We often omit the argument when evaluating it at $g=1$ or $h=1$.

\begin{lemma}\label{lem:Del i}
Let $i\geq 1$. Then there exist $\phi_{i,\infty}\in \CS(V(F_{0,\infty}))$ of parallel weight $s_0+1$ (with respect to the action of $\prod_{v\mid\infty}\SO(2,\BR)\subset\prod_{v\mid\infty}\wt\bH(F_v)$) such that
$$
\Delta^i \CG^{\bK}(h, \phi)=\begin{cases}\theta_{\phi_{i,\infty}\otimes\phi}(h)- \omega(h)(\phi_{i,\infty}\otimes\phi)(0),& i=1,
\\
\theta_{\phi_{i,\infty}\otimes\phi}(h), & i>1.\end{cases}
$$
Here and below, we will view the Fourier coefficients of $\mcG^{\bK}(h,ϕ)$ as functions on $G(\mbA_0)$ by the map $(g_\infty, g_f)\mapsto [g_{v_0}z_0,g_f]\in M_K(\mbC)$ where the point $z_0$ is also the one used in \eqref{eq:def Phi}. Moreover, the equality is understood as one between formal series, i.e., the $\xi$-th Fourier coefficients of the two sides are equal for every $\xi\in F_0$ and the operator $\Delta^i $ is applied on the left Fourier coefficient wise.
\end{lemma}
\begin{proof}
We will always take $\phi_{i,v}$ the standard Gaussian \eqref{eq:Gau infty} for all archimedean places $v\neq v_0$. (Strictly speaking we are defining the Gaussian on a positive definite quadratic space $V_{v}$ by  \eqref{eq:Gau infty}.) At $v=v_0$ we take 
$$
\phi_{i,v_0}=\Delta^{i-1} \Phi_0^+,\quad i\geq 1,
$$
where $\Phi_0^+$ is the Schwartz function defined by \eqref{eq:def Phi}.  Then they all have weight $s_0+1$.

We start with the case $i=1$.
By Lem.~\ref{lem:Del Phi0} and the definition \eqref{Gr Ku2}, we evaluate the $\xi$-th term of  \eqref{eq:sum GK} at $[z,g_f]\in M_K(\BC)$ where $z=g_\infty z_0\in \CD$ and $g_f\in G(\BA_{0,f})$,
$$
\Delta \CG^{\bK}(\xi,h, \phi)(z,g_f)W^{(s_0+1)}_\xi(h_\infty)=\sum _{u\in V_\xi(F_0)} \omega((g_\infty,g_f), h)(\phi_{1,\infty}\otimes\phi)(u)
$$
when $\xi\neq 0$. When $\xi=0$, we need to remove the term corresponding to the zero vector:
$$
\Delta \CG^{\bK}(0,h, \phi)(z,g_f)W^{(s_0+1)}_\xi(h_\infty)=\sum _{u\in V_{\xi=0}(F_0),u\neq 0} \omega((g_\infty,g_f), h)(\phi_{1,\infty}\otimes\phi)(u).
$$
Note that the term for $u=0$ is independent of $(g_\infty,g_f)$
$$
\omega((g_\infty,g_f), h)(\phi_{1,\infty}\otimes\phi)(0)=\omega(h)(\phi_{1,\infty}\otimes\phi)(0).
$$
This proves the case $i=1$. 

When $i>1$, since the above term for $u=0$ is constant in the variable $(g_\infty,g_f)$, it is annihilated by the operator $\Delta^{i-1}$. The proof is complete.

\end{proof}

\subsection{Comparison of the two Green functions}

We extend the theorem of Ehlen--Sankaran \cite{ES} from $\BQ$ to any totally real field $F_0$. Our argument will be close to \cite[\S4.4]{KRY} where the special case $\dim V=3$ was proved. 

For notational consistence with $\CG^{\bf K}(\xi, h_\infty,\phi)$ defined by \eqref{Gr Ku2}, we define
\begin{align}\label{Gr B h}
\CG^{\bB}(\xi, h_\infty,\phi):=\CG^{\bB}(\xi,\phi),
\end{align} 
cf. \eqref{Gr Bs}.
Next we consider the difference
\begin{align}\label{eq:G K-B}
 \CG^{\bK-\bB}(\xi,h_\infty,\phi):= \begin{cases}\CG^{\bK}(\xi,h_\infty,\phi)-\CG^{\bB}(\xi,h_\infty,\phi),&\xi\in F_{0,+},\\
 \CG^{\bK}(\xi,h_\infty,\phi),& \text{otherwise}
 \end{cases}
\end{align}
as a current on $M_K(\BC)$. Note that $  \CG^{\bK-\bB}(\xi,h_\infty,\phi)$ is a smooth function on the complement of the support of the divisor $Z(\xi,\phi)$.

\begin{lemma}
For every $h_\infty\in \prod_{v\mid\infty}\wt\bH(F_{0,v})$, the difference $  \CG^{\bK-\bB}(\xi,h_\infty,\phi)$ extends to a smooth function on $M_K(\BC)$ (and will be denoted by the same notation).
\end{lemma}

\begin{proof} It suffices to consider the case $\xi\in F_{0,+}$. By \eqref{ddc GB} and \eqref{ddc GK}, both $\CG^{\bB}(\xi,\phi)$ and $\CG^{\bK}(\xi,h_\infty,\phi)$ define Green functions for the same divisor $Z(\xi,\phi)$. Hence
 the $dd^c$ operator sends their difference to a smooth $(1,1)$-form (as currents).  Then the elliptic regularity theorem implies that the current on the compact manifold $M_K(\BC)$ defined by  $  \CG^{\bK-\bB}(\xi,h_\infty,\phi)$ is represented by a smooth function. Since $  \CG^{\bK-\bB}(\xi,h_\infty,\phi)$ is smooth on the complement of the support of the divisor $Z(\xi,\phi)$, it must therefore extend to a smooth function on $M_K(\BC)$.  Alternatively, we may use the local expansion near the singular set $Z(\xi,\phi)$. By \cite[Thm.~ 5.5]{Br12} and the obvious analog for Kudla's Green function, the assertion follows from the fact that the difference $-\Ei(-2\pi a R(u,z))- \phi(u,z,s_0)$ is smooth when $z$ is near $\CD_u$. This reduces to the fact that   
 $ (-\Ei(-\sinh^2t))-(-2\phi^{(s=s_0)}_2(\exp(tY_0)))$ is smooth in a neighborhood of $t=0$. The proof is complete.

\end{proof}

Consider the space $L^2(M_K(\BC))$ (for the fixed volume form) of $L^2$-functions on $M_K(\BC)$ with inner product denoted by $\pair{\cdot,\cdot}$. Then the subspace ${\rm LC}(M_K(\BC))$ of the locally constant functions has dimension equal to the number of connected components of $M_K(\BC)$. For $ f\in L^2(M_K(\BC))$ we denote by $f^\circ$ the orthogonal projection to the orthogonal complement of ${\rm LC}(M_K(\BC))$ in $L^2(M_K(\BC))$.  Note that $\CG^{\bK-\bB}(\xi,h_\infty,\phi)$ is smooth and hence we can define $\CG^{\bK-\bB,\circ}(\xi,h_\infty,\phi)$.

Similar to \eqref{eq:sum GK}, we define the generating series,
\begin{equation}\label{geo error 0}
\CG^{?}(h,ϕ):= \sum_{\xi\in F_0} \CG^{?}(\xi,h_\infty,\omega(h_f)\phi) W^{(s_0+1)}_\xi(h_\infty),
\end{equation}
where $?$ represents $\bB,\bK$ and $\bK-\bB$ respectively. We note that the definition depends on the fixed archimedean place $v_0$ of $F_0$.
Similarly we set
\begin{equation}\label{geo error}
\CG^{\bK-\bB,\circ}(h,ϕ) := \sum_{\xi\in F_0}\CG^{\bK-\bB,\circ}(\xi,h_\infty,\omega(h_f)\phi) W^{(s_0+1)}_\xi(h_\infty).
\end{equation}
Again all of them  are  understood as formal sums over $\xi\in F_0$.

\begin{lemma}\label{lem:Del i K-B}
Let $i\geq 1$ and let $\phi_{i,\infty}\in \CS(V(F_{0,\infty}))$ be as in Lem.~ \ref{lem:Del i}. Then 
$$
\Delta^i \CG^{\bK-\bB}(h, \phi)=\begin{cases}\theta^\circ_{\phi_{i,\infty}\otimes\phi}(h),& i=1,
\\
\theta_{\phi_{i,\infty}\otimes\phi}(h), & i>1,\end{cases}
$$ 
as functions on $G(\mbA_0)$. (The equality is understood as one between formal series.)

\end{lemma}
\begin{proof}
Since $\Delta$ is a self-adjoint operator, the image $\Delta \CG^{\bK-\bB}(\xi,h, \phi)$ must be orthogonal to  ${\rm LC}(M_K(\BC))$. In particular, 
$$
\Delta \CG^{\bK-\bB}(h, \phi)=\Delta \CG^{\bK-\bB,\circ}(h, \phi).
$$
By \eqref{Del GB} and Lem.~  \ref{lem:Del i}
\begin{align*}
\Delta \CG^{\bK-\bB}(h, \phi)=\theta_{\phi_{1,\infty}\otimes\phi}(h)- \omega(h)(\phi_{1,\infty}\otimes\phi)(0)
-2s_0\sum_{\xi\in F_{0,+}} A(\xi,\omega(h_f)\phi)W^{(s_0+1)}_\xi(h_\infty).
\end{align*}
Note that the subtracted terms in the RHS are all locally constant functions on $M_K(\BC)$ and hence 
\begin{align*}
\Delta \CG^{\bK-\bB,\circ}(h, \phi)=\theta^\circ_{\phi_{1,\infty}\otimes\phi}(h).
\end{align*}
The case for $i>1$ follows now easily.
\end{proof}

\begin{theorem}\label{thm ES}Suppose that $\phi\in \CS(V(\BA_{0,f}))^K$  is invariant under $K_{\wt \bH}\subset\wt \bH(\BA_{0,f})$  by  the Weil representation. 
The generating series $ \CG^{\bK-\bB,\circ}(h,\phi)$
lies in the space $\CA_\infty(\wt\bH(\BA_0), K_{\wt\bH},  \dim V/2)$, in the sense that, for every point $[z,g]\in   M_{K}(\BC)$, the value of the generating series at $[z,g]$ is the Fourier expansion of some element in $\CA_\infty(\wt\bH(\BA_0), K_{\wt\bH}, \dim V/2)$. 
\end{theorem}
 
\begin{proof}
The proof is similar to \cite[Thm.~4.4]{KRY} except we need to justify the asserted properties of the function in the $\wt\bH(\BA_0)$-variable. Fix $\phi$ and denote
\begin{align}
\varphi(g,h)= \CG^{\bK-\bB,\circ}(h,\phi) (g_\infty z_0, g_f), \quad g\in G(\BA_0).
\end{align}
The strategy is to show that the image of $\varphi(g,h)$ under the Laplacian operator has the asserted properties and then we use the spectral theory to recover the corresponding properties of $\varphi(g,h)$.

We use the spectral theory of $\Delta$ on the compact manifold $M_K(\BC)=\coprod_{j}\Gamma_j\bs \CD $. It suffices to consider each connected component of $M_K(\BC)$, say of the form $M=\Gamma\bs \CD$. Consider the space $L^2(M)$. Denote the eigenvalues of $-\Delta$  with multiplicities by
$$\lambda_0=0<\lambda_1\leq \lambda_2 \leq\cdots$$ with the eigenfunctions  $$
-\Delta\varphi_i=\lambda_i \varphi_i,
$$which are smooth and normalized by $\pair{\varphi_i,\varphi_i}=1$. Here $\varphi_0$ is the constant function $\frac{1}{\vol(M)}$.

We {\em claim} that for an integer $N\geq 3(\dim_\BR M /2+1)$, the kernel function
  \begin{align}\label{eq:KN}
K_{N}(x,y):=\sum_{i>0}\frac{1}{\lambda_i^N}\varphi_i(x)\ov\varphi_i(y),\quad x,y\in M
 \end{align} 
is a continuous function on $M\times M$.
 In fact, by \cite[Lem.~5.11\,(2)]{Br12} we have  for any integer $N_0> \dim_\BR M /2$, there  exists a constant $C>0$ such that
  \begin{align*}
|\!| \varphi_i |\!|_{L^\infty}&\leq C \lambda_i^{N_0}
 \end{align*} 
 holds for all $i\geq1$. It follows that when $N\geq 3N_0$,
\begin{eqnarray*}
\sum_{i>0}\frac{1}{\lambda_i^N}|\varphi_i(x)\ov\varphi_i(y)|&\leq& \sum_{i>0} \frac{1}{\lambda_i^N} |\!| \varphi_i |\!|_{L^\infty}^2\\
&\leq& C^2   \sum_{i>0} \frac{1}{\lambda_i^N} \lambda_i^{2N_0} \\
&\leq &C^2  \sum_{i>0} \frac{1}{\lambda_i^{N_0}} \\
&<&\infty \quad\quad\quad\quad \mbox{(\cite[Lem.~5.11\,(1)]{Br12})}.
\end{eqnarray*}
 Therefore the sum in \eqref{eq:KN} converges absolutely and uniformly in $(x,y)\in M\times M$.
 
Let $\Phi_N =\phi_{N,\infty}\otimes \phi\in\CS(V(\BA_0)),\, N\geq 1$ where $\phi_{N,\infty}\in \CS(V(F_{0,\infty}))$ is  as in Lem.~ \ref{lem:Del i}. Note that $\phi_{1,\infty} = Φ_0^+$ from \eqref{eq:def Phi}. Then
we have a spectral expansion
\begin{align}
\label{eq:theta lift}
\theta^\circ_{\Phi_1}(g,h)=\sum_{i>0 } \varphi_i(gz_0) \theta_{\Phi_1}^{\varphi_i}(h)
\end{align}
where $\theta_{\Phi_1}^{\varphi_i}$ is the ``theta lifting" of $\varphi_i$, defined by
 $$
  \theta_{\Phi_1}^{\varphi_i}(h):=\pair{\theta_{\Phi_1}(\cdot,h),\varphi_i}.
 $$
The theta function
$\theta_{\Phi_j}(\cdot,\cdot)$  is smooth on $ G(\BA_0)\times \wt\bH(\BA_0)$, left $G(F_0)\times \bH(F_0)$-invariant with parallel weight $\dim V/2$ under $\prod_{v|\infty}\wt\SO(2,\BR)\subset\prod_{v|\infty} \wt\bH(F_{0,v})$.

By Lem.~ \ref{lem:Del i K-B}, the function $\varphi(\cdot,h)$ is the unique solution to $\Delta \varphi=\theta_{\Phi}(g,h)$ and $\pair{\varphi,\varphi_0}=0$. We have
$$
\varphi(g,h)=\sum_{i>0}\pair{\varphi(\cdot,h), \varphi_i}\varphi_i(g),
$$where by \cite[Lem.~5.11(4)]{Br12} the sum is absolutely convergent and uniformly in $g$ by the smoothness of $\varphi(\cdot,h)$ (for every given $h$).  The coefficients are 
\begin{eqnarray*}
\pair{\varphi(\cdot,h), \varphi_i}&=&-\frac{1}{\lambda_i}\pair{\varphi(\cdot,h), \Delta\varphi_i}\\
&=&-\frac{1}{\lambda_i} \pair{ \Delta\varphi(\cdot,h),\varphi_i} \quad \mbox{(Self-adjointness of $\Delta$)} \\
&=&-\frac{1}{\lambda_i}\pair{\theta_{\Phi_1}(\cdot,h),\varphi_i} \quad \mbox{(Lem.~ \ref{lem:Del i K-B})}\\
&=&-\frac{1}{\lambda_i} \theta_{\Phi_1}^{\varphi_i}(h).
\end{eqnarray*}
It follows that 
$$
\varphi(g,h)=-\sum_{i>0}\frac{1}{\lambda_i} \theta_{\Phi_1}^{\varphi_i}(h)\varphi_i(g).
$$

Using $\Delta^N \varphi(g, h)=\theta_{\Phi_N}(g,h)$, the same argument shows for all $N\geq 1$,
 \begin{align*}
 \varphi(g,h)=-\sum_{i>0}\frac{1}{\lambda_i^N} \theta_{\Phi_N}^{\varphi_i}(h)\varphi_i(g).
 \end{align*}
When $N$ is large enough, by the continuity of the kernel function $K_N$, we have
 $$
 \varphi(g,h)= \pair{\Delta^N \varphi(\cdot, h), K_N(\cdot, g)}.
 $$ It follows from this identity and  the properties of the theta function $\theta_{\Phi_N}$ stated above that the function  $ \varphi(g,\cdot)$ is smooth on $\wt \bH(\BA_0)$,  left $\bH(F_0)$-invariant with weight $\dim V/2$.
 \end{proof}

 \begin{remark}The projection of the generating series $\CG^{\bK}(h,\phi)$ and $\CG^{\bB}(h,\phi)$ in \eqref{geo error 0} to the finite dimensional space  ${\rm LC}(M_K(\BC))$ of locally constant functions is essentially understood by Garcia--Sankaran \cite[Thm.~5.10]{GaS} and Bruinier \cite{Br12} respectively. We will not need these results in this paper. \end{remark}
   \begin{remark}In the case $F_0=\BQ$, the theorem of Ehlen and Sankaran shows that the function  $\CG^{\bK-\bB}(h,\phi)$  has at worst ``exponential growth at $\infty$" \cite[Def.\ 2.8, pp.2104]{ES}.
 \end{remark}
 \begin{corollary}\label{coro ES O}Let $\phi$ be as above.
 The generating function $\CG^{\bK-\bB}(h,\phi)$ when evaluating at  degree-zero (on every connected component) zero-cycle, lies in $\CA_\infty(\wt\bH(\BA_0), K_{\wt\bH}, \dim V/2)$.  \end{corollary}

\subsection{Back to hermitian spaces}
 We now return to the set up in \S\ref{s:SV KR}. In particular $V$ is an $F/F_0$-hermitian vector space of $F$-dimension $n$. 
 Then the two Green functions are defined similarly to the orthogonal case, cf. \cite[\S8.3]{Z19} for details. Let $\nu$ be a place of $\bF$ and 
let $v_0$ be the unique place of $F_0$ below $\nu$. Define $ \CG^{\bK-\bB}(h,\phi)$ similarly to \eqref{geo error 0}, as a function on $M_{\wt K}\otimes_{\bF,\nu}\BC$. Note that now the Weil representation of $ \wt\bH(\BA_{0})$ factors through  $ \bH(\BA_{0})$.
  \begin{theorem}\label{thm:dif inf}Suppose that $\phi\in \CS(V(\BA_{0,f}))^K$  is invariant under $K_{\bH}\subset \bH(\BA_{0,f})$  by  the Weil representation.  Then the generating function $ \CG^{\bK-\bB}(h,\phi)$  
when evaluating at  degree-zero (on every connected component) zero-cycle, lies in $\CA_\infty(\bH(\BA_0), K_\bH, n)$.
 \end{theorem}
 \begin{proof}This is standard by the pull-back of the assertion Cor.~ \ref{coro ES O} for the induced quadratic space $V$, cf. \cite[\S3.B]{Liu1}.
 \end{proof}

\section{Almost Modularity}\label{s:alm mod}

In this section, we describe the simple arithmetic intersection theory we use in the paper. It is modeled on \cite{BGS} and takes into account the fact that our integral models are only defined away from finitely many primes. Let $S$ be a finite set of non-archimedean primes of a number field $\bfF$ and $\CX\to\Spec O_\bfF[S^{-1}]$ a regular, flat and proper scheme with smooth generic fiber $X$.

\subsection{Admissible extension}\label{ss:adm ext}

Recall that the arithmetic Chow group $\wh\Ch{}^1(\CX)$ (with $\BQ$-coefficients) is generated by arithmetic divisors, i.e. $\BQ$-linear combinations of tuples $\left(Z, g_Z\right)$, where $Z$ is a divisor on $\CX$ and $g_Z = (g_{Z,v})_{v\colon \bfF\to \ov \mbQ}$ a tuple of Green functions for the divisors $Z_{v}(\BC)$ on the complex manifold $X_{v}(\BC)$ with respect to $v\colon \bfF\incl\ov \BQ\subset\BC$, cf.\ \cite[\S3.3]{GS}. 
Its relations are given by the $\mbQ$-span of principal arithmetic divisors, i.e. tuples associated to rational functions $f\in \bfF(\CX)^\times$:
\begin{equation}\label{eq convention rat fct}
\left(\div(f),(- \log |f|^2_{v})_{ v\in\Hom_\BQ(\bfF,\ov\BQ)}\right).
\end{equation}
If for example $\bfF=\BQ$, then $(\mcX_p,0)=(0,2\log|p|)$ in $\wh\Ch{}^1(\CX)$, where $\mcX_p$ is the fiber of $\CX$ over a prime $p$. In case $p\in S$ and hence $\mcX_p = \emptyset$, this means $(0,2\log |p|) = 0$ in $\wh\Ch{}^1(\mcX)$.

Fix a K\"ahler metric on each $X_v(\mbC)$. It induces a Laplace operator on differential forms on $X_v(\BC)$.
We call a Green function $g_{Z,v}$ {\em admissible} (with respect to the chosen metrics) if
\begin{equation}\label{eq:admissible_PL}
dd^c\, g_{Z,v} +\delta_Z =\omega_{Z,v},
\end{equation}
where $\omega_{Z,v}$ is a {\em harmonic} $(1,1)$-form. An admissible Green function exists for every divisor and is unique up to adding a locally constant function, i.e. a function that is constant on each connected component of $X_v(\BC)$. We denote by $\wh\Ch{}^{1,\adm}(\CX)$ the subgroup of $\wh\Ch{}^1(\CX)$ generated by $(Z,g_Z)$ with all $g_{Z,v}$ admissible. Note that principal arithmetic divisors are admissible since for these $ω_{Z,v} = 0$ in \eqref{eq:admissible_PL}. So an arithmetic divisor defines an admissible Chow class if and only if it is itself admissible.

Consider the subgroup of $\wh\Ch{}^{1,\adm}(\CX)$ spanned by the classes of $(0, c)$ with $c = (c_v)_v$ a tuple of locally constant functions on the $X_v(\BC)$: 
$$
\wh\Ch{}^1_\infty(\CX):=\mr{span}\, \{(0, c)\mid c \text{ locally constant}\} \subset \wh \Ch{}^{1,\adm}(\CX).
$$
These are the arithmetic divisors ``supported at archimedean places".
Then there is a natural map
\begin{equation}\label{eq:iso Ch}
\xymatrix{\frac{\wh \Ch{}^{1,\adm}(\CX)}{\wh\Ch{}^1_\infty(\CX) }\ar[r] &\Ch^{1}(X),}
\end{equation}
sending  (the equivalence class of) $(Z, g_Z)$ to the generic fiber of $Z$.
\begin{lemma}  \label{lem:Ch iso}
Assume further that $\CX$ is smooth over $\Spec O_\bfF[S^{-1}]$. Then the map \eqref{eq:iso Ch} is an isomorphism.
\end{lemma}

\begin{proof}
The map $\wh\Ch{}^{1,\adm}(\mcX)\to \Ch^1(X)$ is surjective; its kernel is generated by all vertical classes. Those are the elements of $\wh\Ch{}^1_{\infty}(\mcX)$ and the vertical classes at the non-archimedean primes, i.e. classes of the form $(Z,0)$ with $Z\subseteq \mcX_v$ for some non-archimedean place $v\notin S$. For dimension reasons, any such $Z$ is a $\mbQ$-linear combination of irreducible components of the fiber $\mcX_v$. Since $\mcX$ is assumed to be smooth, those are actually the connected components of $\mcX_v$.

The Stein factorization $\mcX\to \Spec R \to  \Spec O_\bfF[S^{-1}]$ has the property that $R = \mcO_\mcX(\mcX)$ agrees with the normalization of $O_\bfF[S^{-1}]$ in $\bfF(\mcX)$, cf. \cite[Tag 03H0]{Stacks}. The connected components of $\mcX_v$ are then in bijection with $|\Spec R/\mfp_v|$. By finiteness of the class number of $R$, each is equivalent to an element of $\wh\Ch{}^1_\infty(\mcX)$ as was to be shown.
\end{proof}

\subsection{$S$-punctured arithmetic intersection pairing}
\label{ss:AIP}
We come to the arithmetic intersection pairing, cf. \cite[\S2.3]{BGS}. Assume in the following that $\CX$ is smooth over $O_\bfF[S^{-1}]$; write
$$
\BR_S:=\BR/{\rm Span}_\BQ\{\log |q_v| \mid v \in S\}.
$$
This quotient is viewed as a $\BQ$-vector space. 

\begin{definition}\label{def Z1} 
Let $\wt\CZ_{1}(\CX)$ be the group of $1$-cycles on $\CX$. Define $ \CZ_{1}(\CX) $  as the quotient of $\wt \CZ_{1}(\CX)$ by the subgroup generated by $1$-cycles that are rationally trivial on a closed fiber of $\CX$.  The latter means $\div(f) = 0$ whenever $Y\subseteq \mcX_v$ is a $2$-dimensional integral closed subscheme in a special fiber and $f\in \mbF_v(Y)^\times$ a meromorphic function.
\end{definition}

Then we have an arithmetic intersection pairing, adapting the definition in  \cite[\S2.3]{BGS} to our $S$-punctured scheme $\CX$,
\begin{align}
 \begin{gathered}
	\xymatrix@R=0ex{
	 \wh\Ch{}^1(\CX)\times\wt \CZ_{1}(\CX)  \ar[r] & \BR_S,
	}
	\end{gathered}
\end{align}
which factors through a pairing
 \begin{align}\label{eq:AI} 
 \begin{gathered}
	\xymatrix@R=0ex{
(\cdot,\cdot)\colon \quad\wh\Ch{}^1(\CX)\times \CZ_{1}(\CX)   \ar[r] & \BR_S
	}
	\end{gathered}
\end{align}
by \cite[Prop.\ 2.3.1 (ii)]{BGS}. We call it the \emph{$S$-punctured arithmetic intersection pairing} and next give its definition in terms of properly intersecting representatives. Assume that $Z\subseteq \mcX$ is a closed subscheme of codimension $1$, viewed as a divisor, and $C\subseteq \mcX$ a closed curve such that $Z\cap C$ is artinian. Let $g_Z = (g_{Z,v})_{v\colon \bfF \to \bar \mbQ}$ be a family of Green functions for $Z$. Then
\begin{equation}\label{eq:example_pairing}
\big((Z,g_Z), C\big) = \log |\mcO_{Z\cap C}| + \frac{1}{2}\sum_{v\colon \bfF\to \bar \mbQ} g_v(C(\mbC_v)).
\end{equation}
Extending this relation $\mbQ$-linearly in both arguments determines \eqref{eq:AI} uniquely.
Let $\CZ_1(\CX)_{\deg=0}$ be the subgroup of $\CZ_1(\CX)$ consisting of $1$-cycles with vanishing degree on each  connected component of the generic fiber of $\CX$. Then there is an induced $\BQ$-bilinear pairing
$$
\xymatrix{(\cdot\,,\cdot)\colon\frac{\wh \Ch{}^1(\CX)}{\wh\Ch{}^1_\infty(\CX) }\times \CZ_1(\CX)_{\deg=0}  \ar[r] &\BR_S,}
$$
which we may restrict to the admissible classes
\begin{align}
\xymatrix{(\cdot\,,\cdot)\colon\frac{\wh \Ch{}^{1,\adm}(\CX)}{\wh\Ch{}^1_\infty(\CX) }\times \CZ_1(\CX)_{\deg=0}  \ar[r] &\BR_S.}
\end{align}
By Lem.~ \ref{lem:Ch iso}, we obtain a $\mbQ$-bilinear pairing
\begin{align}\label{eq:adm int}
\xymatrix{(\cdot\,,\cdot)^{\adm}\colon\Ch^1(X)\times \CZ_1(\CX)_{\deg=0}  \ar[r] &\BR_S.}
\end{align}
Note again that the above are all pairings between $\BQ$-vector spaces only.

\subsection{Application: An almost modularity}
We now return to the case of our Shimura variety $M=M_{\wt K}$ and the integral model $\mcM = \CM_{\wt K}$ over $\Spec O_{\bfF}[\mfd^{-1}]$ defined in \S\ref{s:SV KR}. We endow the  automorphic line bundle $\bm{\omega}$ with its Petersson metrics $|\!|\cdot|\!|_{\rm Pet}$ at the archimedean places and use the induced  Kähler metric to define the notion of admissible extension in 
 \S\ref{ss:adm ext}.  Recall that we defined a divisor $Z(h,\phi)\in \Ch{}^1(M)$ for each $\phi\in \mcS(V(\BA_{0,f}))^{K}$ and $h\in \bH(\BA_0)$ in \S\ref{s:gen div}.  Then we can apply the pairing \eqref{eq:adm int}  to obtain ``an almost modularity" result:
 \begin{theorem}\label{thm mod}  Suppose that $\phi\in \mcS(V(\BA_{0,f}))^{K}$ is $ K_\bH$-invariant under the Weil representation.
Let $z\in \mcZ_1(\CM)_{\deg=0} $. Then
the generating function $h\in\bH(\BA_0)\mapsto\left(Z(h,\phi),  z\right )^\adm$
 lies in $\CA_{\rm hol}(\bH(\BA_0), K_\bH, n)_{\ov\BQ}\otimes_{\ov \BQ}\BR_{S,\ov\BQ}$, where $\BR_{S,\ov\BQ}:=\BR_{S}\otimes_{\BQ}\ov\BQ$. 
 \end{theorem}
 \begin{proof}This follows from Thm.~\ref{thm:mod E}.
 \end{proof}
To apply the result, we still need to find a convenient lifting of $Z(\xi,\phi)$ to $\wh \Ch{}^{1,\adm}(\CM)$. We can do so whenever $ϕ$ factors as $ϕ = ϕ_\mfd \otimes 1_{\wh{Λ}^\mfd}$ as in \eqref{eq:defn_int_model_KR_div}, which we assume from now on. Then, for $\xi\in F_{0,+}$, we endow the integral KR divisor $\mcZ(ξ,ϕ)$ with the automorphic Green function $\CG^{\bB}(\xi,\phi)$ to define an arithmetic KR divisor 
\begin{equation}\label{eq:KR  Ar}
 \wh\mcZ^\bB(ξ,ϕ)=\left(\CZ(ξ,ϕ),\left( \CG^{\bB}_{\nu}(\xi,\phi) \right)_{\nu}\right )\in\wh\Ch{}^{1} (\CM).
\end{equation}
 Since the automorphic  Green function is admissible by \eqref{ddc GB}, we even have $\wh\mcZ^\bB(ξ,ϕ)\in\wh\Ch{}^{1,\adm} (\CM)$. For $\xi=0$, we set
\begin{equation}\label{eq:KR  Ar 0}
\wh\CZ^\bB(0,\phi)=-\phi(0) \,c_1(\wh{\bm{\omega}})\in \wh\Ch{}^{1,\adm}(\CM),
\end{equation}
where $\wh{\bm{\omega}}=(\bm{\omega}_\CM,|\!|\cdot|\!|_{\rm Pet})$ is {\em any} extension of the automorphic line bundle $\bm{\omega}$ to a line bundle $\bm{\omega}_\CM$ on the integral model $\CM$, endowed with its Petersson metrics.

Finally, we define a family of arithmetic divisors with the Kudla's Green function:  for $\xi\in F_{0}^\times$,
\begin{align}\label{eq:Z Ku}
 \wh\mcZ^\bK(ξ,h_\infty,ϕ):=\left(\CZ(ξ,ϕ), \left( \CG^{\bK}_\nu(\xi,h_ {\nu|_{F_0}},\phi) \right)_{\nu} \right)\in\wh\Ch{}^{1} (\CM),
\end{align}
and for $\xi=0$
\begin{align}\label{eq:Z Ku 0}
\wh{\mcZ}^\bK(0, h_\infty, ϕ)=-\phi(0) \,c_1(\wh{\bm{\omega}})+\left(0, \left( \CG^{\bK}_\nu(0,h_ {\nu|_{F_0}},\phi)  \right)_{\nu} \right ),
\end{align}
 where $h_\infty=(h_v)_{v|\infty}\in \bH(F_{0,\BR})$; cf. \eqref{Gr Ku2}. The special value at $h_\infty=1$ is denoted by
 \begin{align}\label{eq:Z Ku h=1}
 \wh\mcZ^\bK(ξ,ϕ):=\left(\CZ(ξ,ϕ), \left( \CG^{\bK}_\nu(\xi,\phi) \right)_{\nu} \right)\in\wh\Ch{}^{1} (\CM), \quad \xi\neq 0,
\end{align}
cf.  \eqref{Gr Ku h=1}.
\part{Intersection with CM cycles}

\section{CM cycles}
\subsection{Derived CM cycles}\label{ss:LCM}
This section recapitulates the definition of the derived CM cycle in \cite[\S7]{Z19}. Recall from \S\ref{s:SV KR} that $\mcM=\CM_{\wt K}\to \Spec O_\bfF[\mfd^{-1}]$ denotes the integral model of the Shimura variety for group $\wt G$ and level $\wt K$. Let $F[t]^\circ_{\deg n}$ denote the monic polynomials $α$ of degree $n$ such that $α(0)t^n\ov{α}(t^{-1}) = α(t)$. For example, the $F$-linear characteristic polynomial of an element of $G(F_0) = \U(V)(F_0)$ lies in $F[t]^\circ_{\deg n}$.

Assume that $(A_0,A,\ov{η})\in \mcM(S)$ with $S$ connected and that $φ\in \End^\circ_F(A)$. Pick any point $s\in S$ and a prime $\ell$ that is invertible in $s$. Then $φ$ acts $F$-linearly on the $\ell$-adic rational Tate-module $V_\ell(A)$ and its $\mbQ_{\ell}\otimes_\mbQ F$-linear characteristic polynomial is known to have coefficients in $F$. It is moreover independent of the choice of $s$ and $\ell$ and we denote it by $\mr{char}(φ)$.

\begin{definition}
Let $α\in F[t]^\circ_{\deg n}$ and $µ\in K_\mfd\backslash G(F_{0,\mfd})/K_\mfd$. The \emph{big fat CM cycle} $\mcC(α,µ)\to \mcM$ is the scheme representing the functor taking $S/O_\bfF[\mfd^{-1}]$ to the set of isomorphism classes in
$$\left\{(A_0,A,\ov{η},φ) \left\vert
\begin{array}{c}\text{$(A_0,A,\ov{η})\in \mcM(S)$ and $φ\in \End_{O_F}(A)[\mfd^{-1}]$}\\
\text{s.th. $φ^*λ_A = λ_A$, $\mr{char}(φ) = α$ and $ηφη^{-1} \in µ$ for $η\in \ov{η}$}\end{array}\right\}\right..$$ 
\end{definition}

An auxiliary role is played by the Hecke correspondence $\mr{Hk}_µ$ defined by $µ$. It is, by definition, the scheme representing the functor taking $S/O_\bfF[\mfd^{-1}]$ to the set of isomorphism classes in
$$\left\{(A_0,A,\ov{η}_A,B,\ov{η}_B,φ) \left\vert
\begin{array}{c}\text{$(A_0,A,\ov{η}_A),\,(A_0,B,\ov{η}_B)\in \mcM(S)$ and}\\
φ\in \Hom_{O_F}(A,B)[\mfd^{-1}]\text{ s.th.}\\
\text{$φ^*λ_B = λ_A$ and $η_Bφη_A^{-1} \in µ$ for $η_A\in \ov{η}_A$, $η_B\in \ov{η}_B$}\end{array}\right\}\right..$$ 
The natural map
\begin{equation}
\label{eq Hecke mu}
\mr{Hk}_µ\to \mcM\times_{\Spec O_\bfF[\mfd^{-1}]}\mcM
\end{equation}
is then finite. It is immediate from the moduli descriptions that there is a decomposition
\begin{equation}\label{eq:Hecke_diag_intersection}
\xymatrix{\mr{Hk}_µ \times_{\mcM\times_{\Spec O_\bfF[\mfd^{-1}]} \mcM} Δ_{\mcM} = \coprod_{α\in F[t]^\circ_{\deg n}} \mcC(α,µ).}
\end{equation}
As \eqref{eq Hecke mu} is a finite map, only finitely many of the $\mcC(α,µ)$ are non-empty. Moreover, one sees that each map $\mcC(α,µ)\to \mcM$ is finite and unramified, which is \cite[Prop.~ 7.9]{Z19}. From $F_0\neq \mbQ$ it follows that $\mcC(α,µ)\to \Spec  O_\bfF[\mfd^{-1}]$ is proper.

We assume from now on that $α$ is irreducible. Then $E := F[t]/(α(t))$ is a degree $n$ field extension of $F$. It comes with a distinguished root of $α$, namely the image of $t$. The Galois involution of $F/F_0$ extends to $E$ by letting $t\mapsto t^{-1}$ and we denote by $E_0$ its fixed field. Whenever $(A_0,A,\ov{η},φ)\in \mcC(α,µ)$ one gets a \emph{$\ast$-embedding}
$$E\to \End^\circ(A),\ t\mapsto φ,$$
meaning an embedding that is equivariant for the Galois involution on the left and the Rosati involution on the right. It follows from the classification of endomorphism rings of abelian varieties (cf. \cite{Mum-AV}) that $\mcC(α,µ)\neq \emptyset$ only for $α$ such that $E_0$ is totally real and (then necessarily) $E/E_0$ a quadratic CM extension. We assume this to be the case and make the further assumption that $α$ is the characteristic polynomial of some element of $G(F_0)$.

The generic fiber $C(α,µ)$ of $\mcC(α,µ)$ is then $0$-dimensional as it consists of points with CM by $E$, but the integral structure may still be complicated, depending on the $O_F$-order generated by $α$. We set
$$\LCM(α,µ)' := [\mcO_{\mr{Hk}_µ}\overset{\mbL}{\otimes} \mcO_{Δ_\mcM}]\vert_{\mcC(α,µ)} \in K'_0(\mcC(α,µ)).$$
Here $K'_0(\mcC(α,µ))$ denotes the $K$-group of coherent sheaves with $\mbQ$-coefficients on the scheme $\mcC(α,μ)$. Concretely, $\LCM(α,μ)'$ is defined as follows. The cohomology sheaves $\mcE^i = H^i(\mcO_{\mr{Hk}_µ}\overset{\mbL}{\otimes} \mcO_{Δ_\mcM})$ have support on the intersection \eqref{eq:Hecke_diag_intersection}. Let $\mcE^i\vert_{\mcC(α,µ)} \subseteq \mcE^i$ denote the direct summand supported on $\mcC(α,µ)$. Then
$$\LCM(α,µ)' = \sum_{i\in \mbZ} (-1)^i [\mcE^i\vert_{\mcC(α,µ)}]$$
where $[\mcE]$ denotes the class of $\mcE$ in $K'_0(\mcC(α,µ))$. The non-trivial statement now is that $\LCM(α,µ)'$ lies in the filtration step generated by sheaves supported on $1$-dimensional subschemes, denoted as $F_1K'_0(\mcC(α,µ))$ in \cite[Appendix B]{Z19}. Given a coherent sheaf $\mcF$ on $\mcC(α,µ)$ with $\leq 1$-dimensional support $C$, one constructs a $1$-cycle as follows. Let $η_1,\ldots,η_r$ be the generic points of the $1$-dimensional irreducible components of $C$. Let $C_i$ denote the closure of $η_i$ with reduced scheme structure. Then take
$$\sum_{i = 1}^r \mr{len}_{\mcO_{C,η_i}} \mcF_{η_i} \cdot [C_i] \in \wt{\mcZ}_1(\mcM).$$
Since $\mcC(α,µ)$ is itself $1$-dimensional away from a finite set of primes, this descends to the $K$-group as a well-defined map
$$F_1K'_0(\mcC(α,µ))\to \mcZ_1(\mcM).$$
We define $\LCM(α,μ)$ as the image of $\LCM(α,μ)'$ in $\mcZ_1(\mcM)$.
This definition easily extends to a weighted version, compare \cite[(7.16)]{Z19}.
\begin{definition}\label{def:cm_cycle}
For $\varphi \in \mcS(K_\mfd\backslash G(F_{0,\mfd})/K_\mfd, \mbQ)$ a Hecke function at $\mfd$, we set
$$\LCM(α,\varphi) := \sum_{µ\in K_\mfd\backslash G(F_{0,\mfd})/K_\mfd} \varphi(µ)\, \LCM(α,µ) \in \mcZ_1(\mcM).$$
\end{definition}

\subsection{Elementary CM cycles}
We fix $α$ as in the previous section, in particular we assume that $E$ is a CM field. We write $F^1 = \{x\in \mr{Res}_{F/\mbQ} \mbG_m \mid N_{F/F_0}(x) = 1\}$ and $F^\mbQ := Z^\mbQ$ in this section. We similarly define the two tori $E^1 = \{x\in \mr{Res}_{E/\mbQ} \mbG_m \mid N_{E/E_0}(x) = 1\}$ and $E^\mbQ = \{x\in \mr{Res}_{E/\mbQ} \mbG_m \mid N_{E/E_0}(x) \in \mbG_m\}$. Consider the product torus
$$T:= F^\mbQ\times_{\mbG_m} E^\mbQ \isoarrow F^\mbQ \times E^1.$$
Given a $*$-embedding $σ\colon E\to  \End(V)$,
$$σ(E^\mbQ) = σ(E)\cap G^\mbQ.$$
Hence any such $σ$ induces an embedding $σ\colon T\to \wt G$ by taking the identity on the $F^\mbQ$-factors. Moreover, $σ$ makes $V$ into a one-dimensional $E/E_0$-hermitian space by lifting the hermitian form along the trace $\tr_{E/F}$. It is of signature $(0,1)$ at precisely one place $φ'_0\colon E_0\to \mbR$ which has to lie above $φ_0$. Let $Φ'$ denote the inverse image of $Φ$ under the restriction map $\Hom(E,\mbC)\to \Hom(F,\mbC)$, which is a CM type for $E$. Define a Shimura datum $h_T$ for $T$ where the $F^\mbQ$-component is \eqref{eq Sh dat Z} and where the $E^1$-component is given by
\begin{equation}\label{eq h_T}
(h_{T,φ'})_{φ'\in Φ'}\colon \Res_{\mbC/\mbR} \mbG_m \to  E^1_\mbR \isoarrow \prod_{φ'\in Φ'} \mbC^1
\end{equation}
with
$$h_{T,φ'}(z) = \begin{cases}
1 & \text{if $φ' \neq φ'_0$}\\
\ov{z}/z & \text{if $φ' = φ'_0$}.\end{cases}$$
Then $σ$ becomes a map of Shimura data, i.e. $σ\circ h_T \in \{\wt h\}$, and we denote by $\mathbf{E}$ the reflex field of $h_T$. It contains $\bfF$, as well as $E$ via $φ'_0$.
\begin{definition}\label{def element CM cycle}
Let $σ\colon E\to \End(V)$ be a $*$-embedding, $g\in G(\mbA_{0,f})$ any element and denote by $C(σ,g)$ the discrete Shimura variety (over $\mathbf{E}$)
$$C(σ,g):=[T(\mbQ)\backslash T(\mbA_f) / σ^{-1}(g\wt Kg^{-1})].$$
It naturally maps to the Shimura variety $\mathbf{E}\otimes_{\mathbf{F}}M_{g\wt Kg^{-1}}$. Applying the Hecke isomorphism $g\wt K\colon M_{g\wt Kg^{-1}} \to  M_{\wt K}$ defines the so-called \emph{elementary CM cycle}
$$C(σ,g) \to  \bfE\otimes_\bfF M_{\wt K}.$$
\end{definition}
Let $\mcC(σ,g)\to \Spec  O_\bfE[\mfd^{-1}]$ denote the integral model obtained from the normal closure of $ O_\bfE$ in the generic fiber $C(σ,g)$. It defines an element of $\mcZ_1( O_\bfE \otimes_{ O_\bfF}\mcM)$, cf. Def. \ref{def Z1}.

\subsection{Connected components}
The group $\wt G = F^{\mbQ}\times_{\mbG_m} G^{\mbQ}(V)$ has center $Z = F^{\mbQ}\times_{\mbG_m} F^{\mbQ}$. Using the isomorphism
$$\begin{array}{rcl}
\wt G & \isoarrow & F^\mbQ\times G\\
(z,g) & \mapsto & (z,z^{-1}g)
\end{array}$$
and that the derived group of $G$ is $\SU(V)$, one sees that $Z$ is also the maximal abelian quotient of $\wt G$,
$$\begin{array}{rcl}
π\colon \wt G^{\mr{ab}} & \isoarrow & Z\\
(z,g) & \mapsto & (z, z^{1-n} \det(g)).
\end{array}$$
\begin{proposition}\label{prop conn comp of M}
Taking $\mbC$-points with respect to the inclusion $\bfF\subseteq \mbC$, the map
$$\begin{aligned}
M_{\wt K}(\mbC) = \wt G(\mbQ)\backslash X\times \wt G(\mbA_f)/\wt K &\to Z(\mbQ)\backslash  Z(\mbA_f)/π(\wt K)\\
[x,h] &\mapsto π(h).
\end{aligned}$$
induces a bijection on connected components.
\end{proposition}
\begin{proof}
This is a direct application of \cite[Thm.~ 5.17]{Mil-Shimura}. The derived group $\wt G^\mr{der}$ equals $\SU(V)$, which is a form of $\SL_n$ and consequently simply connected. It follows that one may apply the cited result to obtain
$$π_0(M_K(\mbC)) \iso Z(\mbQ)^\dagger \backslash Z(\mbA_f)/π(\wt K),$$
where $Z(\mbQ)^\dagger := Z(\mbQ)\cap \mr{Im}(π\vert_Z\colon  Z(\mbR)\to Z(\mbR))$. In the case at hand, the real points $Z(\mbR)$ are connected and hence the map $π\vert_{Z(\mbR)}$ is surjective. It follows that $Z(\mbQ)^\dagger = Z(\mbQ)$ which concludes the proof.
\end{proof}

The degree of an elementary CM cycle $C(σ,g)$, viewed as a function on $π_0(M_{\wt K}(\mbC))$, may be understood through Prop.~ \ref{prop conn comp of M} as well. Namely, independently of $σ$, the composition $π\circ σ$ equals
\begin{equation}\label{eq:defn_nu}
ν := (\id, N_{E/F})\colon F^\mbQ\times E^1\to F^\mbQ\times F^1
\end{equation}
where we have used the identification $Z\iso F^\mbQ\times F^1,\ (z_1,z_2)\mapsto (z_1, z_1^{-1}z_2)$ for the target. The degree of $C(σ,g)$ is then a scalar multiple of the characteristic function $1_{π(g)\mr{Im}(ν(\mbA_f))π(\wt K)}$. We now transfer this observation to big fat CM cycles.

\begin{proposition}\label{prop cm-cycle decomp}
The CM cycle $C(α,µ)$ is a union of elementary CM cycles in the sense that there exists a finite set of pairs $(σ_i,g_i)_{i\in I}$ as in Def.~ \ref{def element CM cycle} such that over $\mbC$
\begin{equation}\label{eq CM_g decomp}
C(α,µ) = \sum_{i\in I} C(σ_i,g_i).
\end{equation}
\end{proposition}
\begin{proof}
The translation between moduli and group theory description of the complex points $M_{\wt K}(\mbC)$ is given in \cite[Thm.~3.5]{RSZ3}: For a point $\underline{A} = (A_0,A,\ov{η})\in M_{\wt K}(\mbC)$, the $n$-dimensional $F$-vector space
\begin{equation}\label{eq:hermitian_space_for_AV_over_C}
W(\underline{A}) := -\Hom_F(H_1(A_0(\mbC),\mbQ),H_1(A(\mbC),\mbQ))
\end{equation}
is isometric to $V$. \emph{Transport de structure} of the complex structure on $\Hom_F(\Lie A_0, \Lie A)$ and the level structure $\ov{η}$ along \emph{any} choice of isometry $γ\colon W(\underline{A})\isoarrow V$ defines a point $[x,g\wt K]\in X\times \wt G(\mbA_f)/\wt K$ whose $\wt G(\mbQ)$-orbit represents $\underline{A}$.

For fixed $\underline{A}$ and fixed isometry $γ$, corresponding to $[x,g\wt K]$, the $F$-linear, polarization-preserving quasi-endomorphisms $φ$ of $A$ with characteristic polynomial $α$ are in bijection with the set
$$\{u\in G(F_0)\mid \mr{char}(u) = α,\ u x u^{-1} = x\},$$
the bijection being $φ\mapsto γφγ^{-1}$. Elements $u\in G(F_0)$ with $\mr{char}(u) = α$ in turn are in bijection with $*$-embeddings $σ\colon E\to \End(V)$ via $σ(t):=u$. (Here, $t$ continues to denote the distinguished element of $E = F[t]/α(t)$.)
Given $u$, or equivalently $σ$, denote its centralizer in $G$ by $T_u$. Then $x$ factors over $T_{u,\mbR}$ and is, in fact, uniquely determined by the property that the composite map
$$\mbC^\times \overset{x}{\to}  T_{u,\mbR} \overset{σ^{-1}}{\to } T_\mbR$$
equals \eqref{eq h_T}. Taking the level structure into account, one obtains that
\begin{equation}\label{eq CM_g unif}
C(α,µ) = \wt G(\mbQ)\backslash \big\{(x,g,u)\}
\end{equation}
where $x\in X,\ g\in \wt G(\mbA_f)/\wt K,\ u\in G(F_0)$ and the $\wt G(\mbQ)$-action are as follows.
\begin{enumerate}
\item The characteristic polynomial of $u$ is $α$. Let $T_u$ be its centralizer; use $σ\colon t\mapsto u$ as above to identify it with $T$.
\item The map $x\colon \mbC^\times\to \wt G_\mbR$ is determined by $u$ as explained above.
\item The level structure datum $g\wt K$ and $u$ are compatible with $µ$ in the sense that $u \in gµg^{-1}$.
\item The action of $\wt G(\mbQ)$ on such triples is componentwise. It is the usual action on $x$ and $g$ and the conjugation action on $u$.
\end{enumerate}
All that is left is to rewrite \eqref{eq CM_g unif} in terms of elementary CM cycles. For this we fix representatives $u_i$ for $ G(F_0)\backslash \{u\in G(F_0)\text{ s.th. }\mr{char}(u) = α\}$ and decompose $µ$,
$$µ = \coprod_{j\in J} \wt Km_j.$$
Let $σ_i\colon T\to \wt G$ be the $*$-embedding corresponding to $u_i$ with image $T_i$. Let $x_i\in X$ be the point determined by $σ_i$. Then
$$\begin{aligned}
C(α,µ) &= \sum_{i,j} T_i(\mbQ)\backslash \left(\sum_{g \in \wt G(\mbA_f)/\wt K} 1_{(g\wt K m_jg^{-1})}(u_i)\cdot (x_i, g, u_i)\right)\\
& = \sum_{i,j} \sum_{\ov g\in T_i(\mbA_f)\backslash \wt G(\mbA_f)/\wt K} 1_{\ov g \wt K m_j\ov g^{-1}}(u_i)\cdot T_i(\mbQ)\backslash \left( \sum_{s\in T_i(\mbA_f)/T_i(\mbA_f)\cap \ov g\wt K\ov g^{-1}} (x_i, s\ov{g}, u_i)\right).
\end{aligned}$$
Here, the intersection $T_i(\mbA_f)\cap \ov g \wt K\ov g^{-1}$ is independent of the chosen representative $\ov g$. The innermost sum agrees with
$$\left[T_i(\mbQ)\backslash T_i(\mbA_f)/(T_i (\mbA_f)\cap \ov g\wt K\ov g^{-1})\right] \cdot \ov g\wt K,$$
which is precisely the elementary cycle $C(σ_i,\ov g)$.
\end{proof}

\begin{corollary}\label{cor cm-cycle degree decomp}
The degree of $C(α,µ)$, as a function on the set of connected components, is constant along $\mr{Im}(ν(\mbA_f))$-cosets.
\end{corollary}
\begin{proof}
Each of the cycles $C(σ_i,g_i)$ in \eqref{eq CM_g decomp} has degree function a scalar multiple of the indicator function of a double coset $1_{\mr{Im}(π(g_i)ν(\mbA_f))π(\wt K)}$.
\end{proof}

\subsection{Prescribed good reduction}
It is the derived cycle $\LCM(α,µ)$ that is of interest in the context of the AFL. The following result allows to suitably modify it by elementary CM cycles with maximal level at the prime one is interested in. Write $Λ = Λ^{\mfd}\subseteq V$ for the $ O_F[\mfd^{-1}]$-lattice used to define the level structure away from $\mfd$ in §\ref{ss:integral models}.

\begin{proposition}\label{prop modif cycle exists}
Let $p\nmid \mfd$ be a prime. There exists a $*$-embedding $σ\colon E\to \End(V)$ such that $Λ_p$ is $ O_{E,p}$-stable. For any such $σ$ there are pairs $λ_j \in \mbQ,\ g_j\in G(\mbA_{0,f}^p)$, with $j\in J$ finite, such that $C(α,µ) - \sum_{j\in J} λ_j \cdot C(σ,g_j)$ has degree $0$ on each connected component of $M_{\wt K}(\mbC)$.
\end{proposition}
\begin{proof}
For the first statement, it is enough to find a $*$-embedding $σ$ such that there is \emph{some} self-dual $ O_{E,p}$-lattice in $V_p$. Indeed, the group $\U(V_p)$ acts transitively on the set of self-dual lattices in $V_p$. The claim then follows from the fact that $\U(V)$ is dense in $\U(V_p)$ which allows to conjugate $σ$ suitably.

We have assumed $α$ to be the characteristic polynomial of an element of $G(F_0)$, so there exists some $*$-embedding $σ$ and we fix one. Let $w$ be a place of $F_0$ above $p$ and let $V_{w}$ be the corresponding factor of $V_p$. If $w$ is split in $F$, then all places $v$ of $E_0$ above $w$ are also split (in $E$) and the existence of a self-dual $ O_{E,w}$-lattice in $V_{w}$ is automatic.

So assume that $w$ is inert in $F$ and consider the orthogonal decomposition
$$V_{w} = V_{w,\mr{split}} \times \prod_{v\mid w\ \mr{non-split}} V_{v}$$
where the split factor is the product of all factors for places $v\mid w$ of $E_0$ that split in $E$ and where the remaining factors are for the inert $v$. (Note that $v$ is necessarily inert if it is non-split because $w$ is inert in $F$.) We claim that a self-dual $ O_{E,w}$-lattice exists if and only if all inert factors $V_{v}$ are quasi-split as $F_w$-hermitian spaces, i.e. have discriminant in $N_{F/F_0}(F_w^\times)$. Namely in this case let $L\subseteq V_{v}$ be any $ O_{E,v}$-lattice. Then the dual lattice $L^\vee$ is also $ O_{E,v}$-stable and hence of the form $L^\vee = π^rL$ where $π\in E_{v}$ is a uniformizer. The index $[L:π^rL]$ is even since $V_{v}$ is assumed quasi-split, but it also equals $rf$ where $f$ is the inertia degree of $E_{v}/F_{w}$. This degree is necessarily odd, so $r$ has to be even and $π^{r/2}L$ is a self-dual lattice.

Since $V_{w}$ is itself quasi-split, there is an even number of ``bad'' non-quasi-split factors $V_{v}$. We claim that one may suitably replace those in a sense we now explain.
Lifting the hermitian form along $\mr{tr}_{E/F}$ makes $V$ into a $1$-dimensional hermitian $E$-vector space. Changing its discriminant at the evenly many ``bad'' $v$ just exhibited defines a new $1$-dimensional $E/E_0$-hermitian space $V'$. Composing its hermitian form again with $\mr{tr}_{E/F}$ defines an $n$-dimensional hermitian $F$-vector space $V'$ together with a $*$-embedding $σ'\colon E\to \End(V')$.

We claim that $V\iso V'$ as hermitian $F$-vector spaces and that all inert factors $V'_{v}$ of $V'$ are quasi-split. Indeed, the localizations $V_{u}$ and $V'_{u}$ are isomorphic for all places $u\neq w$ of $F_0$ by construction. Since $V$ and $V'$ must differ at an even number of places, this already implies $V\iso V'$ over $F$. The claim on the inert factors of $V'$ follows from the elementary fact that two $E_{v}$-hermitian spaces are isomorphic if and only if their restrictions to $F_{w}$ along the trace of $E_{v}/F_{w}$ are.
Taking $σ\colon E\to \End(V)$ as the embedding induced from some choice of isometry $V\iso V'$ concludes the proof of the first statement in the proposition.

Now fix $σ$ such that $Λ_p$ is $ O_{E,p}$-stable. Let
$$J:=Z(\mbQ)\backslash  Z(\mbA_f)/π(\wt K)\mr{Im}(ν(\mbA_f)).$$
For each $j\in J$ choose some element $g_j\in G(\mbA_{0,f}^p)$ lifting $j$. Note for the existence of $g_j$ that $ν\vert_{F^\mbQ}$ is the identity and we are merely considering $j$ up to $\mr{Im}(ν(\mbA_f))$. So we only need weak approximation for $F^1$ (i.e. the density of $F^1(\mbQ) \subseteq F^1(\mbQ_p)$), which follows from the fact that any element of $F^1(\mbQ_p)$ is of the form $x/σ(x)$ with $x\in F^\times(\mbQ_p)$. Then we use the surjectivity of $π\colon G(\mbA_f)\to F^1(\mbA_f)$. Setting
$$λ_j:= \frac{\deg C(α,µ)_j}{\deg C(σ,g_j)},$$
where $C(α,µ)_j$ is the part of the cycle contained in the connected components bundled by $j$, constructs the sought for $(λ_j,g_j)$ by Cor.~ \ref{cor cm-cycle degree decomp}.
\end{proof}

\section{Elementary CM cycle intersection}\label{s:el CM}
\subsection{Main Result}
\label{ss:Setting}
Throughout this section, $σ\colon E\to \End(V)$ is a $*$-embedding such that $Λ[\mfd^{-1}]$ is stable under $ O_E[\mfd^{-1}]$ and $g\in G(F_{0,\mfd})$.
Let $C:=C(σ,g)$ be the corresponding elementary CM cycle, denote by $\mcC := \mcC(σ,g) \in \mcZ_1( O_\bfE\otimes_{ O_\bfF} \mcM)$ its integral model viewed as a cycle.

Also, $ϕ = ϕ_\mfd \otimes 1_{\wh {Λ}^\mfd} \in \mcS(V(\mbA_{0,f}))^K$ will always denote a Schwartz function that is the standard function away from $\mfd$. Even though all definitions and results in the following only depend on the $ϕ_\mfd$-component, we will usually use $ϕ$ in our notation.

The arithmetic intersection formalism from \S\ref{ss:AIP} provides intersection numbers
$$\left( \widehat{\mcZ}^\bK(ξ,h_\infty,ϕ),\ \mcC \right) \in \mbR_\mfd$$
for $ξ\in F_0$ and $h_\infty\in \bH(F_{0,\mbR})$. The result below relates these to the Fourier coefficients of a modular function. Since $h_\infty$ is allowed to vary, it will be slightly stronger to formulate this for uniform (in $h_\infty$) lifts to $\mbR$. To this end, denote by
$$\left\langle\widehat{\mcZ}^\bK(ξ,h_\infty,ϕ),\ \mcC \right\rangle\in \mbR$$
the following smooth function in $h_\infty$. If $ξ \neq 0$, it is the intersection number in the sense of \eqref{eq:example_pairing}. Thm.~\ref{thm CM inter} below shows that $\mcZ(ξ,ϕ) \cap \mcC$ is artinian for $\xi\neq 0$, so \eqref{eq:example_pairing} is applicable. For $ξ = 0$, in accordance with definition \eqref{eq:Z Ku 0}, choose any lift $λ$ of $(-\phi(0) \, c_1(\wh{\bm{\omega}}), \mcC)$ to $\mbR$, then put
\begin{equation}\label{eq:def_lift_omega}
\left\langle\widehat{\mcZ}^\bK(0,h_\infty,ϕ),\ \mcC \right\rangle = λ + \left\langle  \left(0, \left( \CG^{\bK}_\nu(0,h_ {\nu|_{F_0}},\phi)  \right)_{\nu} \right ) ,\ \mcC \right\rangle.
\end{equation}

For each non-archimedean place $v$ of $F_0$, we define a distinguished maximal compact open subgroup of $\bH( F_{0,v})$
\begin{align}\label{eq:max open}
K_{\bH,v}^\circ:=m'( c_v)^{-1} \bH( O_{F_0,v}) m'( c_v),
\end{align}
where $c_v O_{F_0,v}=\delta_{F_0/\BQ,v}$ and $m'( c_v)= \left(\begin{matrix}c_v &\\ &1\end{matrix}\right)$ (note that $m'( c_v)$  is in $\GL_2(F_{0,v})$ rather than $\SL_2(F_{0,v})$). Set $K_{\bH}^{\mfd,\circ}=\prod_{v\nmid\fkd}K_{\bH,v}^\circ$.
\begin{theorem}\label{thm elementary CM modularity}
Assume that the CM type $Φ$ is unramified at all $p\nmid \mfd$ in the sense below. Assume further that $ϕ_\mfd \in \mcS(V(F_{0,\mfd}))^{K_\mfd\times K_{\bH,\mfd}}$ is $K_{\bH, \mfd}$-invariant under the open compact subgroup $K_{\bH, \mfd} \subseteq \bH(F_{0,\mfd})$ and put $K_\bH=K_{\bH,\mfd}\times K_{\bH}^{\mfd,\circ}$.

Then there exist coefficients $ε_ξ \in \big(\sum_{\ell\mid \mfd} \mbQ \log \ell\big)$ when $ξ \neq 0$, resp. $ε_0\in \mbR$, such that the generating series
\begin{equation}\label{eq CM gen series}
\sum_{ξ\in F_0}\left[ε_ξ + \left\langle \widehat{\mcZ}^\bK(ξ,h_\infty,ϕ),\ \mcC \right\rangle \right]W_\xi^{(n)}(h_\infty) 
\end{equation}
is the restriction to $h_\infty\in \bH(F_{0,\BR})$ of some element in $\CA_\infty(\bH(\BA_0),K_\bH,n)$.
\end{theorem}

\begin{definition}\label{def:unramified_CM_type}
The CM type $Φ$ is unramified at a prime $p$ if it satisfies the following condition.  For any choice of isomorphism $ι\colon \mbC\isoarrow \mbC_p$ (equivalently, for one choice), the resulting $p$-adic CM type $ι\circ Φ$ is the inverse image under $\Hom(F_p,\mbC_p)\to \Hom(F_p^u,\mbC_p)$ of a $p$-adic CM type $Φ^u$ for $F_p^u$. Here, $F_p^u$ denotes the maximal subalgebra of the $p$-adic completion $F_p$ that is unramified over $\mbQ_p$.
\end{definition}

Assuming $Φ$ to be unramified simplifies the computation of intersection numbers in that it allows to apply the comparison isomorphism from \cite{M-Th}. It is also the assumption that makes the arguments of Howard \cite{H-CM} from the case $F_0 = \mbQ$ go through without essential changes. (Note that $Φ$ is trivially unramified if $F_0 = \mbQ$.)

The proof of Thm. \ref{thm elementary CM modularity} is by computing the intersection numbers in question and comparing them with the Fourier coefficients of a known-to-be-modular analytic generating series, which is carried out in the current section and the next. The final assembly of results will happen after Prop. \ref{prop:Weil del J em}. The following remark provides two a priori reduction steps.

\begin{remark}\label{rmk:intersection_generating_series}
(1) An immediate relation for the intersection numbers of interest is
$$\left\langle \widehat{\mcZ}^\bK(ξ,h_\infty,ϕ),\ \mcC(σ,g)\right\rangle_{O_{\bfE}\otimes_{O_\bfF}\mcM_{\wt K}} = \left\langle \widehat{\mcZ}^\bK(ξ,h_\infty,gϕ),\ \mcC(σ,1)\right\rangle_{O_{\bfE}\otimes_{O_\bfF}\mcM_{g\wt K g^{-1}}}$$
where the two intersections are taken for the Shimura variety of the indicated level (after base change to $ O_\bfE[\mfd^{-1}]$). In particular we may (and will) assume that $g = 1$ from now on.

(2) Extending the base field from $\bfE$ to $\bfE'$ in the above definition multiplies the intersection numbers by $[\bfE':\bfE]$. So we may replace $\bfE$ by a larger number field and assume in the following that all the points $C(\mbC)$ are defined over $\bfE$. Moreover, we are free to compute the intersection numbers on the normalization of $\mcC$ and will take $\mcC := \coprod_{C(\bfE)} \Spec  O_\bfE[\mfd^{-1}]$ in the following.
\end{remark}

\subsection{The intersection $\widehat{\mcZ}^\bK(ξ,h_\infty,ϕ)\cap\, \mcC$}
\label{ss:describing CM intersection}

The definition of the elementary CM cycle $C$ is in terms of Shimura data. We now explain why its points correspond to abelian varieties with $O_E[\mfd^{-1}]$-action under the isomorphism $\Sh_{\wt K}\big(\wt G, \big\{\wt h\big\}\big) \iso M_{\wt K}$. To this end, we recall some details of the construction of this map that will not be used elsewhere: To a Deligne homomorphism that occurs in the conjugacy class of the Shimura datum,
$$\wt h = (h_{Z^\mbQ}, h)\colon \mbC \to (\mbR\otimes_\mbQ F) \times \End_{\mbR \otimes_{\mbQ} F}(\mbR \otimes _\mbQ V),$$
one associates the two complex tori,
$$A'_0 = (\mbR\otimes_{\mbZ} Λ_0)/ Λ_0,\ \ A' = (\mbR\otimes_\mbZ (Λ_0\otimes_{O_F} Λ))/(Λ_0\otimes_{O_F} Λ),$$
where the complex structures are given by $\wt h$. Here, $Λ_0$ is an invertible $O_F$-module together with an $F$-valued hermitian form $(\ ,\ )_0$ such that $Λ_0[\mfd^{-1}]$ is self-dual with respect to the $\mbQ$-valued alternating pairing $\mr{tr}_{F/\mbQ}ζ^{-1}(\ ,\ )_0$; the element $ζ\in F^\times$ is an auxiliary choice of traceless element as in \cite[§3.1]{RSZ3}. The forms on $Λ_0$ and $Λ_0\otimes_{O_F}Λ$ then define quasi-polarizations on $A'_0$ and $A'$ that are principal away from $\mfd$; we will not need this and refer to \cite[§3.4]{RSZ3}.

Given a coset $(z,g)\in \wt G(\mbA_f)/\wt K$ (cf.~\S\ref{ss:S data}) as an additional datum, one then considers the isogeneous quasi-polarized abelian varieties $(A_0, A)$ with Tate modules equal to
$$z\wh{Λ_0}\ \ \text{and}\ \ (z\otimes z^{-1}g)(\wh{Λ_0}\otimes_{\wh{O_F}} \wh{Λ}).$$
Together with a suitable definition of level structure, these define the image of $[\wt h, (z,g)]$ in $M_{\wt K}(\mbC_ν)$, where $ν\colon E\to \mbC$ lies above $φ_0$. We refer to \cite[Thm. 3.5]{RSZ3} and the references there for the full construction, because we just need the following consequence. For $(A_0, A, \ov{η}) \in C(\bfE)$, the map $σ\colon E \to \End_F(V)$ centralizes the corresponding pair $(h,g)$ and hence provides an action $E\to \End^\circ(A)$. Since $σ(O_E)$ is assumed to stabilize $Λ[\mfd^{-1}]$, one even obtains that
$$σ(E)\cap \End(A)[\mfd^{-1}] = σ(O_E[\mfd^{-1}]).$$

The universal object over $\mcC$ is obtained as the Néron model of the universal object over $C$ and hence also has the $ O_E[\mfd^{-1}]$-action by functoriality.
For $(A_0,A,\ov{η})\in \mcC$, we will use the shorthand
$$L(A_0,A) := \Hom_{O_F}(A_0,A)\ \ \ \text{and}\ \ \ V(A_0,A):= \Hom_{F}^\circ(A_0,A).$$ We endow these with the $E$-valued hermitian form $(\ ,\ )$ such that
$$\langle\ ,\ \rangle = \tr_{E/F}\circ (\ ,\ ).$$
Next, we define a notion of KR divisor on $\mcC$, very similarly to \cite[Def. 3.5.1]{H-CM}.
\begin{definition}\label{def CM cycle algebraic}
For $ζ\in E^\times_0$ and $µ\subseteq V(F_{0,\mfd})$ stable under $K_\mfd$, let
$$\mcC(ζ,µ) := \left\{(A_0, A,\ov{η},x)\left\vert \begin{array}{c}\text{$(A_0,A,\ov{η})\in \mcC$ and $x\in L(A_0,A)[\mfd^{-1}]$}\\
\text{s.th. $(x,x)=ζ$ and $η(x)\in µ$ for $η\in \ov{η}$} \end{array}\right\}\right..$$
\end{definition}
\begin{lemma}
Each $\mcC(ζ,µ)$ is an artinian scheme. It is empty if $ζ$ is not totally positive.
\end{lemma}
\begin{proof}
The first statement is equivalent to $V(A_0,A) = 0$ for all $(A_0,A,\ov{η})\in C(\bfE)$. Assume for the sake of contradiction that there were $0\neq x\colon A_0\to A$. It would imply the existence of an $ O_E$-linear isogeny
$$ O_E\otimes_{ O_F} A_0 \to  A.$$
But the domain has CM type $Φ'$, while the codomain has CM type $\{Φ' \cup φ'_0\}\setminus \{\ov{φ'_0}\}$, so such an isogeny cannot exist.

The spaces $V(A_0,A)$ are positive definite, implying the second statement.
\end{proof}

Since $\mcC$ is normal by definition, the $\mcC(ζ,µ)$ define Cartier divisors. (Using the normality is actually not necessary as Prop. \ref{prop CM intersection decomp} below shows.) Next we define a Green function for $\mcC(ζ,µ)$, still assuming $ζ\neq 0$.

 Given a point $(A_0,A,\ov{η})\in C(\bfE_ν)$ above some place $ν\colon \bfE\to \mbC$, consider the hermitian $F$-vector space $W(A_0,A)$ from \eqref{eq:hermitian_space_for_AV_over_C}. It comes with a hermitian $E$-action (because $E$ acts on $A$ compatibly with the polarization) and we view it as a $1$-dimensional hermitian $E$-vector space by lifting its $F$-valued form along $\mr{tr}_{E/F}$. Let $v = ν\vert_{E_0}$. Then $W(A_0,A)$ is $v$-nearby to $V$, meaning that it is negative definite at $v$, positive definite at all other archimedean places, and furthermore $W(A_0,A)(\mbA_{E_0,f}) \iso V(\mbA_{E_0,f})$. In fact, there is a natural equality of hermitian $\mbA_{F,f}$-modules,
$$W(A_0,A)(\mbA_{F_0,f}) = {\prod_{p < \infty}}'V_p(A_0, A),$$
so every choice $η\in \ov{η}$ defines an isometry $η\colon W(A_0, A)(\mbA_{F_0,f}) \iso V(\mbA_{F_0,f})$ of hermitian $\mbA_{F,f}$-modules. Since $(A_0, A, \ov{η}) \in C$, the choice can even be made $E$-linearly.
\begin{definition}\label{def CM stack arith divisors}
\begin{enumerate}
\item Let $h_\infty\in \SL_2(E_{0,\mbR})$ be a parameter. Define the function $g(ζ,h_\infty,µ)$ on $C(\mbC\otimes_{\mbQ} \bfE)$ as follows. Its value at a point $(A_0,A,\ov{η})\in C(\bfE_ν)$ is\footnote{In this paper we use $\Ei(4πa_v v(ζ))$ instead of $\Ei(2πa_v v(ζ))$ due to our different convention from \cite[\S12]{Z12}.}
\begin{equation}\label{eq def g_nu}
\sum_{x\in W(A_0, A)_ζ} -\Ei(4πa_v v(ζ)) 1_{µ}(ηx).
\end{equation}
Here, $v = ν\vert_{E_0}$ is the place of $E_0$ below $ν$, while $η\in \ov{η}$ is any choice. The sum \eqref{eq def g_nu} is well-defined because $µ$ is $K$-stable. Also, $a_{v}$ is as in  the Iwasawa decomposition \eqref{h infty} of the $v$-component of $h_\infty$.

\item If $ϕ_\mfd = \sum_i λ_i 1_{µ_i} \in \mcS(V(F_{0,\mfd}))^{K_\mfd}$, set
$$\widehat{\mcC}^\bK(ζ,h_\infty,ϕ) := \sum_{i} λ_i \left(\mcC(ζ,µ_i),g(ζ,h_\infty,µ_i)\right).$$ 
\end{enumerate}
\end{definition}

\begin{definition}\label{def CM stack arith divisor 0th}
The cycle for $ζ = 0$ is again only defined as an element of $\wh\Ch{}^1(\mcC)$:
$$\wh\mcC^\bK(0,h_\infty, ϕ) := -ϕ(0) \big[\,\wh {\bm{\omega}}\vert_{\mcC}  + \big(0,\, (\log|a_{v}|)_{ν,\, v = ν\vert_{E_0}}\big)\,\big],$$
cf. \eqref{eq:Z Ku 0}.\end{definition}

\begin{proposition}\label{prop CM intersection decomp}
For every $ξ\in F_0$ and $h_\infty \in \SL_2(F_{0,\mbR})$, there is an identity of Arakelov divisors on $\mcC$,
\begin{equation}\label{eq:CM_cycle_intersection_identity}
\widehat{\mcZ}^\bK(ξ,h_\infty,ϕ)\cap \mcC = \sum_{ζ\in E_0,\ \tr_{E_0/F_0}(ζ) = ξ} \widehat{\mcC}^\bK(ζ,h_\infty,ϕ).
\end{equation}
\end{proposition}
\begin{proof}
One is reduced to $ϕ_\mfd = 1_µ$ being an indicator function. In case $ξ \neq 0$, it is immediate from the moduli description that
$$\mcZ(ξ,µ) \cap \mcC = \coprod_{ζ\in E_0,\ \tr_{E_0/F_0}(ζ)=ξ} \mcC(ζ,µ),$$
proving the claim for the algebraic parts of the divisors. Still assuming $ξ \neq 0$, we turn to the archimedean components. Let $\CG^\bK_ν(ξ,h_\infty,µ)$ denote the $ν$-component Green function of $\widehat{\mcZ}^\bK(ξ,h_\infty,µ)$ (cf. \eqref{Gr Ku2}), where $ν\colon \bfE\to \mbC$. The aim is to compute $\CG^\bK_ν(ξ,h_\infty,µ)(\underline A)$ for $\underline A = (A_0, A, \ov{η})\in \mcC(\bfE_ν)$. Set $v := ν\vert_{E_0}$, $w := ν\vert_{F_0}$ and recall that $W(A_0, A)$ (as $E$-hermitian space) is the $v$-nearby space of $V$. The value of $\CG^\bK_ν$ at $\underline A$ is by definition
\begin{equation}\label{eq CM eval Green Fct}
\sum_{u\in W(A_0, A)_ξ} -\Ei\big(4πa_{w}R(u,x)\big)\cdot 1_µ(η u),
\end{equation}
where $η\in \ov{η}$ is any choice and where $x$ is a certain negative line in $\ell_x \subset \mbR \otimes_{F_0, w} W(A_0, A)$ that is determined by the Hodge structure on $W(A_0, A)$. (The precise definition, which we do not need, may be found in \cite[\S3]{RSZ4}.) But the Hodge structure on $W(A_0,A)$ is $E$-stable, so this line is the unique $E$-stable negative definite one, namely the $v$-eigenspace $\ell_x = \mbC \otimes_{E,v} W(A_0, A)$. Thus
$$R(u,x) = \langle\mr{pr}_{\ell_x}u,\mr{pr}_{\ell_x}u\rangle = v\big((u,u)\big)$$
for every $u\in W(A_0, A)$. It follows that \eqref{eq CM eval Green Fct} may be rewritten as
$$\sum_{ζ\in E_0,\ \mr{tr}_{E_0/F_0}(ζ) = ξ}\ \sum_{u\in W(A_0, A)_ζ} -\Ei\big(4πa_{v} v(ζ)\big)\cdot 1_µ(ηu)$$
as was to be shown.

The case $ξ = 0$ is handled analogously: Subtracting the term for $ζ = 0$ in \eqref{eq:CM_cycle_intersection_identity}, we obtain
\begin{equation}\label{Gr Ku2 repeat}
\wh{\mcZ}^\bK(0,h_\infty,µ) \cap \mcC - \wh {\mcC}^\bK(0, h_\infty, µ) = \left(0,\ \bigg(\sum_{(u,g)} 1_µ(g^{-1}u)\cdot \left(\CG_\nu^{\bf K}(u, h_\infty)\times  {\bf 1}_{g \,\wt K} \right) \bigg)_ν \right)
\end{equation}
where the $\nu$-term on the right hand side is a sum over $(u,g)\in V^{(w)}_{\xi=0}(F_0)\times \wt G(\BA_{f})/\wt K$ with $u\neq 0$, cf. \eqref{Gr Ku2}. Namely the contributions of the $(u = 0)$-summand in \eqref{Gr Ku m=0} and of the metrized automorphic bundle $\wh{\bm{\omega}}\vert_\mcC$ in \eqref{eq:Z Ku 0} precisely cancel with Def. \ref{def CM stack arith divisor 0th}.
The same arguments as before refine this further according to $(u,u) = ζ \neq 0,\ \mr{tr}_{E_0/F_0}(ζ) = 0$.
\end{proof}

From now on we always consider $V$ as a $1$-dimension hermitian $E$-vector space. Let $\mbV$ denote the rank $1$ hermitian $\mbA_E$-module which is positive definite at all archimedean places of $E$ and which is such that $\mbV_f \iso \mbA_{E,f}\otimes_E V$. It is \emph{incoherent} in the sense of Kudla \cite{K}. Given $ζ\in E_0^\times$, let $\Diff(ζ,\mbV)$ denote the set of places $v$ such that $\mbV_{v}$ does not represent $ζ$. This set is always non-empty of odd cardinality since the unique adelic hermitian space that represents $ζ$ is the adelification of $(E, ζ N_{E/E_0})$ which is coherent. For a non-split place $v$ of $E_0$, we again denote by $V^{(v)}$ the nearby hermitian $E$-vector space; it is the unique hermitian space that represents $ζ$ whenever $\Diff(ζ,\mbV) = \{v\}$.

We write
$$\mbA_E^1 = \{x\in \mbA_E^\times|N_{E/E_0}(x) = 1\}$$
and similarly with other objects where $N_{E/E_0}$ is defined, e.g. $E^1$ (as earlier) or $\wh O_E^1$. 
An important role is played by the following orbital integral. Let $ϕ\in \mcS(\mbV_f)$ be as before, let $ζ \neq 0$, assume $\mbV_{f,ζ}\neq \emptyset$ and let $x_ζ\in \mbV_{f,ζ}$ be some choice. Set
\begin{equation}\label{eq def orbital integral CM}
\Orb(ζ,ϕ) := \int_{\mbA^1_{E,f}} ϕ(t\cdot x_\zeta) dt
\end{equation}
where the volume is normalized by the open compact $σ^{-1}(K)\subseteq \mbA^1_{E,f}$. It is independent of the choice of $x_ζ$ since the hyperboloid $\mbV_{f,ζ}$ is a principal homogeneous space for $\mbA^1_{E,f}$. By assumption, $ϕ^\mfd = 1_{\widehat{Λ^\mfd}}$ is the characteristic function for an $ O_E[\mfd^{-1}]$-lattice and $σ^{-1}(K^\mfd) = O_E[\mfd^{-1}]\, \wh{\ }{}^{\,\,,\,1}$. Hence one may factorize
\eqref{eq def orbital integral CM} as
\begin{equation}
\Orb(ζ,ϕ) = \Orb_\mfd(ζ,ϕ_\mfd) \prod_{v < \infty,\, v\nmid \mfd} \Orb_{v}(ζ,ϕ_{v})
\end{equation}
where the local volumes are normalized by $σ^{-1}(K_\mfd)$ and the $ O_{E,v}^1$ respectively and where the local orbital integrals are defined similarly to \eqref{eq def orbital integral CM}.

We also need the following volumes:
\begin{equation}\label{eq volumes}
τ(Z^\mbQ) := \mr{Vol}\left(Z^\mbQ(\mbQ)\bs Z^\mbQ(\mbA_f)\right),\ \ \ τ(E^1) := \mr{Vol}\left(E^1\bs \mbA^1_{E,f}\right),\ \ \ τ(T) := τ(Z^\mbQ)τ(E^1)
\end{equation}
where the measures are normalized by $K_{Z^\mbQ}$ and $σ^{-1}(K)$, respectively. These are simply the degrees of the respective $0$-dimensional Shimura varieties,
$$τ(Z^\mbQ) = \deg_{\bfE}(\bfE\otimes_{\bfF} M_0)\ \ \ \text{and}\ \ \ τ(T) = \deg_{\bfE} C.$$
We use the arithmetic degree notation $\widehat{\deg}\ \widehat{\mcC}^\bK(ζ,h_\infty,ϕ) := \big( \widehat{\mcC}^\bK(ζ,h_\infty,ϕ), \mcC\big)$ in the following. For $ζ \neq 0$ it is meant in the sense of \eqref{eq:example_pairing}, for $ζ = 0$ it is defined by first choosing a representing metrized divisor for $\wh {\bm{\omega}}\vert_{\mcC}$ in Def. \ref{def CM stack arith divisor 0th} and then applying \eqref{eq:example_pairing}. Then $\widehat{\deg}\ \widehat{\mcC}^\bK(ζ,h_\infty,ϕ)$ is a smooth function in $h_\infty$ in both cases. In the case $ζ = 0$, it is only canonical up to $\sum_{\ell \mid \mfd}\mbQ \log \ell$.
\begin{theorem}\label{thm CM inter}
Let $ζ\in E_0^\times$ and $h_\infty\in \bH(E_{0,\mbR})$.
\begin{enumerate}[wide, labelindent=0pt, labelwidth=!, label=(\arabic*), topsep=2pt, itemsep=2pt]
\item If $\#\Diff(ζ,\mbV) > 1$, then $\widehat{\mcC}^\bK(ζ,h_\infty,ϕ) = (\emptyset, 0)$.
\item If $\Diff(ζ,\mbV) = \{ v\}$ is a singleton with $v$ non-archimedean and $v\nmid \mfd$, then $\widehat{\mcC}^\bK(ζ,h_\infty,ϕ)$ has the form $(\mcC(ζ,ϕ), 0)$. The cycle $\mcC(ζ,ϕ)$ has support above places $ν$ of $\bfE$ with $ν\mid v$ only and
\begin{equation}\label{eq Arakelov deg non-archimedean}
\widehat{\deg}\ \widehat{\mcC}^\bK(ζ,h_\infty,ϕ) = τ(Z^\mbQ)[\bfE:E_0] \max\left\{0,\frac{1+v(δ_{E_0/F_0}ζ)}{2}\right\} \Orb(ζ,ϕ^{v}) \log q_{v}.
\end{equation}
Here $δ_{E_0/F_0}$ denotes the different ideal of the indicated field extension. 
\item If $\Diff(ζ,\mbV) = \{ v\}$ is a singleton with $v$ archimedean, then $\widehat{\mcC}^\bK(ζ,h_\infty,ϕ)$ has the form $(\emptyset, (g_ν)_{ν\colon \bfE\to \mbC})$ where $g_ν\neq 0$ only for $ν\mid v$. For its degree,
\begin{equation}\label{eq Arakelov deg archimedean}
\widehat {\deg}\ \widehat{\mcC}^\bK(ζ,h_\infty,ϕ) = -\frac{τ(Z^\mbQ)[\bfE:E_0]}{2} \cdot \Ei\big(4πa_{v} v(ζ)\big) \Orb(ζ,ϕ).
\end{equation}
\item Finally, the $0$-th term is the following smooth function up to $\sum_{\ell\mid \mfd}\mbQ\log \ell$,
\begin{align}\label{eq:deg infty}
\widehat {\deg}\ \widehat{\mcC}^\bK(0,h_\infty,ϕ) = \wh\deg \,(\wh {\bm{\omega}}\vert_\mcC)  \phi(0) - \frac{τ(Z^\BQ)[\bfE:E_0]}{2}\tau(E^1) \phi(0)\ \sum_{v\colon E_0\to \mbR} \log |a_{v}|.\end{align}
\end{enumerate}
\end{theorem}

Only inert places may occur in case (2). So an implicit statement is that $\mcC(ζ,ϕ)$ never has support above a split place of $E_0$. Sections \S\ref{subsect deg arch}--\S\ref{subsect deg nonarch} are devoted to the proof of the theorem. Note that there is nothing to show for case (4) which follows directly from definitions. (The factor $1/2$ here is due to the normalization of the arithmetic intersection pairing in \S\ref{ss:AIP}, resulting from the convention \eqref{eq convention rat fct}.)

\subsection{Archimedean component}
\label{subsect deg arch}
Let $ζ\in E_0^\times$, let $ν$ be an archimedean place of $\bfE$, set $v := ν\vert_{E_0}$. As before (cf. proof of Prop.~ \ref{prop CM intersection decomp}), for all $(A_0,A,\ov{η})\in C(\bfE_ν)$, $W(A_0,A) \iso V^{(v)}$. Hence $W(A_0,A)$ represents $ζ$ if and only if $\Diff(ζ,\mbV) = \{v\}$, which proves (1) of Thm.~ \ref{thm CM inter} for the archimedean component of $\widehat{\mcC}^\bK(ζ,h_\infty,ϕ)$.

Now assume $\Diff(ζ,\mbV) = \{v\}$. In other words, assume that $\mbV_f$ represents $ζ$, that $v(ζ)<0$ and that $v'(ζ)>0$ for $v'\neq v$ archimedean.
Then one obtains after choice of a base point $c_0 \in C(\bfE_ν)$
$$\begin{aligned}
\sum_{c\in C(\bfE_ν)} g_ν(ζ,h_\infty,ϕ)(c)\
& = \sum_{t\in T(\mbQ)\backslash T(\mbA_f)/σ^{-1}(\wt K)} g_ν(ζ,h_\infty,ϕ)(tc_0) \\
& = \sum_{t\in T(\mbQ)\backslash T(\mbA_f)/σ^{-1}(\wt K)}\ \sum_{x\in V^{(v)}_ζ}-\Ei(4πa_{v}v(ζ))\cdot ϕ(tx)\\
& = -τ(Z^\mbQ) \Ei(4πa_{v} v(ζ)) \Orb(ζ,ϕ).\end{aligned}$$
The last equality follows since $\vol( σ^{-1}(\wt K))=1$ by our convention. Taking the sum over the $[\bfE:E_0]$-many different $ν\mid v$ results in the claimed identity
$$\widehat{\deg}\ \mcC^\bK(ζ,h_\infty,ϕ) = -\frac{τ(Z^\mbQ)[\bfE:E_0]}{2} \Ei(4πa_{v} v(ζ)) \Orb(ζ,ϕ).$$
The factor $1/2$ here is due to the normalization of the arithmetic intersection pairing in \S\ref{ss:AIP}, resulting from the convention \eqref{eq convention rat fct}.
\qed

\subsection{Non-archimedean components (locally)}
The aim of this section is to provide some auxiliary results for $p$-divisible groups with CM that will be needed for statement (2) of Thm.~ \ref{thm CM inter}. Our arguments are very close to those of \cite{H-CM}, albeit disguised by the use of \cite{M-Th}. Let $E/E_0$ be an unramified quadratic extension of $p$-adic local fields. Let $π\in E_0$ be a uniformizer and let $\breve E$ be the completion of a maximal unramified extension of $E$. Denote its residue field by $\mbF$. The point of departure for our computation of deformation lengths is Gross' result on canonical liftings. Let $\mbY$ be a formal height $1$ strict $ O_E$-module over $\mbF$ and let $Y$ be its canonical lift to $ O_{\breve E}$. Let also $\ov {\mbY}$ and $\ov {Y}$ denote the same formal groups, but with the conjugate $ O_E$-action. Write $Y_k$ and $\ov{Y}_k$ for their reductions modulo $π^{k+1}$. All homomorphisms in the following are $ O_E$-linear homomorphisms, unless indicated otherwise.

\begin{theorem}[Gross]
Let $u\in \Hom(\mbY, \ov{\mbY})$. Then $u$ lies in $\Hom(Y_k, \ov{Y}_k)$ if and only if it lies in $π^k\Hom(\mbY,\ov{\mbY})$.
\end{theorem}

The underlying strict $ O_{E_0}$-modules of the above groups are supersingular of height $2$. Choosing principal polarizations of $Y$ and $\ov{Y}$ as strict $ O_{E_0}$-modules, compatibly with the $ O_E$-action, one gets a hermitian form on $\Hom(\mbY, \ov{\mbY})$ by the usual rule $(x,y) := y^* \circ x$. It is easily checked through Dieudonné theory that there is an isometry
$$\Hom(\mbY, \ov{\mbY}) \iso ( O_E, πN_{E/E_0}).$$
Denoting by $\mcZ(u)\subseteq \Spf  O_{E}$ the locus to which $u\colon \mbY\to \ov{\mbY}$ extends as a homomorphism $Y_k\to \ov{Y}_k$, Gross' result may be reformulated as
$$\mr{len}\  \mcO_{\mcZ(u)} = \frac{1 + v\big((u,u)\big)}{2}.$$
Here $v$ is the normalized valuation of $E_0$.

In order to apply this result in the context of polarized abelian varieties one needs to pass from supersingular polarized strict $ O_{E_0}$-modules to ``plain'' polarized $p$-divisible groups with $ O_{E}$-action. This has been achieved in \cite{M-Th} through the use of displays. Let $Φ'$ be a $p$-adic CM type for $E$ satisfying $Φ'\sqcup \ov{Φ'} = \Hom(E,\mbC_p)$ and which is the inverse image of a CM type ${Φ'}{}^{,u}$ for the maximal subfield $E^u$ of $E$ unramified over $\mbQ_p$, similar to Def. \ref{def:unramified_CM_type}. Also fix an element $φ'_0\in Φ'$ and extend it to a map $φ'_0\colon \breve E\to \mbC_p$. Then $φ'_0(\breve E)$ contains the reflex fields for $Φ'$ and $\{\ov{φ'_0}\}\cup Φ' \setminus \{φ'_0\}$.

By definition, a \emph{hermitian $ O_{E}$-$\mbZ_p$-module} of signature $(1,0)$ (resp. $(0,1)$) in the sense of \cite{M-Th} is a supersingular $p$-divisible group of height $[E:\mbQ_p]$ together with an $ O_{E}$-action of CM type $Φ'$ (resp. $(Φ'\cup \{\ov{φ'_0}\}) \setminus \{φ'_0\}$) and a principal polarization as $p$-divisible group that is compatible with the $ O_{E}$-action. By the main result \cite[Thm.~ 3.1]{M-Th}, the categories of such objects are equivalent to those of the polarized height $2$ strict $ O_{E_0}$-modules with $ O_{E}$-action considered before. Hence Gross' result carries over in the following form.

Let $X$ and $\ov X$ be hermitian $ O_{E}$-$\mbZ_p$-modules of signature $(1,0)$ resp. $(0,1)$ over $\Spf  O_{\breve E}$. (The notation is suggestive, but note that $\ov X$ is not the same as $X$ with conjugate $ O_{E}$-action.) Let $\mbX$ and $\ov \mbX$ (resp. $X_k$ and $\ov X_k$) denote their reductions modulo $π$ (resp. $π^{k+1}$). Endow $\Hom(\mbX, \ov \mbX)$ with the usual $ O_{E}$-valued hermitian form. Let $\mcZ(u)\subseteq \Spf  O_{\breve E}$ denote the locus where the homomorphism $u\colon \mbX\to \ov \mbX$ lifts to a map $X_k\to \ov X_k$.
\begin{corollary}
For $u\colon \mbX\to \ov \mbX$ a homomorphism,
$$\mr{len}\  \mcO_{\mcZ(u)} = \frac{1 + v \big( (u,u) \big)}{2}.$$
\end{corollary}

Assume now that $F/F_0$ is an unramified quadratic extension with $F\subseteq E$ and $F_0 = F \cap E_0$. Assume that the CM type $Φ'$ actually is the inverse image of a CM type $Φ^u$ for $F^u$, the maximal subfield of $F$ unramified over $\mbQ_p$. Let $X_0$ be a hermitian $ O_F$-$\mbZ_p$-module over $\Spf  O_{\breve E}$ of signature $(1,0)$ with respect to $Φ$, denote its reduction mod $π$ by $\mbX_0$. Then
$$X'_0 :=  O_{E_0}\otimes_{ O_{F_0}} X_0$$
has CM by $ O_{E}$ of CM type $Φ'$, i.e is of signature $(1,0)$. Choosing a perfect symmetric $ O_{F_0}$-linear pairing on $ O_{E_0}$ endows $X'_0$ with a principal polarization. By the uniqueness of hermitian $ O_{E}$-$\mbZ_p$-modules of signature $(1,0)$, cf. \cite[\S2]{M-Th}, it is isomorphic to $X$. With the same notational conventions as before one obtains
\begin{equation}\label{eq comp herm lattices}
\Hom_{ O_F}(\mbX_0, \ov \mbX) \isoarrow \Hom(\mbX, \ov \mbX),\ \ u\mapsto \id_{O_{E_0}}\otimes u
\end{equation}
and analogous isomorphisms for the (unique) deformations of all involved objects to $ O_{\breve E}/π^{k+1}$. This allows to also apply Gross' result to deformations of $u\in \Hom_{ O_F}(\mbX_0, \ov \mbX)$.
Note that this is the space we care for most, i.e. it is the local analogue of $L(A_0,A)$ from \S\ref{ss:describing CM intersection}.

The subtlety now lies with the fact that the hermitian forms in \eqref{eq comp herm lattices} are different. The space $\Hom_{ O_F}(\mbX_0, \ov \mbX)$ naturally carries the $ O_F$-valued form $\langle x, y\rangle := y^*\circ x$. It lifts naturally to the $ O_{E}$-valued form $(\ ,\ )$ such that $\mr{tr}_{E/F} (x,y) = \langle x, y\rangle$. For later use, the following proposition also treats the case $\Hom_{ O_F}(\mbX_0, \mbX)$, to which the above formalities extend mutatis mutandis.

\begin{proposition}\label{prop herm lattice locally}
There are isomorphisms of hermitian lattices
$$\Hom_{ O_F}(\mbX_0, \mbX) \iso ( O_E,\ δ_{E_0/F_0}^{-1} N_{E/E_0}),\ \ \ \Hom_{ O_F}(\mbX_0, \ov \mbX) \iso ( O_E,\ π δ_{E_0/F_0}^{-1} N_{E/E_0}),$$
where $δ_{E_0/F_0}$ generates the different ideal of $E_0/F_0$.
\end{proposition}
\begin{proof}
Denote by $M := M(\mbX_0)$ the (covariant) Dieudonné module of $\mbX_0$ and by $N := M(\mbX)$ (resp. $N := M(\ov \mbX)$) the one of $\mbX$ (resp. $\ov \mbX$), depending on the case one is interested in. Up to isomorphism,
\begin{equation}
\label{eq Def M}
M =  O_F\otimes_{\mbZ_p} \breve {\mbZ}_p = \bigoplus_{i\colon F^u\to \breve \mbQ_p} M_i
\end{equation}
endowed with the Verschiebung
\begin{equation}\label{eq Dieud Module 1}
\big[V\colon M_{i+1}\to M_i\big] := σ^{-1} \cdot \begin{cases} p & \text{if}\ i \in Φ^u\\
1 & \text{if}\ i \in \ov{Φ^u}.\end{cases}
\end{equation}
Here $σ$ is the Frobenius of $\breve \mbZ_p$ and $Φ^u$ is the CM type for $F^u$ that induces $Φ$. Again up to isomorphism, there is a unique way to define a polarization on $(M,V)$ that is compatible with the $ O_F$-action. Namely endow $M$ with the $\breve\mbZ_p$-linear extension of a perfect, $ O_F$-hermitian, skew-symmetric pairing
\begin{equation}\label{eq bilin F}
 O_F\times  O_F\to \mbZ_p.
\end{equation}
These are all of the form
$$\langle m_1,m_2\rangle := \mr{tr}_{F/\mbQ_p}(δ_{F/\mbQ_p}^{-1}m_1\ov{m}_2)$$
where $δ_{F/\mbQ_p}$ denotes a totally imaginary generator of the different $\mcD_{F/\mbQ_p}$ and where $m\mapsto \ov m$ is the Galois conjugation $\mr{conj}_{F/F_0} \otimes\mr{id}_{\breve \mbZ_p}$.

The description of $N$ is analogous. The underlying $\breve \mbZ_p$-module with $ O_E$-action is (in both cases)
\begin{equation}
\label{eq Def N}
N =  O_{E}\otimes_{\mbZ_p} \breve {\mbZ}_p = \bigoplus_{j\colon E^u\to \breve \mbQ_p} N_j.
\end{equation}
For $\mbX$ the Verschiebung is
\begin{equation}\label{eq Dieud Module 2}
\big[V\colon N_{j+1}\to N_j\big] := σ^{-1} \cdot \begin{cases} p & \text{if}\ j \in Φ'^{,u}\\
1 & \text{if}\ j \in \ov{Φ'{}^{,u}}.\end{cases}
\end{equation}
In case of $\ov \mbX$ it equals
\begin{equation}\label{eq Dieud Module 3}
\big[V\colon N_{j+1}\to N_j\big] := σ^{-1} \cdot \begin{cases} p & \text{if}\ j\in Φ'^{,u} \setminus \{φ'_0\}\\
p/π & \text{if}\ j = φ'_0\\
1 & \text{if}\ j \in \ov{Φ'{}^{,u}}\setminus\{\ov{φ'_0}\}\\
π & \text{if}\ j = \ov{φ'_0}.
\end{cases}
\end{equation}
The polarization is defined as in the case of $M$ through the choice of a totally imaginary generator $δ_{E/\mbQ_p}$ of the different $\mcD_{E/\mbQ_p}$.

Let $1_M\in M$ and $1_N\in N$ denote the unit elements (i.e. choices of generators) in Defs. \eqref{eq Def M} and \eqref{eq Def N}. Then $\Hom(M,N)$ identifies with $ O_E\otimes_{\mbZ_p}\breve \mbZ_p$ via $u\mapsto u(1_M)/1_N$. The dual map $u^\vee$ is then
$$\begin{array}{rcl}
\mcD_{E/\mbQ_p}^{-1}\otimes_{\mbZ_p} \breve \mbZ_p \underset{δ_{E/\mbQ_p}}{\iso}  O_E\otimes_{\mbZ_p}\breve \mbZ_p  &
\to &
 O_F\otimes_{\mbZ_p}\breve \mbZ_p \underset{δ_{F/\mbQ_p}}{\iso} \mcD_{F/\mbQ}^{-1}\otimes_{\mbZ_p}\breve \mbZ_p\\
x & \mapsto & δ_{F/\mbQ_p} \mr{tr}_{E/F}(δ_{E/\mbQ_p}^{-1} \ov x)
\end{array}$$
where by $\mr{tr}_{E/F}$ we really mean the $\breve \mbZ_p$-linear extension of the trace. The natural lift along $\mr{tr}_{E/F}$ is simply
$$x\mapsto δ_{F/\mbQ_p}δ_{E/\mbQ_p}^{-1} \ov x = δ_{E_0/F_0}^{-1} \ov x,$$
where we have set $δ_{E_0/F_0}^{-1} = δ_{F/\mbQ_p}δ_{E/\mbQ_p}^{-1}$. This element lies in $E_0$ and generates the inverse different $\mcD^{-1}_{E_0/F_0}$.

In order to complete the proof, we need to identify the subspace $\Hom\big((M,V),(N,V)\big) \subseteq  O_E\otimes_{\mbZ_p}\breve \mbZ_p$. In case $N = M(\mbX)$ this is $ O_E\otimes_{\mbZ_p} 1$. A generator as $ O_E$-module is given by the identity and the above shows
$$\Hom_{ O_F}(\mbX_0,\mbX) \iso ( O_E,\ δ_{E_0/F_0}^{-1} N_{E/E_0}).$$
To understand the case $N = M(\ov \mbX)$ we identify $\{j\colon E^u\to \mbC_p\}$ with $\mbZ/2f\mbZ$, where $2f = [E^u:\mbQ_p]$, such that $σ\circ j = j+1$ and such that $φ'_0\vert_{E^u} = 0$. Then $\Hom\big((M,V),(N,V)\big)$ is generated by $x = (x_j)_{j\in \mbZ/2f} \in  O_E\otimes_{\mbZ_p}\breve \mbZ_p$ with
$$x_j = \begin{cases}
π & \text{if }j\in \{0,\ldots,f-1\},\\
1 & \text{if }j\in \{f,\ldots,2f-1\}.\end{cases}$$	
It follows that
$$\Hom_{ O_F}(\mbX_0,\ov \mbX) \iso ( O_E,\ πδ_{E_0/F_0}^{-1} N_{E/E_0})$$
as claimed.
\end{proof}
Combining the proposition with the comparison isomorphism \eqref{eq comp herm lattices} and Gross' formula one obtains the
\begin{corollary}\label{cor local defo length}
For $u\colon \mbX_0\to \ov\mbX$ a homomorphism,
$$\mr{len}\ \mcO_{\mcZ(u)} = \frac{1 + v \big( δ_{E_0/F_0} (u,u) \big)}{2}.$$
\end{corollary}

\subsection{Non-archimedean components (globally)}
\label{subsect deg nonarch}
We return to the global setting. Let $ζ\in E_0^\times$, let $ν$ be a non-archimedean place of $\bfE$ not dividing $\mfd$ and set $v := ν\vert_{E_0}$. Also let $\mbF$ denote an algebraic closure of the residue field $\mbF_ν$.
\begin{proposition}
For $(A_0,A,\ov{η})\in \mcC(\mbF)$ the following hold.
\begin{enumerate}
\item If $v$ is split in $E$, then $V(A_0,A) = 0$.
\item If $v$ is inert in $E$, then $V(A_0,A) \iso V^{(v)}$.
\end{enumerate}
\end{proposition}
\begin{proof}
Recall that $\bfE$ is by definition a subfield of $\mbC$. Choose an isomorphism $\mbC \iso \ov \mbQ_p$ which induces the prime $ν$. This allows to view $Φ$ and $Φ'$ as $p$-adic CM types. Let $w$ be an arbitrary prime of $E$ over $p$ and let $w_0$, $u$ and $u_0$ denote its restrictions to $E_0$, $F$ and $F_0$, respectively. The isotypic factors $A_0[u^\infty]$ and $A[w^\infty]$ have to be isoclinic $p$-divisible groups since they have CM by a field. Their slopes are necessarily
\begin{equation}\label{eq slopes CM}
λ_{u}:=
\frac{\# Φ \cap \Hom(F_u, \ov \mbQ_p)}
{\#\Hom(F_u, \ov \mbQ_p)}
\ \ \mathrm{resp.}\ \ 
λ_w:=
\frac{\# \left(\{\ov{φ'_0}\}\cup Φ'\setminus \{φ'_0\}\right) \cap \Hom(E_w, \ov \mbQ_p)}
{\#\Hom(E_w, \ov \mbQ_p)}.
\end{equation}
In general one has
$$\frac{\# Φ\cap \Hom(F_u, \ov \mbQ_p)}{\#\Hom(F_u, \ov \mbQ_p)} = \frac{\# Φ'\cap \Hom(E_w, \ov \mbQ_p)}{\#\Hom(E_w, \ov \mbQ_p)}$$
since $Φ'$ is the inverse image of $Φ$ under the restriction map $\Hom(E,\mbC)\to \Hom(F,\mbC)$. One deduces the following properties.
\begin{itemize}
\item If $w/w_0$ is inert, then $u/u_0$ is also inert and $λ_w = λ_u = 1/2$.
\item If $w/w_0$ is split and $w_0\neq v$, then $λ_w = λ_u$ and $λ_w + λ_{\ov{w}} = 1$. Indeed, the addition of $φ'_0$ and removal of $\ov{φ'_0}$ in \eqref{eq slopes CM} does not come in since $φ'_0$ lies above $v$.
\item If $w/w_0$ is split and $w_0 = v$, then
$$\# \left(\{\ov{φ'_0}\}\cup Φ'\setminus \{φ'_0\}\right)\cap \Hom(E_w, \ov \mbQ_p) = \# Φ'\cap \Hom(E_w, \ov \mbQ_p) \pm 1.$$
In particular, $λ_w\neq λ_{u}$ and one obtains statement (1) of the proposition.
\end{itemize}
Now assume that $v$ is inert. The first aim is to show $V(A_0, A) \neq 0$. The above computation of slopes shows that $A_0[v^\infty]$ and $A[v^\infty]$ are supersingular. In other words, they are (isomorphic to) the $ O_F$-$\mbZ_p$-module $\mbX_0$ and $ O_E$-$\mbZ_p$-module $\ov{\mbX}$ from Prop.~ \ref{prop herm lattice locally}, respectively. It follows that there is an $ O_E$-linear isogeny $ O_E\otimes_{ O_F} A_0[v^\infty]\to A[v^\infty]$. It extends to an isogeny $ O_E\otimes_{ O_F} A_0[p^\infty]\to A[p^\infty]$ by the above computations of slopes.

Set $A'_0:= O_{E}\otimes_{ O_F} A_0$ which is an abelian varieties over $\mbF$ with CM by $ O_{E}$ of CM type $ Φ'$. The just given argument shows that $A'_0[p^\infty]$ and $A[p^\infty]$ are $ O_{E}$-linearly isogeneous. This does not yet yield that $A'_0$ and $A$ are isogeneous, but it shows at least that $A'_0$ is isogeneous to some $A'$ with CM by $ O_{E}$ of CM type $Φ$. Then $A'$ together with its CM deforms to characteristic $0$. But over $\mbC$, any two $ O_{E}$-CM abelian varieties with same CM type are isogeneous. This finally proves that $A'_0$ and $A$ are isogeneous and hence $V(A_0,A) \neq 0$.

For dimension reasons one obtains that
$$\mbQ_\ell \otimes_{\mbQ} V(A_0,A) \iso
\begin{cases}\Hom_{F_\ell}(\RV_\ell(A_0), \RV_\ell(A)) & \ell \neq p\\
\Hom^\circ_{F_p}(A_0[p^\infty],A[p^\infty])& \ell = p.\end{cases}$$
In particular $V(A_0,A)$ has to be of dimension $1$ over $E$. It is positive definite by the positivity of the Rosati involution. 
Its localizations at $\ell \neq p$ are isomorphic to $V_\ell$ (viewed as $1$-dimensional $E$-vector space) by the existence of a level structure.

This leaves us to check $V^{(v)}_{w_0}\iso V(A_0,A)_{w_0}$ for $w_0$ above $p$. If $w_0$ is inert in $E$ with $u_0 := w_0\vert_{F_0}$, the statement $\Hom(A_0[u_0^\infty],A[w_0^\infty]) = V^{(v)}_{w_0}$ is precisely Prop.~ \ref{prop herm lattice locally}. (Both cases of the proposition occur, depending on whether $v = w_0$ or not.) For $w_0$ split in $E$, there is nothing to show since there is only a single isomorphism class of $1$-dimensional hermitian spaces.
\end{proof}

\emph{Proof of Thm. \ref{thm CM inter} (1) and (2).}
Part (1) is implied by the previous proposition: $\mbF\otimes_{ O_\bfE[\mfd^{-1}]}\mcC(ζ,ϕ) \neq \emptyset$ only if $ζ$ is represented by $V^{(v)}$, which can only happen if $v$ is inert and $\Diff(ζ,\mbV) = \{v\}$.

Assume this to be the case, and let $c = (A_0,A,\ov{η},u)\in \mcC(ζ,ϕ)(\mbF)$ be any point. The completed local ring of $\mcC(ζ,ϕ)$ in $c$ is $ O_{\breve \bfE_ν}$, which is finite free of degree equal to the ramification index $e(ν/v)$ over $ O_{\breve E_{v}}$. Applying the Serre--Tate Theorem and Cor.~ \ref{cor local defo length}, one obtains
$$\mr{len}\ \mcO_{\mcC(ζ,ϕ),c} = e(ν/v)\cdot \frac{1 + v(δ_{E_0/F_0} ζ)}{2}.$$
The length is in particular independent of $c$ and all that is left to do is to count the points of $\mcC(ζ,ϕ)(\mbF)$.

For $c$ as above, the level structure yields a coset of identifications
$$σ^{-1}(K^p) η^p \colon  \mbA_{E,f}^p\otimes_{E} V(A_0, A) \iso \mbV_f^p.$$
By Prop.~ \ref{prop herm lattice locally}, it can naturally be extended to an identification
$$σ^{-1}(K^v)η^{v}\colon  \mbA_{E,f}^{v}\otimes_{E} V( A_0, A)\iso \mbV_f^v$$
by looking at $p$-divisible groups. By ``natural'' we mean that it is such that $x\in V( A_0, A)$ lies in $\Hom(A_0,A)[\mfd^{-1}]$ if and only if $η^{v,\mfd}(x) \in \widehat {Λ}^{v,\mfd}$ and if it is a homomorphism at $v$.
The latter is equivalent to $v(δ_{E_0/F_0}ζ) > 0$ by Prop.~ \ref{prop herm lattice locally}. It follows that for a quasi-homomorphism $x\in V(A_0,A)_ζ$, the length of $\mcC(ζ,ϕ)$ at the point $(A_0,A,\ov{η},x)$ is
$$e(ν/v) ϕ^{v}(η^{v}(x)) \max\left\{0, \frac{1+v(δ_{E_0/F_0} ζ)}{2}\right\}.$$
The set of such $x$ form a simply transitive $E^1$-orbit while $\mcC(ζ,ϕ)(\mbF)$ is a simply transitive $T(\mbQ)\backslash T(\mbA_f)/σ^{-1}(\wt K)$-orbit. So
\begin{equation}\begin{split}
\widehat{\deg}_ν\ \mcC(ζ,ϕ)  = &\, τ(Z^\mbQ) e(ν/v)\,\max\left\{0,\frac{1 + v(δ_{E_0/F_0} ζ)}{2}\right\}\cdot  \log q_ν  \\
& \ \ \times\sum_{t\in E^1 \backslash \mbA^1_f/σ^{-1}(K)}\ \sum_{x\in V(A_0,A)_{ζ}} ϕ^{v}(tx) \\
= &\, τ(Z^\mbQ) e(ν/v)\, \max\left\{0,\frac{1 + v(δ_{E_0/F_0} ζ)}{2}\right\} \cdot \Orb(ζ,ϕ^{v}) \cdot \log q_ν.
\end{split}	
\end{equation}
Summing over all $ν\vert v$ one obtains the claimed expression
\begin{equation}
\widehat{\deg}_{v}\ \mcC(ζ,ϕ) = τ(Z^\mbQ)[\bfE:E_0] \cdot\max\left\{0,\frac{1 + v(δ_{E_0/F_0} ζ)}{2}\right\} \cdot \Orb(ζ,ϕ^{v}) \log q_{v}.
\end{equation}
\qed

\section{Comparison with analytic generating series}

\subsection{The analytic generating series}
\label{ss:analytic_gen_series}
Consider the split $E_0$-quadratic space $V'=E_0\times E_0$ with its quadratic form $u'= (u_1,u_2)\mapsto \fkq(u')=u_1u_2$. The special orthogonal group $\SO(V',\fkq)$ can be identified with the $E_0$-group $G':=\GL_{1,E_0}$, via the action on the $V'$ by $g\cdot (u_1,u_2)= (g^{-1}u_1,gu_2)$. 

Denote by $χ\colon \BA_{E_0}^\times/E_0^\times\to \{\pm 1\}$ the quadratic character associated to $E/E_0$ by class field theory. 
For $\phi'\in \CS(V'(\BA_{E_0}))$ we
recall from \cite[\S12]{Z19} that one defines a regularized integral
\begin{align}\label{Zan}
J(\phi', s)=\int_{[G']}\left(\sum_{u'\in V'(E_0)} \phi'(g^{-1}\cdot u')\right)|g|^s\chi(g)\,dg.
\end{align}
It decomposes  as a sum 
\begin{align}\label{Zan0}
J(\phi', s)=\Orb(0_\pm, \phi',s)+\sum_{\zeta=\fkq(u')\in E_0^\times}\Orb(\zeta, \phi',s),
\end{align}
where for $\zeta=\fkq(u'_\zeta)\neq 0$,
\begin{align}\label{Zan1.5}
\Orb(\zeta, \phi',s):=\int_{G'(\BA_{E_0})} \phi'(g^{-1}\cdot u'_\zeta) |g|^s\chi(g)\,dg.
\end{align}
(Note that this notation slightly differs from \cite[\S12]{Z19}, where it was denoted by $\Orb(u'_\zeta, \phi',s)$.) We refer to {\em loc. cit.} for the definition of nilpotent orbital integrals 
$\Orb(0_\pm, \phi',s):=
\Orb(0_+, \phi',s)+
\Orb(0_-, \phi',s)$.

Define the analytic generating function,
$$J(h,\phi', s):=J(\omega(h)\phi', s),\quad h\in \bH(\BA_{E_0}).
$$
We recall from \cite[Thm.~12.9]{Z19} for future reference.
\begin{theorem}\label{thm n=1 gl}
The function $\bH(\BA_{E_0})\times \BC\ni (h,s) \mapsto J(h,\phi', s)$ is smooth, left $\bH(E_0)$-invariant, and entire in $s\in\BC$.
\end{theorem}

\begin{remark}\label{rem SE}  
The integral \eqref{Zan} can be viewed as the theta lifting for the pair $$(\SO(V',\fkq),\quad\SL_2),$$ from the automorphic representation $\chi|\cdot|^s$ of $\SO(V')\simeq \GL_{1}$ to $\SL_2$ (cf. \cite[Rem.~12.10]{Z19}). Therefore,  $ J(h,\phi', s)$ should be a degenerate Eisenstein series for the induced representation $\Ind_{B(\BA_{E_0})}^{\bH(\BA_{E_0})}(\chi\, |\cdot|^s)$ ($B$ the Borel subgroup of upper triangular matrices). However, it may not be associated to a standard section in our application below.\end{remark}

\label{subsect eis series}
\subsection{Local Results}
\label{ss:Eis_local}

Let  $E_0$ be a local field and $E/E_0$ a (possibly split) quadratic extension. Denote by $χ\colon E_0^\times\to \{\pm 1\}$ the corresponding character from local class field theory. For $\zeta\in E_0^\times$, we 
define the normalized local orbital integral
with a transfer factor,
$$
\Orb(\zeta,\phi',s)=\chi(u_1)|u_1|^{-s}\int_{E_0^\times}\phi'(u_1 t^{-1},t u_2)|t|^s\chi(t)\,d^\times t
$$
where $(u_1,u_2)\in E_0\times E_0$ is any choice such that $\zeta=u_1u_2$. Set 
$$
\Orb(\zeta,\phi')=\Orb(\zeta,\phi',0),\quad \del(\zeta,\phi')=\frac{d}{ds}\Big|_{s=0} \Orb(\zeta,\phi',s).
$$
For $h\in \bH(E_0)$, we set 
$$
\Orb(\zeta,h,\phi',s):=\Orb(\zeta,\omega(h)\phi',s)
$$
and similarly for the other orbital integrals.

Now 
let $E_0$ be a local $p$-adic field and $q$ the cardinality of its residue field. Let $E/E_0$ be an unramified quadratic field extension. Let $v$ be the normalized valuation on $E_0$ and $|\cdot | = q^{v(\cdot)}$ the normalized absolute value. For each $c\in E_0^\times$, consider the characteristic function
\begin{align}\label{eq:phi'c}
\phi'_c={\bf 1}_{  O_{E_0}\times c \cdot O_{E_0} }\in \CS(V'(E_0)).
\end{align}
Normalize the Haar measure on $E_0^\times$ such that $\vol(O_{E_0}^\times)=1$.

\begin{proposition}\label{prop orb non-arch}
We have
$$
\Orb(\zeta,\phi'_c) =\begin{cases}
1 & \text{if}\ v(c^{-1}\zeta)\geq 0\ \text{and}\ \equiv 0\!\!\mod 2, \\
0& \quad\text{otherwise}.
\end{cases}
$$
When $v(\zeta)- v(c)$ is odd,
\begin{align}\label{eq FC of derivative key identity}
\del(\zeta,\phi'_c) =-
\max\left\{0, \frac{1+v(c^{-1}ζ)}{2}\right\}  \log q.
\end{align}
\end{proposition}
\begin{proof}
This follows by a straightforward computation.
\end{proof}

We recall from \cite[\S12]{Z12} the analogous result for the archimedean extension $E\mbC$ of $E_0=\mbR$,  the standard character $ψ_0(b) := \exp(2πib)$.
Now we have $V'=\BR\times \BR$ and let
\begin{align}\label{eq:Gau}
\phi'(x,y)=2^{-1}(x+y)e^{-\pi(x^2+y^2)}\in \CS(\BR\times \BR).
\end{align}
 Write $h\in \SL_2(\BR)$ according to the Iwasawa decomposition (cf. \eqref{kappa in SO2})
\begin{align*}
h=\left(\begin{matrix} 1 & b \\
& 1
\end{matrix}\right)
\left(\begin{matrix} a^{1/2} & \\
& a^{-1/2}
\end{matrix}\right)\,\kappa_\theta,\quad a\in\BR_{+},\quad b\in \BR.
\end{align*}
Normalize the Haar measure on $E_0^\times=\BR^\times$ to be $dt/|t|$ where $dt$ is the Lebesgue measure on $\BR$. We recall the following result from \cite[Lem.~12.5]{Z12}.\footnote{Here the Gaussian functions differ slightly from \cite[\S12]{Z19} which should be corrected as the ones here.}
\begin{proposition}\label{prop orb inf}
Let $\zeta\in\BR^\times$. Then
\begin{align*}
\Orb(\zeta,h,\phi')=\chi_1(\kappa_\theta)\begin{cases}a^{1/2}e^{2 \pi i \zeta(b+ia)} ,& \zeta>0,\\
0, & \zeta<0,
\end{cases}
\end{align*}
and when $\zeta<0$,
\begin{align}\label{eq FC of derivative key identity inf}
\del(\zeta,h,\phi')=\frac{1}{2}\chi_1(\kappa_\theta)a^{1/2} \,e^{2\pi i  \zeta(b+ia )}\,\Ei(-4\pi a|\zeta|).
\end{align}
Here the character $\chi_1$ is defined by \eqref{chi}.
\end{proposition}

\subsection{Comparison}
From now on $F_0,F,E_0$ and $E$ again denote the global fields as in previous sections. 
Given an element $c\in (\mbA^{\mfd}_{E_0,f})^\times$, denote by $\mbV_c$ the rank-$1$ adelic hermitian space
$$\mbV_c := V_\mfd \times (\mbA^{\mfd}_{E,f}, c x\ov y) \times \mbV_\infty$$
where $\mbV_\infty$ is positive definite at all archimedean places, and $V_\mfd$ is the base change to $F_{\mfd}$ from the $F/F_0$-hermitian space $V$ fixed in \S\ref{ss:integral models}, which is now viewed as an $E/E_0$-hermitian space (see the paragraph after the proof of Prop. \ref{prop CM intersection decomp}). Define a Schwarz function $ϕ_c\in \mcS(\mbV_c)$ as
$$ϕ_c := ϕ_\mfd\otimes 1_{\wh {O_E}}^\mfd \otimes ϕ_\infty$$
where $ϕ_\mfd$ is from the definition \eqref{eq:KR gen} of the KR-generating series and where $ϕ_\infty$ is the Gaussian 
$$ϕ_\infty(x) =e^{-2π\,\mr{tr}_{E_0/\mbQ}\pair{x, x}}.$$
For $v|\infty$, we normalize the local Haar measure on $E^1_{v}$ by $\vol(E^1_{v})=1$.

Choose $\phi'=\prod_{v}\phi'_v$ such that for every $v$, $\phi'_v$ is a transfer of $\phi_v$ (and zero function on the other isomorphism class of hermitian space), i.e., they satisfy
$$
\Orb(\zeta,\phi'_v)=\begin{cases} \Orb(\zeta,\phi_v),& \text{if $\BV_v$ represents $\zeta$},\\
0, &\text{otherwise}.
\end{cases}
$$

\begin{proposition}\label{prop orb V}
\begin{enumerate}
\item
Let $v\nmid\fkd$ and $v<\infty$. Then
\begin{align}\label{eq:orb phi non-arch}
 \Orb(ζ,ϕ_{v})=\begin{cases}
1 & \text{if}\ v(δ_{E_0/F_0}ζ)\geq 0\ \text{and}\ \equiv 0\!\!\mod 2, \\
0& \quad\quad\text{otherwise}.
\end{cases}
\end{align}

\item
Let $v\mid\infty$. Then, for  $
h=\left(\begin{matrix} 1 & b \\
& 1
\end{matrix}\right)
\left(\begin{matrix} a^{1/2} & \\
& a^{-1/2}
\end{matrix}\right)\,\kappa_\theta\in  \SL_2(\BR)$ (cf. \eqref{kappa in SO2}),  we have
\begin{align}\label{eq:orb phi arch}
 \Orb(ζ,h_v,ϕ_{v})=\chi_1(\kappa_\theta)\begin{cases}
a^{1/2}e^{2 \pi i \zeta(b+ia)}& \text{if}\ ζ\geq 0, \\
0& \text{otherwise}.
\end{cases}
\end{align}
\end{enumerate}
\end{proposition}\begin{proof}
This follows by a straightforward computation.
\end{proof}

  By Prop.~\ref{prop orb V} \eqref{eq:orb phi non-arch} and Prop.~\ref{prop orb non-arch} 
 we can and do choose $\phi'_v$ as $\phi'_{c_v}$ defined by \eqref{eq:phi'c} when $v$ is unramified (including the split case which is easy to verify); by Prop.~\ref{prop orb V} \eqref{eq:orb phi arch} and Prop.~\ref{prop orb inf}  we can and do choose $\phi'_v$ as \eqref{eq:Gau} when $v|\infty$.

Let $δ_{E_0/F_0}$, $δ_{E_0/\mbQ}$ and $δ_{F_0/\mbQ}$ denote the different ideals of the indicated field extensions; we will also view them as elements in the respective rings of ideles. We now apply the previous considerations in the case of the incoherent adelic hermitian space $\mbV$ from Thm.~ \ref{thm CM inter}. It is of the form $\mbV_c$ for $c\in \mbA^{\mfd,\times}_{E_0,f}$ satisfying
$$c\,\wh {O_{E_0}}{}^\mfd = δ_{E_0/F_0}^{-1}\wh {O_{E_0}}{}^\mfd.$$
The Schwartz function $ϕ$ considered in the same theorem agrees with $ϕ_c$ for a suitable choice of isomorphism $\mbV \iso \mbV_c$.

Set
\begin{equation}\label{eq def an series}
\delJ(h,\phi') := \frac{d}{ds}\Big|_{s=0} J(h,\phi',s),\quad h\in \bH(\BA_{E_0}).
\end{equation}
Write
$$\delJ(h,\phi') = \sum_{ζ\in E_0} \delJ(\zeta,h,\phi') $$
for its Fourier expansion (cf. \eqref{eq:def F exp}). 
\begin{theorem}\label{thm comparison CM cycle}
\begin{enumerate}[wide, labelindent=0pt, labelwidth=!, label=(\arabic*), topsep=2pt, itemsep=2pt]
The following statements hold true.
\item For all $0\neq ζ$ with $\Diff(ζ,\mbV)=\{v\}$ a singleton and $v\nmid \mfd$, we have equalities in $\BR$:
\begin{equation}
2τ(Z^\mbQ)[\bfE:E]\cdot  \delJ(\zeta,h_\infty,\phi') =- \widehat{\deg}\ \widehat{\mcC}^\bK(ζ,h_\infty,ϕ) W_\zeta^{(1)}(h_\infty).
\end{equation}
\item If instead $0\neq \zeta$ with $\Diff(ζ,\mbV)=\{v\}$ and $v\mid \mfd$, then $\delJ(\zeta,h_\infty,\phi')=d_\zeta \, W_\zeta^{(1)}(h_\infty)$ for a constant $d_\zeta$ such that
$$2τ(Z^\mbQ)[\bfE:E]\cdot  d_\zeta  \in \BQ\log q_v\subseteq\sum_{\ell\mid \mfd} \mbQ\log \ell.$$
\item When $\zeta=0$, there is a constant $d_0^{\mr{corr}}\in \mbR$ such that
$$2τ(Z^\mbQ)[\bfE:E]\cdot   \delJ(0,h_\infty,\phi') = [d_0^{\mr{corr}} -\widehat {\deg}\ \widehat{\mcC}^\bK(0,h_\infty,ϕ)]W_0^{(1)}(h_\infty).$$
\end{enumerate}
\end{theorem}
\begin{proof}By the choice of $\phi'$ and $\phi$, we have $\delJ(\zeta,h_\infty,\phi')\neq 0$ only for $ζ = 0$ or $ζ$ with $\Diff(ζ,\mbV) = \{v\}$ a singleton.  

We first consider the case $\Diff(\mbV,ζ) = \{v\}$ with $v\nmid \mfd$ non-archimedean. Then $v$ is necessarily inert in $E$ and $v(δ_{E_0/F_0}ζ)$ is odd. If $v(δ_{E_0/F_0}ζ) \leq 0$, then both sides in statement (1) vanish by Thm.~ \ref{thm CM inter} and Prop.~ \ref{prop orb non-arch}. In case $v(δ_{E_0/F_0}ζ)\geq 1$ we would like to apply the identity \eqref{eq FC of derivative key identity} to $c=\delta_{E_0/F_0}^{-1}$ and to obtain
\begin{align*}
\delJ(\zeta,h_\infty,\phi')=&\del(\zeta,\phi'_v) \Orb(\zeta,h_\infty,\phi'^{v})
\\=&-\max\left\{0, \frac{1+v(\delta_{E_0/F_0}ζ)}{2}\right\} \Orb(\zeta,\phi'^{v}) \log q_v. 
\end{align*}  Since $\zeta$ is totally positive, by Prop.~\ref{prop orb V} \eqref{eq:orb phi arch}, the archimedean components give us
\begin{align}\label{eq:Orb2Wh}  
 \prod_{v|\infty}\Orb(\zeta,h_v,\phi'_v)= W_\zeta^{(1)}(h_\infty).
\end{align}
Taking into account Thm.~ \ref{thm CM inter} (2) \eqref{eq Arakelov deg non-archimedean}, statement (1) (for non-archimedean $v$) follows.

We now assume $v$ to be archimedean.  Then $\zeta$ is negative exactly at $v$.
By \eqref{eq FC of derivative key identity inf},  we have
\begin{align*}
\delJ(\zeta,h_\infty,\phi')=&\del(\zeta,h_v,\phi'_v) \Orb(\zeta,h_\infty^v,\phi{}'^{v})
\\=& \frac{1}{2}a_v^{1/2} \,e^{2\pi i \, v(\zeta)(b_v+ia_v )}\,\Ei(4\pi a_vv(\zeta)) \Orb(\zeta, h_\infty^v,\phi'^{v}). 
\end{align*}
Now statement (1) (for archimedean $v$) follows from Thm.~\ref{thm CM inter} (3) and Prop.~\ref{prop orb V} \eqref{eq:orb phi arch}.

Coming to Part (2), we now assume $\Diff(\mbV,ζ) = \{v\}$ with $v\mid \mfd$. Then we still have 
 \begin{align*}
\delJ(\zeta,h_\infty,\phi')=&\del(\zeta,\phi'_v) \Orb(\zeta,h_\infty,\phi'^{v}).
\end{align*}
By \eqref{eq:Orb2Wh}, we may set $d_\zeta=\del(\zeta,\phi'_v) \Orb(\zeta,\phi'^{v,\infty})$.
 Clearly, $\Orb(\zeta,\phi'_{v},s)$ is a polynomial in $q_{v}^s$ and $q_v^{-s}$ with $\mbQ$-coefficients. Hence
$$\del(\zeta,\phi'_{v})\in \mbQ\log q_{v}.$$
Similarly we have  $\Orb(\zeta,\phi'^{v,\infty})\in \mbQ$ 
 and hence the term $ d_\zeta$ lies in $\mbQ\log q_{v}$.

We come to the claim (3) on the $0$-th Fourier coefficient. This essentially follows from the last part of the proof of \cite[Prop.~14.5]{Z19}. More precisely, by {\em loc. cit.} we have
$$
\delJ(0, h_\infty, \phi')=\left(\Orb(0_+, \phi')\ \sum_{v\colon E_0\to \mbR} \log |a_{v}|+c\right)W_0^{(1)}(h_\infty),
$$
for some constant $c\in \BR$, and
$$\Orb(0_+, \phi')=-\Orb(0_-, \phi')=\frac{1}{2}\tau(E^1)\phi(0).
$$
(This last identity was a result of Jacquet.) We obtain
$$  \delJ(0,h_\infty,\phi') = \left(\frac{1}{2}  \tau(E^1) \cdot \phi(0)\ \sum_{v\colon E_0\to \mbR} \log |a_{v}|+c\right) W_0^{(1)}(h_\infty).$$
By Thm.~\ref{thm CM inter} (4), statement (3) is proven.
\end{proof}

\begin{lemma}\label{lem:Weil KH}
Let $v\nmid\fkd$ be a non-archimedean place of $F_0$ (in particular $v$ is unramified in $F/F_0$).
\begin{enumerate}[wide, labelindent=0pt, labelwidth=!, label=(\arabic*), topsep=2pt, itemsep=2pt]
\item  Let $\phi'_v=\prod_{w|v, w\in\Sigma_{E_0}}\phi'_w$ be as above. Then $\phi_v'$ is  $K_{\bH,v}^{\circ}$-invariant under the Weil representation.  

\item  Let $\phi_v=\prod_{w|v, w\in\Sigma_{E_0}}\phi_w$ be as above. Then  $\phi_v$ is  $K_{\bH,v}^{\circ}$-invariant under the Weil representation.  
\end{enumerate}
Here $K_{\bH,v}^{\circ}$ is defined by \eqref{eq:max open}, and the group $\bH(F_{0,v})$ is viewed as a natural subgroup of $\bH(E_0\otimes_{F_0}F_{0,v})$. 

\end{lemma}
\begin{proof}We only prove (1) since the proof applies verbatim to (2).
Note that the group  $K_{\bH,v}^{\circ}$ is generated by the following matrices 
$$m(a):=\left(\begin{matrix} a &  \\
& a^{-1}
\end{matrix}\right),\quad 
n(b):=\left(\begin{matrix} 1 & b \\
& 1
\end{matrix}\right), \quad n^-(b'):=\left(\begin{matrix} 1 &  \\
b'& 1
\end{matrix}\right),
$$where $ a\in O_{F_0,v}^\times,  b\in \delta_{F_0/\BQ,v}^{-1},b'\in \delta_{F_0/\BQ,v}.$ The invariance under such $m(a)$ is clear.  Note that our function $\otimes_{w|v}\phi'_{w,c_w}$ defined by \eqref{eq:phi'c} for a generator $c_w\in E_{0,w}^\times$ of the fractional ideal $\delta_{E_0/F_0,w}^{-1} $ is then invariant under $n(b)$ when $\psi_v(\tr_{E_0/F_0}(b\,(c_w)_{w|v}))=1$. This equality is equivalent to $\psi_{\BQ_p}(\tr_{E_0/\BQ}(b\,(c_w)_{w|v})))=1$, which holds if $b\delta_{E_0/F_0,w}^{-1} \subseteq \delta_{E_0/\BQ,w}^{-1}$ for every $w|v$, i.e., if $b\in \delta_{F_0/\BQ,v}^{-1}$.
The invariance for such $n^-(b')$ is proved similarly by considering the Fourier transform $\wh{\phi'_{w,c_w}}$, which is ${\bf 1}_{ c_w^{-1} O_{E_0,w}\times  O_{E_0,w} }$ up to a constant multiple. 
\end{proof}

\begin{proposition}\label{prop:Weil del J em}
Let $\phi'=\phi'_\fkd\otimes\phi'^{\fkd}$ be as above and suppose that $\phi'_\fkd$ is $K_{\bH,\fkd}$-invariant under the Weil representation. Then the function $h\in \bH(\BA_0)\mapsto\delJ(h, \phi')$  lies in $\CA_{\infty}(\bH({\BA_0}),K_\bH,n)$ with $K_\bH=K_{\bH,\fkd}\times K_{\bH}^{\fkd,\circ}$.
\end{proposition}
\begin{proof}By Thm.~\ref{thm n=1 gl}, $\delJ$ is smooth. The weight-$n$ condition follows from Prop.~\ref{prop orb inf}. The $K_\bH$-invariance follows from Lem.~\ref{lem:Weil KH}.
\end{proof}

\begin{proof}[Proof of Thm. \ref{thm elementary CM modularity}.]
Given $ϕ\in \mcS(V(\mbA_{0,f}))$ as in Thm. \ref{thm elementary CM modularity}, choose a matching function $ϕ'\in \mcS(V'(\mbA_{E_0}))$ as in Thm. \ref{thm comparison CM cycle}. Define
$$f(h) := -2τ(Z^\mbQ)[\bfE:E]\cdot \delJ(h, ϕ'),\ \ \ h\in \bH(\mbA_0).$$
Then $f \in \CA_{\infty}(\bH({\BA_0}),K_\bH,n)$ by Prop. \ref{prop:Weil del J em} and this is also the space that appears in Thm. \ref{thm elementary CM modularity}. Restricting to $\bH(F_{0,\infty})$, Thm. \ref{thm comparison CM cycle} provides the Fourier expansion
$$f(h_\infty) = \sum_{ξ\in F_0} \left(ε_ξ + \sum_{\mr{tr}_{E_0/F_0}(ζ) = ξ} \widehat{\deg}\ \widehat{\mcC}^\bK(ζ,h_\infty,ϕ) \right) \cdot W_ξ^{(n)}(h_\infty)$$
for correcting scalars $ε_ξ \in \sum_{\ell\mid \mfd} \mbQ\log \ell$ when $ξ \neq 0$, resp. $ε_0\in \mbR$. The inner sum equals
$$\left\langle \widehat{\mcZ}^\bK(ξ,h_\infty,ϕ),\ \mcC \right\rangle$$
by Prop. \ref{prop CM intersection decomp} and the proof of Thm. \ref{thm elementary CM modularity} is complete.
\end{proof}

\part{Application to the AFL}

\section{Orbital integrals}

In this paper, we mainly consider the semi-Lie algebra version of the AFL and hence require the analytic notions from \cite{Z19} only in this setting. We refer to \cite[\S2]{Z19} and \cite{Z14} for the group setting and its relation with the definitions here.

\subsection{Orbit matching}\label{ss: orb match}
Consider the symmetric space
\begin{equation}\label{Sn def}
     S_{n} := \{\,γ\in \Res_{F/F_0}\GL_n\mid γ\ov {γ}=1_n\,\},
\end{equation}
and the $F_0$-vector space
\begin{equation}\label{V' def}
V'_{n}=F_0^{n}\times (F_0^{n})^\ast,
\end{equation} 
where $(F_0^{n})^\ast=\Hom_{F_0}(F_0^n,F_0)$ denotes the $F_0$-linear dual space. For convenience we will identify $F_0^{n}$ (resp. $(F_0^{n})^\ast$) with the space of column vectors (resp. row vectors). With the tautological pairing,
we will view $V_n'$ as an $(F_0\times F_0)/F_0$-hermitian space.
Let
\[
   G' :=\GL_{n,F_0},
\]
and consider the (diagonal) action of $G'$ on the product $S_{n}\times V'_{n}$,
$$
h\cdot(\gamma, (u_1,u_2))=(h^{-1} \gamma h, ( h^{-1}u_1,u_2 h)).
$$
Next, whenever $V$ is an $F/F_0$-hermitian space of dimension $n$, we let $\U(V)$ act diagonally on the product $\U(V)\times V$. (Here, we view $V$ as a $2n$-dimensional vector space over $F_0$.)

There is  a natural bijection of orbit spaces of \emph{regular semisimple} elements,  cf. \cite[\S2.2]{Z19},
\begin{align}\label{eq:orb mat 1}
 \xymatrix{\coprod_{V} \bigl[(\U(V)\times V)(F_0) \bigr]_\rs	   \ar[r]^-\sim&  [(S_{n}\times
V'_{n})(F_0)]_\rs},
\end{align}
where the disjoint union runs over the set of isometry classes of $F/F_0$-hermitian spaces $V$ of dimension $n$. Here the left (resp. right) hand side denotes the orbits under the action of the group $\U(V)(F_0)$ (resp. $\GL_{n}(F_0)$). The bijection is called the {\em matching relation} between regular semisimple orbits and denoted by $(g,u)\leftrightarrow (γ,u')$.

For $\alpha\in F[t]^\circ_{\deg n}$, we denote by $S_{n}(\alpha)$ the subscheme of $S_n$ consisting of elements with characteristic polynomial $\alpha$. For $\xi\in F_0$, we let $V_{n,\xi}'$ denote the subscheme of $V_{n}'$ defined by $u_2u_1=\xi$. We denote by $[(S_{n}(\alpha)\times V_{n}')(F_0)]$ and $[(S_{n}(\alpha)\times V_{n,\xi}')(F_0)]$ the set of $\GL_{n}(F_0)$-orbits. Similar notation applies to unitary groups.

\subsection{Orbital integral matching: Smooth transfer}\label{ss: transfer}

We recall the definition of orbital integrals from \cite[\S11]{Z19}, also cf. \cite[\S2.2]{RSZ2}. Now let $F/F_0$ be a quadratic extension of local fields of characteristic zero (the split $F=F_0\times F_0$ is similar and simpler). Let 
$$\xymatrix{\eta=\eta_{F/F_0}\colon F_0^\times\ar[r]&\{\pm 1\}}$$ 
be the quadratic character associated to $F/F_0$ by local class field theory. Here and thereafter we will also denote by $\eta$ the character   $\eta\circ \det$ of $\GL_n(F_0)$.

For $(\gamma, u') \in (S_{n} \times V_{n}')(F_0)_\rs$, $\Phi' \in \CS((S_{n} \times
V_{n}' )(F_0))$, and $s \in \BC$, we define the orbital integral
\begin{equation}\label{Orb(gamma,f',s)}
   \Orb((\gamma,u'),\Phi',s) := \int_{\GL_{n}(F_0)} \Phi'(h\cdot (\gamma, u')) \lv \det h \rv^s \eta(h)\, dh,
\end{equation}
with special values
\begin{align}\label{eq:orb s=0}
   \Orb((\gamma,u'), \Phi') &:=\omega(\gamma,u') \Orb((\gamma,u'), \Phi', 0),\ \ \text{and}\\
	\del((\gamma,u'), \Phi') &:= \omega(\gamma,u')\,\frac{d}{ds} \Big|_{s=0} \Orb((\gamma,u'), \Phi', s).
\end{align}
Here, $\omega(\gamma,u')\in \{\pm 1\}$ is the so-called transfer factor from \cite[(2.13)]{Z19}. We emphasize that, unlike in \cite{Z14}, the transfer factor is already included in the special values here.

On the unitary side, for
$(g,u)\in(\U(V)\times V)(F_0)_\rs$ and  $\Phi \in \CS((\U(V)\times V)(F_0))$, we define
\begin{align}\label{eq:orb U lie}
   \Orb((g,u),\Phi) := \int_{\U(V)(F_0)} \Phi(h\cdot (g,u)) \, dh.
\end{align}

For simplicity let us now assume that $F$ is non-archimedean. Then there are exactly two isometry classes of $F/F_0$-hermitian spaces with dimension $n$, denoted by $V_0$ and $V_1$. When $F/F_0$ is unramified, we will assume that $V_0$ has a self-dual lattice.  Then the orbit bijection \eqref{eq:orb mat 1} specializes to 
\[
 \xymatrix{ \bigl[(\U(V_0)\times V_0)(F_0) \bigr]_\rs \coprod \bigl[(\U(V_1)\times V_1)(F_0) \bigr]_\rs	   \ar[r]^-\sim&  [(S_{n} \times V_{n}')(F_0)]_\rs}.
\]
\begin{definition}\label{def st loc}
A function $\Phi' \in \CS((S_{n} \times
V_{n})(F_0))$ and a pair of functions $(\Phi_0,\Phi_1) \in \CS((\U(V_0)\times V_0)(F_0)) \times \CS((\U(V_1)\times V_1)(F_0))$ are   (smooth) \emph{transfers} of each other if for each $i \in \{0,1\}$ and each $(g,u) \in (\U(V_i)\times V_i)(F_0)_\rs$,
\begin{align}\label{eq:def st loc}
   \Orb((g,u),\Phi_i) = \Orb((\gamma, u'),\Phi')
\end{align}
whenever $(\gamma, u')\in \left(S_{n} \times
V_{n}'\right)(F_0)_\rs$ matches $(g,u)$.
\end{definition}
\begin{definition}\label{def part st}
For a fixed $V\in\{V_0,V_1\}$ and a regular semisimple $\alpha\in F[t]^\circ_{\deg n}$,  we say that $\Phi'$ partially (relative to $\alpha$) transfers to $\Phi\in \CS(\U(V)\times V)(F_{0})$, if we only require the equality \eqref{eq:def st loc} in Def.~ \ref{def st loc} to hold for matching orbits $(\gamma, u')\in (S_{n} (\alpha)\times
V'_n)(F_{0})_\rs$ and $(g,u)\in (\U(V)(\alpha)\times V)(F_{0})_\rs$; and $ \Orb((\gamma, u'),\Phi')=0$ for any other $(\gamma, u')\in (S_{n} (\alpha)\times
V'_n)(F_{0})_\rs$.
\end{definition}

Smooth transfers exist for $p$-adic fields and commute with the Weil representation (see \cite[Appendix A]{Z19}). 
For archimedean local fields, a specific partial transfer of the Gaussian test functions has been constructed in \cite[\S12]{Z19}.

\subsection{Modular analytic generating functions}
\label{ss: n>1}
Now we return to the global situation. Our goal  in this subsection is to define an analytic generating series (of orbital integrals) that will match the generating series defined later by intersection numbers. Our proof of the AFL conjecture will be based on a comparison of the two generating series.

Recall from \S\ref{ss:LCM} that we have fixed an irreducible $\alpha\in  F[t]^\circ_{\deg n}$ such that the field $E=F[t]/(\alpha)$ is a CM extension of a totally real subfield $E_0$.  Then $S_n(\alpha)(F_0)$ consists of exactly one $G'(F_0)$-orbit and we fix a representative $\gamma\in S_n(\alpha)(F_0)$.

Via the action of $E_0$ induced by $\gamma$, the vector space $V'_n$ carries the structure of a rank-one free $(E_0\times E_0)/E_0$-hermitian module. Let \begin{align}\label{eq: q' on V'}
\fkq'\colon 
\xymatrix@R=0ex{ V'_n \ar[r] &E_0}
\end{align}
be the associated quadratic form over $E_0$.

Now, for every $v\mid \infty$, we fix the archimedean $\Phi'_v\in\CS((S_n\times V')(F_{0,v}))$ to be the (partial) Gaussian test function constructed in \cite[\S12.4, (12.10), (12.11)]{Z19} relative to $\gamma$. 
For $h\in \bH(\BA_0), \Phi'=\otimes_v\Phi'_v\in \CS((S_n\times V')(\BA_0))$ with $\Phi_v'$ the fixed Gaussian test function for $v\mid\infty$, and $s\in\BC$, one defines a regularized integral in {\em loc. cit.}
\begin{align}\label{eq: def J s}
J(h,\Phi',s)=&\int_{[G']} \left(\sum_{(\gamma,u')\in (S_n(\alpha)\times V')(F_0)}r(g)\omega(h)\Phi'(\gamma, u')\right)|\det(g)|_{F_0}^s\eta(g)\,dg.
\end{align}
Here the Weil representation $\omega$ of $\bH(\BA_0)$ is though the factor $V'$ with its natural quadratic form. By \cite[Thm.~12.14]{Z19}, the function $(h,s)\in \bH(\BA_0)\times\BC\mapsto J(h,\Phi', s)$ is smooth in $(h,s)$, entire in $s\in\BC$, and left invariant under $\bH(F_0)$. There is an expansion according to the orbits
\begin{align}\label{Z g2l}
J(h,\Phi',s)=&\sum_{(\gamma,u')\in [(S_n(\alpha)\times V')(F_0)] \atop u'\neq 0} \Orb((\gamma,u'),\omega(h)\Phi',s)\\
&= \Orb((\gamma,0_\pm),\omega(h)\Phi',s)+\sum_{(\gamma,u')\in [(S_n(\alpha)\times V')(F_0)]_\rs} \Orb((\gamma,u'),\omega(h)\Phi',s).\notag
\end{align}
Here we refer to \cite[\S12.6]{Z19} for the regularization defining $\Orb((\gamma,0_\pm),\omega(h)\Phi',s)$, and the orbit of $(\gamma,0)$ has no contribution essentially because the character $\eta$ is non-trivial when restricted to its stabilizer (see the proof of \cite[Thm.~12.9]{Z19}). For regular semisimple $(\gamma,u')$ and  decomposable $\Phi'=\otimes_v\Phi'_v\in  \CS((S_n\times V')(\BA_0))$, we have
\begin{align}\label{orb gl}
\Orb((\gamma,u'),\Phi',s):=\prod_{v\in\Sigma_{F_0}}\Orb((\gamma,u'),\Phi'_{v},s),
\end{align}
where the local orbital integrals are defined by \eqref{Orb(gamma,f',s)}. Then, for $\xi\in F_0^\times$, the $\xi$-th Fourier coefficient of $J(\cdot,\Phi',s)$ (cf. \eqref{eq:def F coeff}) is equal to
\begin{align}\label{eqn: xi coeff gl}
\sum_{(\gamma,u')\in [(S_n(\alpha)\times V'_\xi)(F_0)]} \Orb((\gamma,u'),\omega(h)\Phi',s).
\end{align}

\subsection{The decomposition of the special value at $s=0$}
\label{ss:del J}
We set the special value at $s=0$,
\begin{equation*}
J(h,\Phi'):=J(h,\Phi',0) .
\end{equation*}
\[
\Orb((\gamma,u'),\omega(h)\Phi'):=\Orb((\gamma,u'),\omega(h)\Phi',0).
\]
Then the decomposition \eqref{Z g2l} specializes to
\begin{equation}\label{J s=0}
J(h,\Phi')=\sum_{(\gamma,u')\in [(S_n(\alpha)\times V')(F_0)]\atop u'\neq 0} \Orb((\gamma,u'),\omega(h)\Phi').
\end{equation}

We set the special value of the first derivative at $s=0$,
\begin{equation}
\begin{aligned}
   \delJ(h,\Phi')&:=\frac{d}{ds}\Big|_{s=0}   J(h,\Phi',s),\\
 \del((\gamma,u'),\Phi'_v) &:= \frac{d}{ds}\Big|_{s=0}  \Orb((\gamma,u'),\Phi'_v,s).
  	\end{aligned}
\end{equation}

\begin{lemma}\label{lem:delJ sm}Suppose that $\Phi'$ is $K_\bH$-invariant under the Weil representation. Then
we have $\delJ(-,\Phi')\in \CA_\infty(\bH(\BA_{0}),K_\bH, n)$.
\end{lemma}
\begin{proof}
By \cite[Thm.~12.14]{Z19}, the function $(h,s)\in \bH(\BA_0)\times \BC\mapsto J(h,\Phi',s)$ is smooth. It follows that $\delJ(-,\Phi')$ is smooth.
\end{proof}

Now we introduce
\begin{equation}\label{delJ}
\begin{aligned}
 \delJ_v(h,\Phi') &:=       \delJ_v(\omega(h)\Phi'),\quad \text{where}\\
       \delJ_v(\Phi') &:=\sum_{(\gamma,u')\in [(S_n(\alpha)\times V')(F_0)]\atop u'\neq0} \del((\gamma,u'),\Phi'_v)\cdot  \Orb((\gamma,u'),\Phi'^{v}).
  	\end{aligned}
\end{equation}

Then by Leibniz's rule, we obtain a decomposition, 
\begin{align}\label{eqn J' dec}
   \delJ(h,\Phi')= \del(0_\pm,\omega(h)\Phi')+\sum_{v}    \delJ_v(h,\Phi'),
   \end{align}
where the nilpotent term $\del(0_\pm,\omega(h)\Phi')$ is defined in \cite[\S12.7]{Z19} and we do not need the precise form in this paper.

\subsection{Preparation for the comparison}\label{ss:prep comp}
Let  $V$ be the $n$-dimensional $F/F_0$-hermitian space we used to define the Shimura variety $\Sh_{\wt K}\bigl(\wt G, \big\{\wt h\big\}\bigr)$ in \S\ref{ss:S data}. 

Let $\Phi=\otimes_{v<\infty}\Phi_v \in\CS((\U(V)\times V)(\BA_{0,f}))$ be a pure tensor. 
Let  $\Phi'= \otimes_{v}\Phi'_v \in \CS((S_{n} \times
V'_n)(\BA_0))$ be a pure tensor such that 
\begin{itemize}
\item  for every $v\mid\infty$,  $\Phi'_v$ is the fixed partial Gaussian test function, and 
\item for every non-archimedean $v$, $\Phi_v'$ partially (relative to $\alpha$) transfers to $\Phi_v$ (Def.  \ref{def part st}).
\end{itemize}

If $v$ is a place of $F_0$ that splits in $F$ (necessarily non-archimedean), then $\delJ_v(h,\Phi')=0$.
If $v$ is non-split (including the archimedean places), let $V^{(v)}$ be the $v$-nearby hermitian space of $V$. We recall that $V^{(v)}\simeq V$ if $v$ is the unique archimedean place $\varphi_0$  where $V$ has signature $(n-1,1)$ (cf. \S\ref{ss:S data}), and otherwise  $V^{(v)}$  is characterized by the following conditions 
\begin{itemize}
\item
 for all $w$ other than $v$ and $\varphi_0$, we have $V^{(v)}_{w}\simeq V_w$;
\item at the place $\varphi_0$, $V^{(v)}_{\varphi_0}$ is positive definite;
\item at the place $v$, $V^{(v)}_{v}$ has signature $(n-1,1)$ if $v$ is archimedean.
\end{itemize}
Then, the sub-sum of regular semisimple terms in $\delJ_v(h,\Phi')$  (cf. \eqref{delJ}) is a sum over orbits $(\gamma,u')\in [(S_n(\alpha)\times V')(F_{0})]_\rs $ matching orbits in $ [(\U(V^{(v)})\times V^{(v)})(F_0)]_\rs$.
Moreover, there is a Fourier expansion (cf. \eqref{eq:def F coeff})
\begin{align}
\label{dJ v qexp}
\delJ_v(h,\Phi')=\sum_{\xi\in F_0}\delJ_v(\xi,h, \Phi'),
\end{align}
where $\delJ_v(\xi,h, \Phi')$ is the sub-sum, 
\begin{align}
\label{dJ v xi}
\delJ_v(\xi,h,\Phi')=\sum_{(\gamma,u')\in [(S_n(\alpha)\times V'_\xi)(F_0)]\atop u'\neq 0} \del((\gamma,u'),\omega(h_v)\Phi'_v)\cdot  \Orb((\gamma,u'),\omega(h^v)\Phi'^{v}).
\end{align}
Then we set the value at $h=1$ normalized by the Whittaker function at the archimedean place
\begin{align}\label{eq:delJ v xi}
\delJ_v(\xi, \Phi'):= W^{(n)}_\xi(1)^{-1}\, \delJ_v(\xi,1,\Phi'),
\end{align}
where $ W^{(n)}_\xi(h_\infty):=\prod_{v\mid\infty} W^{(n)}_{v,\xi}( h_v)$ and each local factor $W^{(n)}_{v,\xi}$ is defined by \eqref{Whit}. 
 
Now let $v\nmid\fkd$ be a non-archimedean non-split place. By \eqref{dJ v qexp} and  \eqref{dJ v xi}, and the fact that $\Phi'_\infty$ is a (partial) Gaussian test function,  we have $\delJ_v(\xi,h, \Phi')=0$ unless $\xi\geq 0$ and
  \begin{align}\label{eq:delJ v qexp 1}
\delJ_v(h,\Phi')=\sum_{\xi\in F_0,\,\xi\geq 0}\delJ_v(\xi, \omega(h_f)\Phi')\,W^{(n)}_\xi(h_\infty).
\end{align}
Moreover, \eqref{eq:delJ v xi} becomes
\begin{align}\label{eq:delJ v xi 1}
\delJ_v(\xi, \Phi')=\sum_{(\gamma,u')\in [(S_n(\alpha)\times V'_\xi)(F_0)]\atop u'\neq 0}  \del((\gamma,u'),\Phi'_v)\cdot \Orb((\gamma,u'),\Phi'^{v,\infty}).
\end{align}
Finally we introduce
\begin{align}\label{eq:delJ xi 1}
\delJ (\xi, \Phi'):=\sum_{v} \delJ_v(\xi, \Phi'),
\end{align}
where the local terms vanish for all but finitely many places $v$ of $F_0$ (for a fixed pair $(\xi, \Phi')$).

\section{Intersection numbers and AFL}
In this section we recall the relevant RZ spaces, the local analog of CM cycles and local KR divisors, and the statement of the AFL conjecture. Then we globalize the intersection problem to the Shimura varieties we introduced in Part 1. We also state an application to the arithmetic intersection conjecture of Rapoport, Smithling and the second author in \cite{RSZ3}, assuming the AFL conjecture that will be proved in the next section. 
\subsection{Unitary RZ spaces}\label{ss:RZ}
Let $F/F_0$ be an unramified quadratic extension of $p$-adic local fields with $p$ odd and let $n\geq 1$. In this section, we recall the definition of the Rapoport--Zink formal moduli scheme $\CN_n = \CN_{n, F/F_0}$ associated to the unitary group for the quasi-split 
$n$-dimensional hermitian $F$-vector space, cf. \cite[\S4]{RSZ2} and \cite[\S3]{Z19}.

Let $\breve F$ denote the completion of a maximal unramified extension of $F$. For $\Spf O_{\breve F}$-schemes $S$ (i.e. a $O_{\breve F}$-scheme on which $p$ is locally nilpotent), we consider triples $(X, \iota, \lambda)$, where
\begin{enumerate}
\item[$\bullet$]
$X$ is a $p$-divisible group of absolute height $2nd$ and dimension $n$ over $S$, where  $d:=[{F_0}: \BQ_p]$, 
\item[$\bullet$]  $\iota$ is an action of $O_{F}$ such that the induced action of $O_{F_0}$ on $\Lie X$ is via the structure morphism $O_{F_0}\to \CO_S$,  and
\item[$\bullet$] $\lambda$ is a principal ($O_{F_0}$-relative) polarization. 
\end{enumerate}
Hence $(X, \iota|_{O_{F_0}})$ is a strict $O_{F_0}$-module of relative height $2n$ and dimension $n$. We require that the Rosati involution $\Ros_\lambda$ induces on $O_{F}$  the non-trivial Galois automorphism in $\Gal({F}/{F_0})$, denoted by $ O_F\ni a\mapsto \ov a$, and that the \emph{Kottwitz condition} of signature $(n-1,1)$ is satisfied, i.e.
\begin{equation}\label{kottwitzcond}
   \charac \bigl(\iota(a)\mid \Lie X;\, T\bigr) = (T-a)^{n-1}(T-\ov a) \in \CO_S[T]
	\quad\text{for all}\quad
	a\in O_{F} . 
\end{equation} 
An isomorphism $(X, \iota, \lambda) \isoarrow (X', \iota', \lambda')$ between two such triples is an $O_{F}$-linear isomorphism $\varphi\colon X\isoarrow X'$ such that $\varphi^*(\lambda')=\lambda$.

Over the residue field $\mbF$ of $O_{\breve {F}}$, there is a triple $(\BX_n, \iota_{\BX_ n}, \lambda_{\BX_n})$ such that $\BX_n$ is  supersingular, unique up to $O_F$-linear quasi-isogeny compatible with the polarization. We fix such a triple which we call the {\em framing object}. Then $\CN_n$ by definition represents the functor over $\Spf O_{\breve F}$ that associates to each $S$ the set of isomorphism classes of quadruples $(X, \iota, \lambda, \rho)$ over $S$, where the last entry is an $O_F$-linear quasi-isogeny of height zero defined over the special fiber $\ov S:=S\times_{ \Spf O_{\breve F}} \Spec \mbF$,
\[
   \rho \colon X\times_S\ov S \to \BX_n \times_{\Spec \mbF} \ov S,
\]
such that $\rho^*((\lambda_{\BX_n})_{\ov S}) = \lambda_{\ov S}$. The map $\rho$ is called the {\em framing}.
The formal scheme $\CN_n$ is  formally locally of finite type and formally smooth over $\Spf O_{\breve {F}}$ of relative dimension $n-1$.

The group of quasi-automorphisms of the framing object is
$$\Aut^\circ (\BX_n,\iota_{\BX_n},\lambda_{\BX_n}) = \{g\in \End^\circ_F(\mbX_n, ι_{\mbX_n}),\ g^\vee\circ λ_{\mbX_n} \circ g = λ_{\mbX_n}\}.$$
The condition $g^\vee \circ λ_{\mbX_n}\circ g = λ_{\mbX_n}$ may also be formulated as $gg^* = \mr{id}$, where $g \mapsto g^\ast=\Ros_{\lambda_{\BX_n}}(g)$ denotes the Rosati involution. Then $\Aut^\circ (\BX_n,\iota_{\BX_n},\lambda_{\BX_n})$ acts on $\CN_n$ by changing the framing: 
$$g\cdot (X,\iota,\lambda,\rho) = (X,\iota,\lambda, g \circ \rho).$$

Another description is as follows. Taking $n=1$, there is a unique triple $(\mbE, ι_{\mbE}, λ_{\mbE})$ over $\mbF$ with signature $(1,0)$. Set
\begin{align}\label{eq:BVn}
\BV_n:=\Hom^\circ_{O_F}(\BE, \BX_n),
\end{align}
which is an $n$-dimensional hermitian $F$-vector space with respect to the hermitian form
$$
\pair{x,y}= \lambda_{\mbE}^{-1}\circ y^\vee\circ \lambda_{\mbX_n} \circ x\in \End^\circ _{F}(\mbE)\simeq F.
$$
It is the unique (up to isomorphism) $n$-dimensional hermitian space that does {\em not} contain a self-dual $O_F$-lattice.
 Then there is a natural isomorphism
\begin{equation}\label{Aut cong U}
   \Aut^\circ (\BX_n,\iota_{\BX_n},\lambda_{\BX_n}) \cong \U\bigl(\BV_n\bigr)(F_0), 
\end{equation}
where $\Aut^\circ (\BX_n,\iota_{\BX_n},\lambda_{\BX_n})$ acts by composition on $\mbV_n$.

\subsection{Local intersection numbers}\label{ss:Int}

Now we introduce the local intersection numbers $\Int(g,u)$. In \cite{KR-U1}, Kudla and Rapoport have defined for every non-zero $u\in \BV_n$ a certain divisor $\CZ(u)$ on $\CN_n$, the so-called  (local) \emph{KR divisor}. For its definition, note that $\mcN_1 \iso \Spf O_{\breve F}$, so $(\mbE, ι_{\mbE}, λ_{\mbE})$ deforms to a unique triple $(\mcE, ι_{\mcE}, λ_{\mcE})$ over $O_{\breve F}$, called its \emph{canonical lift}. (This is the universal object over $\CN_1$ with Galois conjugated $O_F$-action.) Then $\mcZ(u)$ is defined as the locus where the quasi-homomorphism $u\colon \BE\to \BX_n$ lifts to a homomorphism from $\CE$ to the universal object over $\CN_n$. By \cite[Prop.\ 3.5]{KR-U1}, $\CZ(u)$ is a relative divisor (or empty). It is the local analog of the global divisor considered in \S\ref{ss:KR}. It follows from the definition that if $g\in \U(\BV_n)(F_0)$, then
\begin{align}\label{act KR}
g \CZ(u)=\CZ(gu).
\end{align}

For simplicity we will write $\CN_n\times\CN_n$ for the fiber product $\CN_n\times_{\Spf O_{\breve F}}\CN_n$.
For $g\in \U(\BV_n)(F_0)$, let $\Gamma_g\subseteq \CN_n\times\CN_n$ be the graph of the automorphism of $ \CN_n$ induced by $g$.
The fixed point locus of $g$ is defined as the intersection
\begin{align}\label{Ng}
\CN_n^g :=\Gamma_g\,\cap \Delta_{\CN_n},
\end{align}
viewed as a closed formal subscheme of $\CN_n$.
We also form the {\em  derived fixed point locus}, denoted by $\LN^g_n$, i.e. the derived tensor product
\begin{align}\label{der Ng}
\LN^g_n :=\Gamma_g\,\jiao \Delta_{\CN_n}:=\CO_{\Gamma_g}\Ltimes_{\CO_{\CN_n\times\CN_n}}\CO_{\Delta_{\CN_n}}
\end{align}
viewed as an element in $K_0^{\CN_n^g}(\CN_n)$, cf. \cite[Appendix B]{Z19}. The  fixed point locus  and its derived version are the local analogs of the global CM cycle and its derived version in \S\ref{ss:LCM}.

For a pair $(g,u)\in  (\U(\BV_n)\times  \BV_n)(F_0)_\rs$, we now set
\begin{equation}\label{def int g u}
   \Int(g,u) := \la \mcZ(u),\, \LN_n^g\ra _{\CN_n} := \chi\left({\CN_{ n}}, \, \mcO_{\mcZ(u)} \Ltimes_{\CO_{\CN_n}}\, \LN_n^g\right) . 
\end{equation}
Here, for a finite complex $\CF$ of coherent $\mcO_{\CN_n}$-modules, we define its Euler–Poincar\'e characteristic as 
$$
\chi(\CN_n,\CF)=\sum_{i,j} (-1)^{i+j} \mr{len}_{O_{\breve F}} H^{j}(\CN_n,H_i(\CF))
$$
if the lengths are all finite.

When $(g,u)$ is regular semi-simple, the intersection $\CZ(u)\cap\CN^g_n$ is a proper {\em scheme} over $\Spf O_{\breve{F}}$ and hence the right-hand side of \eqref{def int g u} is finite.
The number $\Int(g,u)$ depends only on the $\U(\BV_n)(F_0)$-orbit of $(g,u)$. There is an equivalent definition that does not involve the derived fixed point locus  $\LN^g_n$ (cf. \cite[Rem.~ 3.1]{Z19}),
\begin{equation}\label{eq:Int alt}
\Int(g,u)= \chi\left({\CN_{ n}\times\CN_n}, \, \CO_{\Gamma_g}\Ltimes_{\CO_{\CN_n\times\CN_n}} \CO_{\Delta_{\CZ(u)}}\right) . 
\end{equation}

\subsection{The AFL conjecture}\label{ss:AFL}
The Arithmetic Fundamental Lemma conjecture states that $\Int(g,u)$ equals a derivative of the orbital defined in \eqref{Orb(gamma,f',s)}.
\begin{conjecture}[AFL, semi-Lie algebra version]\label{AFLconj rs}
\label{AFL lie}
Suppose that $(\gamma,u')\in (S_{n}\times V_n')({F_0})_\rs$ matches the element $(g,u)\in ( \U(\BV_{n})\times  \BV_{n})(F_0)_\rs$. Then 
\[
\del\bigl((\gamma,u'), \mathbf{1}_{(S_{n}\times V_n')(O_{F_0})}\bigr) 
	   = -\Int(g,u)\cdot\log q. 
\]
\end{conjecture}
The case of a fixed $u$ with $(u,u)\in O_{F_0}^\times$ has been studied first historically, cf. \cite{Z12}, and is known as the \emph{group version} of the AFL, cf. also \cite[Conj. 3.2]{Z19}. In that case there is a natural isomorphism 
$$
\CZ(u)\simeq \CN_{n-1}.
$$
By \cite[Prop. 4.12]{Z19},
the intersection number \eqref{eq:Int alt} is equal to the one considered in \cite{Z12},
$$
\chi\left({\CN_{ n-1}\times\CN_n}, \, \CO_{(1,g)\cdot\Delta_{\CN_{n-1}}  }\Ltimes_{\CO_{\CN_{n-1}\times\CN_n}} \CO_{\Delta_{\CN_{n-1}}}\right),
$$
the orbital integral reduces to the one in {\it loc. cit.} as well, and hence the group version for $\mbV_{n+1}$ is equivalent to Conj. \ref{AFLconj rs}, at least if $q_{F_0}\geq n+1$. 

\subsection{Global intersection numbers}
We now recall the relation of local and global intersection numbers which relies on the non-archimedean uniformization for $\mcM$. We resume the notation from \S\ref{s:SV KR}.
Let $ν\nmid \mfd$ be a non-archimedean place of $\bfF$ and set $v = ν\vert_{F_0}$. Assume $v$ to be inert in $F$. There is a uniformization along the basic locus (cf. \cite[Proof of Thm.~ 8.15]{RSZ3})
\begin{equation}\label{eq unif1}
    \CM_{  O_{\breve \bfF_\nu}}\sphat = \wt G^{(v)} (\BQ)\Big\bs \Bigl[ \CN'\times \wt G (\BA_f^p) /K_{\wt G}^p \Bigr].
\end{equation}
Here the left-hand side is the formal completion of the base change $ \mcM\otimes\mcO_{\breve\bfF_ν}$ along the basic locus of the geometric special fiber $\mcM\otimes\ov{\mbF}_ν$. The basic locus here is the closed subspace consisting of $(A_0,A,\ov{η})$ where $A[v^\infty]$ is a supersingular $p$-divisible group. On the right hand side, $\wt G^{(v)}$ is the analog of $\wt G$, but for the nearby hermitian space $V^{(v)}$. The formal scheme $\mcN'$ is an RZ space of PEL type in \cite{RZ} whose precise definition we omit. We now make the additional assumption that the CM type $Φ$ is unramified at $p$ (cf. Def. \ref{def:unramified_CM_type}). Then the comparison isomorphism \cite[Thm 3.1]{M-Th} applies to $\mcN'$ and \eqref{eq unif1} may be rewritten as
\begin{equation}\label{eq unif2}
  \CM_{  \mcO_{\breve \bfF_\nu}}\sphat \,
	   = \wt G^{(v)}(\BQ) \Big\bs \Bigl[  \CN_{n,F_v/F_{0,v}} \times \wt G (\BA_f^{v}) /\wt K^v \Bigr].
\end{equation}
Here $\mcN_{n,F_v/F_{0,v}}$ is the RZ space from \S\ref{ss:RZ} for the quadratic extension $F_v/F_{0,v}$ and, by abuse of notation,
$$
\wt G (\BA_f^{v}) /\wt K^v = 
    \wt G (\BA_f^{p})/\wt K^p\times \bigl(Z^\BQ(\BQ_p)/K_{Z^\BQ, p}\bigr)\times \prod_{w\mid p,\, w\neq v}G(F_{0,w})/ K_w,
$$
where the product is over places $w$ of $F_0$ above $p$. (In \cite{RSZ3} it is assumed that $v$ is unramified over $\mbQ$ which explains why the assumption on the CM type $Φ$ does not figure explicitly.)

Now we come to the global intersection numbers.
Let $α\in F[t]^\circ_{\deg n}$ be a fixed irreducible polynomial, $\varphi = \varphi_\mfd \otimes 1_{K^\mfd}\in \mcS(K\backslash G(\mbA_{0,f})/K)$ and $ϕ = ϕ_\mfd \otimes 1_{\widehat{Λ}^\mfd}\in \mcS(V(\BA_{0,f}))$. Set $Φ := \varphi\otimes \phi$.
For each $ξ\in F_0$, equation \eqref{eq:AI} defines an intersection number of $\wh{\mcZ}^\bB(ξ,ϕ)$ and $\LCM(α,φ)$ in $\mbR_\mfd$. (Recall that $\LCM(α,φ)$ is defined in Def. \ref{def:cm_cycle}.) However, due to the passage to $\wh\Ch {}^1$ in \eqref{eq:AI}, the information of the individual place-by-place contributions is lost. We assume $ξ \neq 0$ from now on. Then we can work with the representing divisor $\wh{\mcZ}^\bB(ξ,ϕ)$ directly to define a more specific intersection number
\begin{equation}\label{eq int def}
\Int(ξ,Φ) := \frac{1}{τ(Z^\mbQ) [\bF:F]} \left( \widehat{\mcZ}^\bB(ξ,ϕ), \LCM(α,\varphi)\right) \in \mbR
\end{equation}
that lifts the $\mbR_{\mfd}$-valued intersection number from \eqref{eq:AI}.
It is defined as a sum of an archimedean and a non-archimedean part. The \emph{archimedean contribution} is defined as usual by evaluating the Green function of $\wh{\mcZ}^\bB(ξ,ϕ)$ on $\LCM(α,φ)_\bfF$. This is well-defined since the relations in the definition of $\mcZ_1(\mcM)$ are supported in special fibers only (cf. Def. \ref{def Z1}), so that the horizontal part of $\LCM(α,φ)$ is unambiguous. The \emph{non-archimedean contribution} is the arithmetic Euler--Poincaré characteristic
\begin{equation}\label{eq ad hoc int}
χ\big(\mcM,\ \mcO_{\mcZ(ξ,ϕ)} \Ltimes_{\mcO_{\mcM}}\LCM(α,φ)\big).
\end{equation}
The cycles do not intersect in the generic fiber by \cite[Thm. 9.2 (i)]{Z19}, so \eqref{eq ad hoc int} is a finite number. With this definition of $\Int(ξ,Φ)$, we even obtain a well-defined place-by-place decomposition
$$\Int(ξ,Φ) = \sum_{v \in \Sigma_{F_0},\ v\nmid \mfd} \Int_v(ξ,Φ).$$
Note that the local factor vanishes for all but finitely many $v$ (for a fixed pair $(ξ,Φ)$).
The $v$-term here is the sum of the contributions of all $ν\in \Sigma_{\bfF}$ lying above $v$. Part (b) of the following theorem relies on the non-archimedean uniformization. Also see Thm.~\ref{thm CM inter} (1) and (2) for the analogous result for elementary CM cycles.
\begin{theorem}[\protect{\cite[Thm. 9.4]{Z19}}] \label{thm loc glob decomp}
Let $ν$ be a non-archimedean place of $\bfF$ and $v = ν\vert_{F_0}$.
\begin{altenumerate}
\renewcommand{\theenumi}{\alph{enumi}}
\item
If $v$ is split in $F$ and $\xi\neq 0$, then $(\mcC(α,\varphi) \cap \mcZ(ξ,ϕ))_ν = \emptyset$. In particular $\Int_v(ξ,ϕ) = 0$.
\item
Assume $v$ to be inert in $F$ and $ξ\neq 0$. Then
\begin{equation}\label{sum inert}
\Int_{v}(\xi,\Phi) =  
2\log q_{v} \sum_{(g,u)\in [(G^{(v)}(\alpha)\times V^{(v)}_\xi)(F_0)]} \Int(g,u) \cdot \Orb\left((g,u), \Phi^{v}\right).
\end{equation}
Here $G^{(v)} = \U(V^{(v)})$ and $G^{(v)}(α)$ is the subvariety of elements with characteristic polynomial $α$. Moreover, $\Int(g,u)$ is the local intersection number from \eqref{def int g u}.
\end{altenumerate}
\end{theorem}
There is an analogous result for archimedean $ν$ which is proved via complex uniformization  (cf. Thm.~\ref{thm CM inter} (3)). Since this is identically the same as \cite[Cor. 10.3]{Z19}, we will not repeat here.

\subsection{Application to diagonal intersection}

In \cite{RSZ3}, Rapoport, Smithling and the second author study an intersection problem of arithmetic diagonal cycles that is motivated by the Arithmetic Gan--Gross--Prasad Conjecture, cf. \cite{Z12, GGP}. They formulate a conjectural formula for the local contributions to the resulting intersection numbers, cf. \cite[Conj.~8.13~(i)]{RSZ3}. These local terms decompose just as $\Int_v(ξ,Φ)$ does in \eqref{sum inert} above. In the hyperspecial case of \cite{RSZ3}, this reduces their conjecture to the AFL. Anticipating our proof of the AFL for $q_v\geq n$, cf. Thm. \ref{thm main} below, we formulate here the implied intersection identity in the diagonal setting.

Let $u\in V$ be a non-isotropic vector with orthogonal complement $W$. Define the following groups (over $\Spec \mbQ$ and $\Spec F_0$, respectively).
\begin{align*}
H &:= U(W),\ \ \text{an algebraic group over $F_0$,}\\
H^\BQ &:= \bigl\{g \in \Res_{F_0/\BQ} \GU(W) \bigm| c(g)\in \BG_m\bigr\},\ \ \text{$c$ the similitude factor},\\
\wt H &:= Z^\BQ \times_{\BG_m} H^\BQ \iso Z^{\mbQ}\times \Res_{F_0/\mbQ}H,\ \ \mr{and}\\
\wt {HG} &:= Z^\BQ \times_{\BG_m} H^\BQ \times_{\BG_m} G^\mbQ \iso Z^{\mbQ}\times \Res_{F_0/\mbQ}H \times \Res_{F_0/\mbQ} G.
\end{align*}
Let $\{h_H\}$ and $\{h_{\wt H}\}$ be the Shimura data for $H$ and $\wt H$ defined in \S\ref{ss:S data} with $W$ instead of $V$. The product $\{h_{Z^\mbQ} \times h_H \times h_G\}$ defines a Shimura datum for $\wt {HG}$. The reflex field for $\wt H$ and $\wt{HG}$ is again $\bfF$. Let $K_H\subseteq H(\mbA_{0,f})$ be a level subgroup and let $K_{\wt H} := K_{Z^\mbQ}\times K_H$ as well as $K_{\wt{HG}} := K_{Z^\mbQ}\times K_H\times K_G$. Then, for the respective Shimura varieties over $\Spec \bfF$,
$$\Sh_{K_{\wt{HG}}}\big(\wt{HG},\{h_{\wt{HG}}\}\big) =
\Sh_{K_{\wt H}}\big(\wt H, \{h_{\wt H}\}\big)
\times_{\Sh_{K_{Z^\mbQ}}(Z^\mbQ, h_{Z^\mbQ})}
\Sh_{K_{\wt G}}\big(\wt G, \{h_{\wt G}\}\big).$$
Here we wrote $K_{\wt G} := \wt K$ and $\{h_{\wt G}\} := \big\{\wt h\}$ to make the notation more coherent. There is an embedding $H\hookrightarrow G$ that identifies $H$ with the stabilizer of $u$ in $G$. Recall that $K_G^\mfd = \mr{Stab}(\wh{Λ}^\mfd)$ is the stabilizer of a self-dual $O_F[\mfd^{-1}]$-lattice in $V$. We now impose
\begin{itemize}
\item $u\in \wh{Λ}^\mfd$ with $(u,u) \in O_F[\mfd^{-1}]^\times$ and
\item $K_H = H\cap K_G$.
\end{itemize}
In particular, the data of $W$ and $K_H$ satisfy the good reduction assumptions away from $\mfd$ in \S\ref{ss:integral models} and we obtain a smooth integral model $\mcM_{\wt H}$ over $O_\bfF[\mfd^{-1}]$ for the Shimura variety of $\wt H$. It parametrizes tuples $(A_0,ι_0,λ_0,\ov{η}_0, A, ι, λ, \ov{η})$ in the sense of Def. \ref{def RSZ glob}, but for $W$ instead of $V$ and $n-1$ instead of $n$. The closed immersion of Shimura varieties for $\wt H$ and $\wt G$ extends to a closed immersion of integral models with moduli description
\begin{equation}
\label{eq embed H G}
\begin{aligned}
\mcM_{\wt H} & \to \mcM_{\wt G}\\
(A_0,A,\ov{η}) & \mapsto (A_0, A_0\times A, (\id_{A_0} \mapsto u) \times \ov{η}).
\end{aligned}
\end{equation}
The product
$$\mcM_{\wt{HG}} := \mcM_{\wt H}\times_{\mcM_0} \mcM_{\wt G}$$
is a smooth integral model for $\Sh_{K_{\wt {HG}}}\big(\wt{HG}, \{h_{\wt{HG}}\}\big)$. The definition of the Hecke correspondence \eqref{eq Hecke mu} can be made for $\mcM_{\wt H}$ in exactly the same way; it moreover extends by linearity from the case of a coset $µ$ to the Hecke algebra at $\mfd$. Taking the product, we obtain for every function of the form
$$φ_\mfd = 1_{K_{Z^\mbQ, \mfd}} \otimes φ_{H,\mfd} \otimes φ_{G,\mfd} \in \mcS\big(K_{\wt {HG},\mfd}\bs \wt{HG}(\mbQ_{\mfd})/K_{\wt {HG},\mfd}\big)$$
a Hecke operator
$$R(φ_\mfd)\colon \mr{Hk}_{φ_\mfd} \to \mcM_{\wt {HG}} \times_{\mcM_0} \mcM_{\wt{HG}}.$$
(Only the component $φ_{H,\mfd}\otimes φ_{G,\mfd}$ is used to define the correspondence; the trivial $1_{K_{Z^\mbQ}}$ acts trivially.) When we write $R(φ)$ with $φ\in \mcS\big(K_{\wt {HG}}\bs \wt{HG}(\mbA_f)/K_{\wt {HG}}\big)$ in the following, it is understood that $φ = φ_\mfd\otimes φ^\mfd$ with $φ^\mfd = 1_{K_{\wt {HG}}}^\mfd$ standard and $R(φ) = R(φ_\mfd)$. This completes the setting for the definition of the semi-global intersection number from \cite[(8.18)]{RSZ3} in the hyper-special case. Let $v\mid \mfd$ be a place of $F_0$ that is inert in $F$. We define, for $ν\in \Sigma_\bfF$ with $ν\mid v$,
\begin{equation}\label{eq int semi-global}
\begin{aligned}
\Int^\natural_ν(φ) :=& \pair{R(φ)^*\mcM_{\wt H},\ \mcM_{\wt H}}_{\mcM_{\wt {HG}},\, v},\\
\Int_v(φ) :=& \frac{1}{τ(Z^\mbQ)[\bfF:F]} \sum_{ν\mid v} \Int^\natural_ν(φ).
\end{aligned}
\end{equation}
Here, the map $\mcM_{\wt H}\to \mcM_{\wt {HG}}$ is the graph of \eqref{eq embed H G}.
The following is \cite[Conj. 8.13]{RSZ3} in the present notation for $v$. We refer to \cite[\S7]{RSZ3} for the notions of transfer and of Gaussian test function in the present setting.
\begin{theorem}\label{thm semi-global}
Assume that $q_v\geq n$ and that the CM type $Φ$ is unramified at $p$ as in Def. \ref{def:unramified_CM_type}. Let $φ=φ_\mfd\otimes φ^{\mfd}$ be as above with $φ^\mfd$ standard. Assume $φ_\mfd$ to be completely decomposed, $φ_\mfd = \otimes_{w\mid \mfd} φ_w$, and let $φ'=\otimes_{w}φ'_w\in \mcS(G'(\BA_{F_0}))$ be a Gaussian test function such that  $\otimes_{w < \infty}φ'_w$ is a smooth transfer of $φ$.
Assume that for some prime $\ell$ prime to $v$ and some place $\lambda$ above $\ell$, both $φ_λ$ and $φ'_λ$ have regular support. Then
$$
\Int_{v}(φ)=-\delJ_{v}(φ'). 
$$
\end{theorem}
Here $G'=\Res_{F/F_0}(\GL_{n-1} \times \GL_n)$ and $\delJ_{v}$ is the $v$-th component of the global distribution $\delJ$ on $G'(\BA_{F_0})$ defined in \cite[\S7.2]{RSZ3}. (Note that this distribution $\delJ_{v}$  is related to but not the same as the $\delJ_v$-distribution defined in \S\ref{ss: n>1} and \ref{ss:del J}.)
\begin{proof}
The statement was already shown for $n\leq 3$ and $v$ unramified over $p$ in \cite[Thm. 8.15]{RSZ3}. As long as the CM type $Φ$ is unramified, the proof applies for arbitrary $n$ and reduces the statement to the AFL identity. This we prove in the next section whenever $q_v\geq n$.
\end{proof}

\section{Proof of the AFL}\label{s:proof}
This last section is devoted to the proof the AFL in \S\ref{ss:AFL}, at least for not too small residue cardinality:
\begin{theorem}\label{thm main}
Conjecture \ref{AFLconj rs} holds under the assumption that the residue cardinality of $F_0$ is $\geq n$.
\end{theorem}

\subsection{Globalization}\label{ss:glob}
Starting from the data of the AFL identity in question (an unramified quadratic extension of $p$-adic local fields, the integer $n$ and the regular semi-simple pair $(g,u)$), we choose a quadratic extension $F/F_0$ of number fields with $F$ CM and $F_0$ totally real such that:
\begin{itemize}
\item $F_0$ has a place $v$ above $p$ which is inert in $F$ and such that $F_v/F_{0,v}$ is the extension of local fields $F_v/F_{0,v}$ one we would like to prove an AFL identity for.
\item All places $w\neq v$ of $F_0$ above $p$ are split in $F$.
\end{itemize}
Next, choose an $n$-dimensional $F/F_0$-hermitian vector space $V$ such that:
\begin{itemize}
\item The signature of $V$ is $(n-1,1)$ at a unique place $φ_0\colon F_0\to \mbR$ and $(n,0)$ otherwise.
\item The localization $V_p$ contains a self-dual $O_{F,p}$-lattice $Λ_p$.
\end{itemize}
Let $V^{(v)}$ be the $v$-nearby positive definite hermitian space of $V$ and choose an isomorphism $V^{(v)}_v \iso \mbV_n$ to view $(g,u) \in (G^{(v)}\times V^{(v)})(F_{0,v})$. Here we recall that $\BV_n$ is defined by \eqref{eq:BVn}. Note that such an isomorphism is only unique up to $G^{(v)}(F_{0,v})$ so it is really the orbit $G^{(v)}(F_{0,v})\cdot (g,u)$ that is canonical. The main result of \cite{M-LC} is that the intersection number $\Int(g,u)$ is locally constant (for the $p$-adic topology) in $(g,u)$. On the analytic side, it is a priori clear that $\del (γ,u')$ is locally constant in $(γ,u')$. Moreover, the matching relation preserves $p$-adic closeness since it is defined through the invariants of $(g,u)$ and $(γ,u')$, cf. \S\ref{ss: transfer}. It follows that we may replace $(g,u)$ by a suitable $p$-adically close enough \emph{global} pair with the following additional properties.
\begin{itemize}
\item The characteristic polynomial $α$ of $g$ (now an element of $F[t]^\circ_{\deg n}$) is irreducible over $F$.
\item The self-pairing $ξ := \pair{u,u}$ (now an element of $F_0$) is non-zero.
\item For all places $w\neq v$ of $F_0$ above $p$, the ``standard'' orbital integral for the orbit of $(g,u)$ does not vanish
\begin{equation}\label{eq aux additional}
\Orb((g,u),1_{\mr{Stab}(Λ_w)}\otimes 1_{Λ_w}) \neq 0,
\end{equation}
where $Λ_w\subseteq Λ_p$ is the factor corresponding to $w$.
\end{itemize}
Next, we choose a CM type $Φ$ for $F$ which we assume to be unramified at $p$ in the sense of Def. \ref{def:unramified_CM_type}. For example, one may first choose an unramified $p$-adic CM type $Φ\subseteq \Hom(F,\mbC_p)$ and then transport it via any choice of isomorphism $\mbC \iso \mbC_p$. Enlarging $\mfd$ while keeping $p\nmid \mfd$, we assume that $Φ$ is unramified at all $\ell \nmid \mfd$. The choice of $Φ$ also defines an extension of $φ_0$ to $F$. Finally, we complete the data so far to a Shimura datum as in \S\ref{s:SV KR}, subject only to the condition that $p\nmid \mfd$.
\begin{proposition}[\protect{\cite[Proof of Lem. 13.7 and Prop. 13.8]{Z19}}]\label{prop:aux_orbit_separation}
Assume that for all non-archimedean places $w$ of $F_0$, $w\nmid v\mfd$,
\begin{equation}
\label{eq aux cond}
\Orb((g,u),1_{K_w}\otimes 1_{Λ_w}) \neq 0.
\end{equation}
Fix a compact subset $\Omega_v$ of $(G^{(v)} \times V^{(v)})(F_{0,v})$ containing $(g,u)$. Then, after possibly shrinking $K_\mfd$, there exist choices for $ϕ^v = ϕ_\mfd \otimes 1_{\widehat{Λ}^{v\mfd}}\in \mcS(V(\BA^v_{0,f}))$ and $\varphi^v = \varphi_\mfd \otimes 1_{K^{v\mfd}}\in \mcS(K^v\backslash G(\mbA^v_{0,f})/K^v)$ such that the following are equivalent for $(\wt g, \wt u)\in (G^{(v)}(α)\times V^{(v)}_ξ)(F_0)$: 
\begin{enumerate}[wide, labelindent=0pt, labelwidth=!, label=(\arabic*), topsep=2pt, itemsep=2pt]
\item The $G^{(v)}(F_{0,v})$-orbit of $(\wt g, \wt u)$ in $(G^{(v)} \times V^{(v)})(F_{0,v})$ intersects nontrivially with $\Omega_v$ and the away-from-$v$ orbital integral
$\Orb((\wt g,\wt u),Φ^v)$ is non-vanishing with $Φ^v = ϕ^v\otimes \varphi^v$.
\item There is an equality of orbits $G^{(v)}(F_0) \cdot (\wt g, \wt u) = G^{(v)}(F_0)\cdot (g,u)$.
\end{enumerate} 
\end{proposition}
Enlarging $\mfd$, we may assume that condition \eqref{eq aux cond} is satisfied. Since we already imposed \eqref{eq aux additional}, we may do so while maintaining $p\nmid \mfd$. For example, $\mfd$ now also contains all places $w\neq v$ where $α$ or $ξ$ are not integral. We may find and fix a compact $\Omega_v$ as in the proposition such that if an orbit in $(G^{(v)}(α)\times V^{(v)}_ξ)(F_0)$ contributes non-trivially to \eqref{sum inert}, or matches to an orbit in $(S_n(α)\times V'_ξ)(F_0)$ that contributes non-trivially to \eqref{eq:delJ v xi 1}, then it intersects $\Omega_v$. (Loosely speaking, existence of $\Omega_v$ follows from the fact that the local terms $\Int(g,u)$ resp. $\del((\gamma,u'),\Phi'_v)$ in the two formulas can only be non-zero when $(g,u)$ resp. $(γ,u')$ satisfy an integrality condition. We refer to \cite[§13.4 and Cor.~14.8]{Z19} for the precise argument.) Let $K_\mfd$ and $Φ_\mfd$ be then as in the proposition. Let $Φ'\in \mcS((S_n\times V')(\mbA_0))$ be a Gaussian transfer of $Φ$ as in \S\ref{ss:prep comp}. It follows that, for $(γ, u')\in (S_n(α)\times V'_ξ)(F_0)$ such that locally at $v$ the orbit of $(γ, u')$ matches one in $\Omega_v$, the away-from-$v$ orbital integral 
$\Orb((γ, u'),Φ'^{,v})$
is non-vanishing if and only if $\GL_n(F_0)\cdot (γ, u')$ is the matching orbit of $(g,u)$.
\begin{corollary}\label{cor loc at v}
The AFL identity for $(g,u)$ and $(γ,u')$ holds if and only if
\begin{equation}\label{eq loc at v}
2\partial J_{v}(ξ,Φ') = -\Int_{v}(ξ,Φ).
\end{equation}
\end{corollary}
\begin{proof}
This follows by comparing the two local-global decomposition formulas \eqref{eq:delJ v xi 1} and \eqref{sum inert}, each of them containing (at most)  one non-zero term.
\end{proof}
This idea may be developed further. Namely, for $w$ non-archimedean,
$$2\partial J_w(ξ,Φ'),\ \Int_w(ξ,Φ) \in \mbQ\log q_w,$$
so both terms in Cor.~ \ref{cor loc at v} lie in $\mbQ \log (p)$. Note that $\Int_w(ξ,ϕ) = 0$ for all $w\neq v$ above $p$ since these were assumed to be split. 
Moreover, there is a known identity of archimedean contributions: by \cite[Lem.~ 14.4]{Z19}, for an archimedean place $w$,
$$2\partial J_{w}(ξ,Φ') = -\Int_{w}^\bK(ξ,Φ).$$
It follows that
$$2\partial J_w(ξ,Φ') + \Int_w^{\bK-\bB}(ξ,Φ) = -\Int_w(ξ,Φ),\ \ \ \ w\mid \infty.$$
Using the linear independence (over $\mbQ$) of the $\log \ell$ for $\ell$ prime, we obtain that \eqref{eq loc at v}, and hence the AFL identity in question, follow from the identity
\begin{equation}
\label{eq to show} 
2\delJ(\xi,\Phi') + \Int^{\bK-\bB}(ξ,Φ) = -\Int(ξ,Φ)\ \ \ \text{in } \mbR_\mfd
\end{equation}
where $\delJ(\xi,\Phi')$ is defined in \eqref{eq:delJ xi 1}.

\subsection{Modular forms}\label{ss:Int mod}
The way to proceed in \cite{Z19} was to use that all three terms in \eqref{eq to show} are known to be the Fourier coefficients of modular forms if $F_0=\mbQ$. For the $\partial J$-term this was worked out in \cite[\S9]{Z19}, for the $\Int^{\bK-\bB}$-term this is due to Ehlen--Sankaran \cite{ES}, for the $\Int$-term it is the main result of \cite{BHKRY}. Only the statement for $\partial J$ is currently known to carry over without change. For the other two terms we will instead apply the almost modularity results from \S\ref{s:alm mod}. This requires us however to modify $\LCM(α,\varphi)$, since Thms. \ref{thm ES} and \ref{thm mod} only apply to pairings with 1-cycles of degree $0$.

By Prop.~ \ref{prop modif cycle exists}, there exist a family of $*$-embeddings $σ_i\colon E\to \End(V)$ with $σ_i(O_E)$ stabilizing $Λ_p$, Hecke operators $g_i\in G(F_{0,\mfd})$ and scalars $λ_i\in \mbQ$ such that
$$C(α,\varphi)^\circ := C(α,\varphi) - \sum_{i} λ_i C(σ_i,g_i)$$
has degree $0$ on every connected component of $M$. Let $\bfE$ be an extension of $\bfF$ over which the $C(σ_i,g_i)$ are defined and set 
$$\mcC := \sum_i λ_i \mcC(σ_i,g_i)$$
 where the integral models $\mcC(σ_i,g_i)$ of the $C(σ_i,g_i)$ are the normal ones (cf. the paragraph after Def. \ref{def element CM cycle}). Define the modified CM cycle
$$
\LCM(\alpha,\varphi)^\circ= \LCM(\alpha,\varphi)-\mcC\in \mcZ_1(O_{\bfE}\otimes_{O_{\bfF}} \mcM).
$$
At this point, we again enlarge $\mfd$ (but still keep $K$ unchanged and maintain the condition $p\nmid \mfd$) if necessary so that we can assume $Λ[\mfd^{-1}]$ to be $σ_i(O_E)$-stable for all $i$.

Apply the pairing \eqref{eq:adm int} and define 
\begin{equation}\label{int g0}
\Int(h,\Phi)^\circ := \frac{1}{τ(Z^\mbQ) [\bE:F]} \left( Z(h, \phi),\quad  \LCM(\alpha,\varphi)^\circ\right)^\adm,
\end{equation}
For the rest of the section we let   $K_{\bH,\fkd}$ be a compact open subgroup of $\bH(F_{0,\fkd})$ that fixes $\phi_\fkd', \phi_\fkd$ and $\Phi'_\fkd$  under the Weil representation. Recall $ K_{\bH}^{\fkd,\circ}=\prod_{v\nmid\fkd}K_{\bH,v}^{\circ}$ from \eqref{eq:max open}.  Then Lem.~\ref{lem:Weil KH} (2) shows that $\phi_v$ is invariant under $ K_{\bH,v}^{\circ}$ for every $v\nmid\fkd$.
\begin{corollary}[to Thm. \ref{thm mod}]\label{cor hmf}The function
$\Int(-,Φ)^\circ$ lies in $\CA_{\rm hol}(\bH(\BA_0),K_{\bH}, n)_{\ov\BQ}\otimes_{\ov\BQ}\BR_{\mfd,\ov\BQ}$  with $K_\bH=K_{\bH,\fkd}\times K_{\bH}^{\fkd,\circ}$ and $\BR_{\mfd,\ov\BQ}:=\BR_{S,\ov\BQ}$ for $S$ the set of places $v\mid \mfd$.
\end{corollary}

We can spell out the Fourier coefficients $\Int(\xi, h_\infty,\Phi)^\circ$ using the arithmetic KR divisors $\wh\CZ^\bB(\xi,\phi)$ (cf. \eqref{eq:KR  Ar}) on the integral model $\CM$, and using the arithmetic intersection pairing \eqref{eq:adm int}
\begin{equation}\label{int g0 xi}
\Int(\xi,\Phi)^\circ := \frac{1}{τ(Z^\mbQ) [\bE:F]} \left(\wh \CZ^\bB(\xi, \phi) ,\quad\LCM(\alpha,\varphi)^\circ \right)
\end{equation}
for $ξ \in F_{0,+}$. 

Now we define the corresponding automorphic form on the analytic side. Let 
$$\delJ(h,\phi') = \frac{1}{ [E:F]}\sum_i λ_i \delJ (h, \phi'_i)$$
 be the linear combination of (the restriction to $\bH(\BA_0)$ of)  $\delJ$ from \eqref{eq def an series}. In particular by Thm.~\ref{thm comparison CM cycle} (and Prop.~\ref{prop CM intersection decomp}), the $\xi$-th Fourier coefficient with $\xi \neq 0$ is given by
\begin{equation}\label{eq:delJ phi'}
2\delJ(ξ,h_\infty,\phi') = -\frac{1}{τ(Z^\mbQ) [\bE:F]}\left( \widehat{\mcZ}^\bK (ξ,h_\infty,ϕ),\ \mcC\right)W^{(n)}_\xi(h_\infty)
\end{equation}
up to $\sum_{\ell\mid \mfd}\mbQ\log \ell$. Define the modified analytic generating function:
\begin{equation}\label{eq def del J circ}
 \delJ (h,\Phi')^\circ=       \delJ (h,\Phi')-\delJ(h,\phi'),\quad h\in \bH({\BA_0}).
\end{equation}
Then we consider
\begin{equation}\label{eq def del J}
\delJ_{\rm hol}(h):=      2 \delJ (h,\Phi')^\circ+\Int^{\bK-\bB}(h,\Phi)^\circ,
\end{equation}
where the last summand denotes the evaluation of the difference of the two Green functions at the modified CM cycle, cf. \eqref{geo error 0},
$$\Int^{\bK-\bB} (h,Φ)^\circ :=\frac{1}{τ(Z^\mbQ) [\bE:F]} \sum_{ν\in \Sigma_\bfE,\, ν\mid \infty} \CG_ν^{\bK-\bB}(h,ϕ)(C(α,\varphi)^\circ).$$

\begin{proposition}\label{prop hmf}
The function $\delJ_\mr{hol}(h)$ lies in $\CA_{\rm hol}(\bH({\BA_0}),K_\bH,n)_{\ov\BQ}\otimes_{\ov\BQ}\BR_{\mfd,\ov\BQ}$ with $K_\bH=K_{\bH,\fkd}\times K_{\bH}^{\fkd,\circ}$.
\end{proposition}
\begin{proof}
First we note that each summand in \eqref{eq def del J} belongs to $\CA_\infty(\bH(\mbA_0),K_\bH,n)$. For $\delJ(h,\phi')$ the statement is Prop.~\ref{prop:Weil del J em}. For $\delJ(h,Φ')$ this is Lem.~\ref{lem:delJ sm}, together with the fact that the same proof of Lem.~\ref{lem:Weil KH} shows that $\Phi_v'$ is invariant under $ K_{\bH,v}^{\circ}$ for every $v\nmid\fkd$. For $\Int^{\bK-\bB}(h,\Phi)$ it is a consequence of Thm.~\ref{thm:dif inf} and the fact that $\Phi_v$ is invariant under $ K_{\bH,v}^{\circ}$ for every $v\nmid\fkd$.

To complete the proof we need to show that the Fourier expansion of $ \delJ_{{\rm hol}}$ takes the form 
\begin{align}\label{eq:hol qexp}
 \delJ_{{\rm hol}}(h)=\sum_{\xi\in F_{0},\xi\geq 0} A_\xi(h_f) \,W_\xi^{(n)}(h_\infty),\quad A_\xi\in \BR_\mfd,
\end{align}
where $A_\xi(h_f) =0$ unless $\xi$ lies in a fractional ideal of $F_0$ (depending on $h_f\in \bH(\BA_{0,f})$). Similar to the proof of \cite[Prop.~14.5]{Z19}, using the fact that the Weil representation commutes with smooth transfer,  it suffices to consider the case $h_f=1$.

The proof of \cite[Prop.~14.5]{Z19} applies for a general totally real field $F_0$ and shows that the sum
$$2 \delJ (h_\infty,\Phi') + \frac{1}{τ(Z^\mbQ) [\bE:F]} \sum_{ν\in \Sigma_\bfE, ν\mid \infty} \CG_ν^{\bK-\bB}(h_\infty,ϕ)( C(α,\varphi)) $$
has the desired form of \eqref{eq:hol qexp}. It is only left to show that 
$$2\delJ(h_\infty,\phi') + \frac{1}{τ(Z^\mbQ) [\bE:F]} \sum_{ν\in \Sigma_\bfE, ν\mid \infty}  \CG_ν^{\bK-\bB}(h_\infty,ϕ)(C)$$
takes the form of \eqref{eq:hol qexp}. By definition, the automorphic Green functions $\CG_ν^{\bB}(\xi,h_\infty,ϕ)$ are actually independent of $h_\infty$ (cf.  \eqref{Gr B h}), and hence $\CG_ν^{\bB}(h_\infty,ϕ)$ takes the form of \eqref{eq:hol qexp}. So by \eqref{eq:G K-B} and \eqref{geo error 0}, it remains to show that
$$2\delJ(h_\infty,\phi') + \frac{1}{τ(Z^\mbQ) [\bE:F]}\sum_{ν\in \Sigma_\bfE, ν\mid \infty}  \CG_ν^{\bK}(h_\infty,ϕ)(C)$$
takes the form of \eqref{eq:hol qexp}. But here both terms arise by diagonal restriction from generating functions over $E_0$ to which we may apply Thm.~\ref{thm comparison CM cycle} (together with Prop.~\ref{prop CM intersection decomp}). The proof is now complete. 
\end{proof} 

\subsection{Comparison}
For $\xi\in F_{0,+}$, recall that our aim is to show the identity $2\delJ(ξ,Φ') + \Int^{\bK-\bB}(ξ,Φ) = - \Int(ξ,Φ)$ in $\mbR_\mfd$, cf. \eqref{eq to show}. Subtracting the known identity \eqref{eq:delJ phi'} evaluated at $h_\infty=1$,
\begin{equation}\label{eq CM add}
2\delJ(ξ,\phi') = -\frac{1}{τ(Z^\mbQ) [\bE:F]} \left( \wh \mcZ^\bK(ξ,ϕ), \mcC\right)\quad \text{in } \mbR_\mfd,
\end{equation}
cf. \eqref{Gr Ku h=1} and \eqref{eq:Z Ku h=1}, we see that the AFL identity in question would follow from
$$\delJ_\mr{hol}(ξ) = -\Int(ξ,Φ)^\circ,
$$
where the left hand side is the $\xi$-th Fourier coefficient of \eqref{eq def del J}.
Slightly enlarging $\mfd$, we now assume it is divided by all places with residue cardinality $<n$.

\begin{theorem}\label{thm AFL0}
There is an equality in $\CA_{\rm hol}(\bH(\BA_0),K_{\bH}, n)_{\ov\BQ}\otimes_{\ov\BQ}\BR_{\mfd,\ov\BQ}$:
\begin{equation}\label{eq ident hmfs}
\delJ_{\rm hol} (h) = -\Int(h,Φ)^\circ.
\end{equation}
In particular the AFL holds whenever $q_v \geq n$.
\end{theorem}
\emph{Proof.} 
Let $B$ be the product of all primes $\ell\nmid \mfd$ where $\mbZ_\ell\otimes_{\mbZ} O_F[t]/(α(t))$ is not the maximal order in the étale $\mbQ_\ell$-algebra $E_\ell=\BQ_\ell\otimes_\BQ E$. (Note that the case of interest is typically $p\mid B$.) By \cite[Lem. 13.6]{Z19} applied to $\delJ_{\rm hol}(h) +\Int(h,Φ)^\circ$ and $B$, it suffices to show the equality of Fourier cofficients (cf. \eqref{eq:def F coeff} for notation)
$$
\delJ_{\rm hol} (\xi, h_\infty) = -\Int(\xi, h_\infty,Φ)^\circ
$$
for all $\xi\in F_{0}^\times$ such that $w(\xi)=0$ for all $w| B$ (here we write $w(\xi)$ for the valuation  of $\xi$ associated to the place $w$).
Since both are holomorphic of the same weight (cf. Prop.~\ref{prop hmf} and Cor.~\ref{cor hmf}), this is equivalent to
$$
\delJ_{\rm hol} (\xi) = -\Int(\xi,Φ)^\circ
$$
for all $\xi\in F_{0}^\times$ such that $w(\xi)=0$ for all $w|B$.

By an inductive argument, we from now on assume the AFL Conjecture to hold in all cases of residue cardinality (of $F_0$) $\geq n$ and dimension (of the hermitian space $V$) smaller than $n$. It is now sufficient to verify the following claim.

\emph{Claim. For $ξ$ with $w(ξ) = 0$ for all $w\mid B$, we have $\delJ_\mr{hol}(\xi) = - \Int(ξ,Φ)^\circ$ in $\mbR_\mfd$.}
The CM identity \eqref{eq CM add} may be added again which leaves us with
$$2\delJ(ξ,Φ') + \Int^{\bK-\bB}(ξ,Φ) \overset{?}{=} \Int(ξ,Φ)$$
which may be rewritten as
$$2\delJ(ξ, Φ') \overset{?}{=}- \Int^{\bK}(ξ,Φ).$$
From here on, the argument is again as in \cite[Thm.~14.6]{Z19}. Write
$$\begin{aligned}
2\delJ(ξ,Φ') + \Int^\bK(ξ,Φ) = & \sum_{w\nmid\mfd,\, w\nmid \infty} \big(2\delJ_w(ξ,Φ') + \Int_w(ξ,Φ)\big)\\
									& + \sum_{w\nmid\mfd,\, w\mid \infty} \big(2\delJ_w(ξ,Φ') +\Int_w^\bK(ξ,Φ)\big).
\end{aligned}$$
The archimedean terms are known to vanish by \cite[Lem.~14.4]{Z19}. The non-archimedean terms may be rewritten using the local-global decompositions \eqref{eq:delJ v xi 1} and \eqref{sum inert}:
\begin{multline*}
2\delJ_w(ξ,Φ') + \Int_w(ξ,Φ) = \\
\sum_{(g,u)\in [(G^{(w)}(α)\times V^{(w)}_ξ)(F_0)],\atop (γ,u')\leftrightarrow (g,u)}
2\big(\del ((γ,u'), Φ'_w) + \Int_w(g,u)\big) \Orb((g,u),Φ^w) \log q_w.
\end{multline*}
The assumption on $ξ$ and $B$ implies that at least one of the following two cases holds.
\begin{enumerate}[wide, labelindent=0pt, labelwidth=!, label=(\arabic*), topsep=2pt, itemsep=2pt]
\item If $w\nmid ξ$, then $(u,u)$ is a $w$-adic unit. The AFL identity $\del ((γ,u'),Φ'_w) =- \Int_w(g,u)$ is known to reduce to a similar AFL-identity in dimension $n-1$ which holds by our induction hypothesis. The reduction statement is \cite[Prop. 4.12]{Z19}.
\item If $w\mid ξ$, then $w\nmid B$ so $O_F[t]/(α(t))$ is a maximal order at $w$. In this case the above AFL identity is also known, cf. \cite[Prop. 3.10]{Z19}.
\end{enumerate}
This finishes the proof of the Claim, of the theorem and of the AFL for $q_v \geq n$. \qed


\begin{thebibliography}{99}

\bibitem{BGS}{J.-B. Bost, H. Gillet, and C. Soul\'e, \textit{ Heights of projective varieties and positive Green forms,} J. Amer. Math. Soc. {\bf 7} (1994), no. 4, 903--1027.}


\bibitem{Br12}{J. Bruinier, \textit{Regularized theta lifts for orthogonal groups over totally real fields}, J. Reine Angew. Math. {\bf 672} (2012), 177--222.}


\bibitem{BHKRY}{J. Bruinier, B. Howard, S. Kudla, M. Rapoport, and T. Yang,  \textit{Modularity of generating series of divisors on unitary Shimura varieties}, Ast\'erisque {\bf 421} (2020).}


\bibitem{BK03}{J. Bruinier and U. K\"{u}hn, \textit{Integrals of automorphic Green’s functions associated to Heegner divisors},
Int. Math. Res. Not. 2003 (2003), no. 31, 1687--1729.}



\bibitem{ES} {S. Ehlen and  S. Sankaran,
\textit{On two arithmetic theta lifts,} Compos. Math. {\bf 154} (2018), no. 10, 2090--2149. }





\bibitem{GGP}{W. T. Gan, B. Gross, and D. Prasad, \textit{Symplectic local root numbers, central critical
$L$-values, and restriction problems in the representation theory of
classical groups},  Ast\'erisque \textbf{346} (2012), 1--109.}

\bibitem{GaS}{L. E. Garcia and S. Sankaran. \textit{Green forms and the arithmetic Siegel--Weil formula.} Invent. Math., \textbf{215} (2019), no. 3, 863--975.}


\bibitem{GS}{H. Gillet and C. Soul\'e, \textit{Arithmetic intersection theory},  Inst. Hautes Etudes Sci. Publ. Math. \textbf{72} (1990), 93--174. } 

\bibitem{GZ}{B. Gross and D. Zagier, \textit{Heegner points and derivatives of $L$-series}, Invent. Math. \textbf{84} (1986), no. 2, 225--320.}

\bibitem{H-CM}{B. Howard, \textit{Complex multiplication cycles and Kudla-Rapoport divisors}, Ann. Math. \textbf{176} (2012), no. 2, 1097--1171.}

\bibitem{JR}{H. Jacquet and S. Rallis, \textit{On the Gross-Prasad conjecture for unitary groups}, in On Certain L-Functions, Clay Math. Proc. 13, Amer. Math. Soc., Providence, RI, 2011, pp. 205--264.}


\bibitem{K-duke}{S. Kudla, \textit{Algebraic cycles on Shimura varieties of orthogonal type}, Duke Math. J. {\bf 86} (1997), 39--78.}


\bibitem{K}{\bysame, \textit{Central derivatives of Eisenstein series and height pairings}, Ann. of Math. (2) \textbf{146} (1997), no. 3, 545--646.}


\bibitem{K03}{\bysame, \textit{Integrals of Borcherds forms}, Compos. Math. 137 (2003), 293--349.}

\bibitem{Kud-Mil}{S. Kudla and J. Millson, \textit{Intersection numbers of cycles on locally symmetric spaces and Fourier coefficients of holomorphic modular forms in several complex variables}, Inst. Hautes études Sci. Publ. Math. \textbf{71}
(1990), 121--172.}


\bibitem{KR-U1}{S. Kudla and M. Rapoport, \textit{Special cycles on unitary Shimura varieties I. Unramified local theory}, Invent. Math. \textbf{184} (2011), no. 3, 629--682.}

\bibitem{KR-U2}{\bysame, \textit{Special cycles on unitary Shimura varieties II: Global theory},  J. Reine Angew. Math. \textbf{697} (2014), 91--157.}


\bibitem{KRY}{S. Kudla, M. Rapoport, and T. Yang,  \textit{Modular forms and special cycles on Shimura curves}, Annals of Mathematics Studies, vol. 161, Princeton University Press, Princeton, NJ, 2006. }

\bibitem{Liu1}{Y. Liu,  \textit{Arithmetic theta lifting and L-derivatives for unitary groups, I.} Algebra Number Theory 5 (2011), no. 7, 849--921.}

\bibitem{Liu18}{\bysame,  \textit{Fourier-Jacobi cycles and arithmetic relative trace formula},  (with an appendix by Chao Li and Yihang Zhu),  Cambridge J. Math., \textbf{9} (2021), no. 1, 1--147.}

\bibitem{M-AFL}{A. Mihatsch, \textit{On the arithmetic fundamental lemma conjecture through Lie algebras}, Math. Z. \textbf{287} (2017), no. 1--2, 181--197.}

\bibitem{M-Th}{\bysame, \textit{Relative unitary RZ-spaces and the arithmetic fundamental lemma}, J. Inst. Math. Jussieu \textbf{21} (2022), no. 1, 241--301.}

\bibitem{M-LC}{\bysame, \textit{Local Constancy of Intersection Numbers}, Algebra Number Theory \textbf{16} (2022), no. 2, 505--519.}

 
\bibitem{Mil-Shimura}{J. Milne, \textit{Introduction to Shimura Varieties}, \href{https://www.jmilne.org/math/xnotes/svi.pdf}{www.jmilne.org/math/xnotes/svi.pdf}.}

\bibitem{Mum-AV}{D. Mumford, \textit{Abelian varieties}, With appendices by C. P. Ramanujam and Yuri Manin. Corrected reprint of the second (1974) edition. Tata Institute of Fundamental Research Studies in Mathematics, 5. Published for the Tata Institute of Fundamental Research, Bombay; by Hindustan Book Agency, New Delhi, 2008. xii+263 pp. }


\bibitem{OT}{T. Oda and M. Tsuzuki, \textit{Automorphic Green functions associated with the secondary spherical functions}, Publ. Res. Inst. Math. Sci. 39 (2003), 451--533.}



\bibitem{RSZ2}{M. Rapoport, B. Smithling and W. Zhang, \textit{Regular formal moduli spaces and arithmetic transfer conjectures}, Math. Ann. \textbf{370} (2018), no.\ 3--4, 1079--1175.}

\bibitem{RSZ3}{\bysame,  \textit{Arithmetic diagonal cycles on unitary Shimura varieties}, Compos. Math. \textbf{156} (2020), no. 9, 1745--1824.}

\bibitem{RSZ4}{\bysame, \textit{On Shimura varieties for unitary groups}, Pure Appl. Math. Q. (2021) \textbf{17}, no. 2, 773--837.}

\bibitem{RZ}{M. Rapoport and Th. Zink, \textit{Period spaces for $p$-divisible groups}, Annals of Mathematics Studies, vol. \textbf{141}, Princeton University Press, Princeton, NJ, 1996.}


\bibitem{Stacks}{The Stack project Authors,  \textit{Stack project}.}

\bibitem{YZZ1}{X. Yuan, S.-W. Zhang, and W. Zhang, \textit{The Gross--Kohnen--Zagier Theorem over Totally Real Fields}, Compos. Math. {\bf 145} (2009), 1147--1162.}


\bibitem{Z12}{W. Zhang, \textit{On arithmetic fundamental lemmas}, Invent. Math. \textbf{188} (2012), no. 1, 197--252.}

\bibitem{Z14}{\bysame, \textit{Fourier transform and the global Gan--Gross--Prasad conjecture for unitary groups},  Ann. of Math. (2) \textbf{180} (2014), no. 3, 971--1049.}

\bibitem{Z19}{\bysame, \textit{Weil representation and Arithmetic Fundamental Lemma}, Ann. of Math. (2) \textbf{193} (2021), no. 3, 863--978.}

\bibitem{ZZ}{Z. Zhang, \textit{Maximal parahoric arithmetic transfers, resolutions and modularity}, Preprint (2021), \href{https://arxiv.org/abs/2112.11994}{\texttt{arXiv:2112.11994}}.}

\end{thebibliography}
\end{document}